\documentclass[11pt,a4paper]{article}
\usepackage{xcolor}
\usepackage[utf8]{inputenc}
\usepackage[T1]{fontenc}
\usepackage{amsmath,amssymb,amsthm}
\usepackage{mathtools}
\usepackage[a4paper,margin=2.1cm]{geometry}
\usepackage{hyperref}
\usepackage{enumitem}
\usepackage{booktabs}
\usepackage{graphicx}

\newtheorem{theorem}{Theorem}
\newtheorem{proposition}{Proposition}
\newtheorem{lemma}{Lemma}
\newtheorem{corollary}{Corollary}
\theoremstyle{definition}
\newtheorem{remark}{Remark}

\newtheorem{assumption}{Assumption}

\DeclareMathOperator{\vol}{vol}
\DeclareMathOperator{\supp}{supp}
\DeclareMathOperator{\dive}{div}
\DeclareMathOperator{\arccosh}{arccosh}
\DeclareMathOperator{\capac}{cap}
\newcommand{\R}{\mathbb{R}}
\newcommand{\N}{\mathbb{N}}
\newcommand{\C}{\mathbb{C}}
\newcommand{\bg}{\overline{g}}
\newcommand{\ws}{|w|}
\newcommand{\wox}{w \odot x}
\newcommand{\lmin}{\lambda_{\min}}

\title{
	Volume of quasi-homogeneous sublevel sets: Two linear algebra deterministic algorithms with convergence rates}
\author{Didier Henrion\footnote{LAAS-CNRS, University of Toulouse, France and Faculty of Electrical Engineering, Czech Technical University in Prague, Czechia. Email: henrion@laas.fr} \quad and \quad Jean Bernard Lasserre\footnote{LAAS-CNRS, University of Toulouse, and Toulouse School of Economics, France. Email: lasserre@laas.fr}}
\date{\today}

\begin{document}
	\maketitle
	
	\begin{abstract}
		We consider the problem of computing the Lebesgue volume of the unit
		sublevel set of a positive quasi-homogeneous polynomial. Pushing the
		Lebesgue measure of an ambient bounding box forward through the polynomial
		reduces this high-dimensional volume to a one-dimensional moment problem.
		This removes the ambient dimension from the optimization and confines the dimension to a single preprocessing stage, computing
		the moments of the polynomial over the box, which is polynomial in the ambient dimension for sparse
		or separable polynomials.
		
		We propose two deterministic algorithms for the resulting univariate relaxations, each
		returning certified upper and lower bounds on the volume. Both bypass
		semidefinite optimization entirely and rely only on standard numerical
		linear algebra. The first approximates a piecewise-constant function by a
		Chebyshev polynomial, so that each relaxation reduces to a fast cosine
		transform, and converges at a polynomial rate in the relaxation order. The
		second extracts the volume bounds from a single generalized eigenvalue
		problem involving moment and localizing matrices whose size grows linearly
		with the relaxation order, and converges at an exponential rate; the ratio
		governing this rate is determined by an a priori upper bound on the
		polynomial over the bounding box.
		
		Finally, the univariate polynomials produced by either algorithm are
		feasible for the multivariate moment-SOS volume hierarchy. The algebraic and
		geometric rates therefore transfer to the hierarchy itself, improving on its
		best known convergence rates.
	\end{abstract}
	
	\section{Introduction}\label{sec:setup}

	The \emph{volume problem} we address is elementary to state and notoriously
	hard to solve: given a compact set
	\begin{equation}\label{eq:K}
		K \;:=\; \{ x \in \R^n : g(x) \leq 1 \}
	\end{equation}
	described by the unit sublevel set of a polynomial $g$,  compute, or approximate, with guarantees, its
	Lebesgue measure
	\[
	\vol K \;=\; \int_K dx .
	\]
	Beyond its intrinsic interest in computational geometry \cite{hls2009}, this quantity is the
	basic primitive behind computing the probability of a event and
	behind estimating regions of attraction and reachable sets of polynomial
	control systems~\cite{HenrionKorda2014}, and it carries with it the integration
	of polynomials against the uniform measure on $K$. 
	{In addition and importantly, when $K$ has the form \eqref{eq:K} with $g$ homogeneous and nonnegative, then $\vol K$ is also related to:
		
		-- \emph{Partition functions} $Z(g)=\int\exp(-g(x))dx$ for exponential families in statistics as $\vol K=Z(g)/\Gamma(1+n/\mathrm{deg}(g))$;
		see e.g. \cite{Lasserre2017,lasserre2019,WainwrightJordan2008}\,,
		
		-- \emph{Integral discriminants} $\int\exp(g(x))dx$ that appear in computational physics, as described in \cite{Morozov2009} where the authors' goal is to provide an exact analytical formula in terms of hypergeometric functions of invariants of $g$ via non linear numerical algebra. While elegant, this approach (so far successful in a few cases only), reveals that several extraordinary long and complex formula of resultants may be needed; see e.g. \cite{Morozov2009}[p. 30]}

	The difficulty is twofold.
	On the modelling side, $K$ is typically non-convex, possibly disconnected, and
	given only implicitly through its defining polynomial. On the computational
	side, the problem is subject to a severe curse of dimensionality, which any
	practical method must confront.
	
	\paragraph{Intrinsic hardness}
	Even for the simplest sets the problem is intractable. Computing the exact
	volume of a polytope presented by its facets, or by its vertices, is as hard as evaluating a permanent~\cite{DyerFrieze1988}. In the
	black-box model, where $K$ is convex and accessed through a membership oracle,
	no deterministic algorithm making polynomially many oracle calls can approximate
	$\vol K$ within a factor better than exponential in the dimension \cite{Elekes1986,BaranyFuredi1987}. These obstructions leave
	three escape routes: restrict the geometry of $K$, allow randomization, or
	settle for certified approximation. The literature organizes naturally
	along them.
	
	\paragraph{{Randomized methods for convex bodies and some star-shaped sets}}
	When $K$ is convex, randomization changes the picture entirely. The reference \cite{DyerFriezeKannan1991} gave the first fully
	polynomial randomized approximation scheme {(FPRAS)}: for any $\varepsilon,\delta
	> 0$ it returns a value within relative error $\varepsilon$ with probability at
	least $1-\delta$, using a rapidly mixing random walk (ball walk, hit-and-run) to
	sample near-uniformly from $K$ together with a multiphase Monte Carlo
	telescoping of nested bodies. A sustained effort then reduced the oracle
	complexity \cite{KannanLovaszSimonovits1997,LovaszVempala2006,CousinsVempala2018}; the same machinery integrates log-concave densities. The strength of these
	methods is their low-degree polynomial scaling in $n$. Their
	limitations are equally structural: they require convexity (or log-concavity),
	the guarantee is probabilistic rather than deterministic, and for non-convex $K$
	the underlying chain may mix exponentially slowly. Plain rejection Monte Carlo
	over a bounding box $B$ dispenses with convexity and converges at the
	dimension-free rate $O(N^{-1/2})$, but its variance scales like the inverse of
	the volume ratio $\vol K / \vol B$, which is typically exponentially small in
	$n$, precisely the regime one wants to handle. {It turns out that similar FPRAS also work for a class of star-shaped sets
		as described in the apparently overlooked contribution \cite{CDV2010}. A crucial parameter of the computational 
		complexity of the associated FPRAS is
		the fraction of volume occupied by the kernel, as well as the ability to sample in that kernel. Recently this approach has been extended beyond star-shaped compact sets in \cite{Vempala2026}.}

	\paragraph{Symbolic-numeric methods}
	A second route trades statistical estimates for certified digits. Reference \cite{slm2019} realizes the volume
	of a compact semialgebraic set as a {period} of a rational integral and
	evaluate it to absolute precision by numerically integrating the
	associated Picard--Fuchs (holonomic) differential equation, improving on the
	exponential precision bounds of Monte Carlo and moment methods. Very recently, the reference \cite{Ramesh2026} specializes this approach to convex bodies cut out by
	concave polynomials, exploiting convexity to reduce the number of
	creative-telescoping steps by an exponential factor. Such methods are
	deterministic and deliver arbitrary accuracy, but their cost grows steeply with
	both the number of variables and the degree, and they rely on heavy
	computer-algebra machinery; in practice they remain confined to a handful of
	variables.
	
	\paragraph{Moment-SOS hierarchies}
	Closest to the present work is the deterministic moment-sums-of-squares (SOS)
	methodology of \cite{hls2009}. Writing $\vol K$ as the value of
	an infinite-dimensional linear program over measures dominated by the Lebesgue
	measure on a bounding box, and truncating to moments of bounded degree, one
	obtains a hierarchy of semidefinite relaxations whose optimal values form a
	monotone sequence of upper bounds converging to $\vol K$; the dual computes a
	polynomial majorizing the indicator $\mathbf 1_K$, and the same data yield all
	moments of the uniform measure on $K$, hence the integral of any polynomial. The
	scheme is deterministic, comes with an a~priori convergence guarantee, and
	handles non-convex and disconnected sets. It has, however, two well-documented
	weaknesses. First, the indicator function is discontinuous, so its polynomial
	majorants exhibit a Gibbs phenomenon  \cite[Chap. 9]{trefethen2013}. As a result, the bounds, although
	monotone, tighten only slowly, and persistent oscillations near $\partial
	K$ make convergence slow. Second, the
	order-$r$ relaxation manipulates pseudo-moments up to degree $2r$ and
	positive-semidefinite matrices of size $\binom{n+r}{n}$, so its dimension
	explodes with $n$. 
	
	\paragraph{Pushforward and dimension reduction}
	The curse of dimensionality afflicting the moment-SOS hierarchy lives in the
	\emph{ambient} variable $x \in \R^n$. A way to escape it, when $K$ has the form \eqref{eq:K}, is to push the uniform measure forward through $g$ and work
	in the image space, whose dimension is governed by the degree of $g$ rather than
	by $n$. For a positive  {homogeneous} form $g$, it was shown \cite{lasserre2019}  that computing the volume of $K$ reduces to
	a one-dimensional moment problem for the pushforward of the Lebesgue measure
	under $g$, and is approximated as closely as desired by solving a sequence of
	generalized eigenvalue problems for a pair of Hankel matrices whose entries are
	available in closed form. This rests on structural properties of homogeneous sublevel
	volumes: the identity relating $\vol\{g\leq1\}$ to the non-Gaussian integral
	$\int \exp(-g)$~\cite{Lasserre2017}, the convexity (complete monotonicity) \cite{KozhasovLasserre2023} of the volume as a function of the
	coefficients of $g$, and the associated extremal
	problems~\cite{KozhasovLasserre2022}. Because its cost is dictated by $\deg g$
	and not by $n$, the pushforward strategy is naturally suited to the regime of
	moderate-to-large $n$ and small degree.
	
	\paragraph{Scope of this paper}
	Homogeneity is, however, a stringent requirement: it forces every monomial of
	$g$ to share the same total degree, so that $K$ is the unit ball of a norm-like
	form with isotropic scaling. Many sublevel sets of interest are anisotropic,
	with the variables entering at different natural scales. We therefore work in
	the broader \emph{quasi-homogeneous} setting, in which $g$ is homogeneous with
	respect to a vector of positive integer weights.
	This class is wide enough to model anisotropic sublevel sets while retaining the
	exact scaling symmetry that drives the pushforward construction, with the sum of the weights
	playing the role of the ambient dimension.
	
	{On the side of randomized algorithms, remarkable progress has been made since
		the (overlooked) work of \cite{CDV2010} where FPRAS were proved for sampling in a class of (possibly non convex) star-shaped sets, and very recently for an even larger class  of compact sets \cite{Vempala2026}; this opens the door
		to FPRAS for volume computation via such sampling methods.
		On the other hand, for deterministic methods, 
		(with the exception of moment-SOS or symbolic-numeric methods mentioned before), non-convex sets are considered to be significantly harder to address than convex sets.
		
		The scope of the present paper is to also describe 
		a significant progress for deterministic volume algorithms when considering sublevel sets of \emph{quasi-homogeneous}
		polynomials. So, as for randomized algorithms, some non-convex sets are amenable to practical efficient volume approximation (of course modulo some size limitation).}

	\subsection{Quasi-homogeneous polynomials}\label{sec:setup_qh}
	
	Let $w = (w_1, \ldots, w_n) \in \N_{>0}^n$ be a vector of positive integer
	weights, and let $g : \R^n \to \R_+$ be a polynomial that is
	\emph{quasi-homogeneous of weighted degree $m \geq 1$ with respect to $w$}, that
	is,
	\begin{equation}\label{eq:qh}
		g(t^{w_1} x_1, \ldots, t^{w_n} x_n) \;=\; t^m\, g(x),
		\qquad t > 0,\ x \in \R^n.
	\end{equation}
	Equivalently, every monomial $x^\alpha$ appearing in $g$ satisfies the linear
	weighted-degree constraint
	$\langle w, \alpha\rangle := \sum_{i=1}^n w_i \alpha_i = m$. The \emph{weight sum}
	\begin{equation}\label{eq:wsum}
		\ws \;:=\; \sum_{i=1}^n w_i
	\end{equation}
	plays the structural role of the ambient
	dimension. The classical homogeneous case $g(\alpha x) = \alpha^m g(x)$ is
	recovered for $w = (1, \ldots, 1)$, for which $\ws = n$.
	
	We make one standing assumption on the values of $g$.
	
	\begin{assumption}[Positivity]\label{ass:pos}
		The polynomial $g$ is strictly positive on $\R^n \setminus \{0\}$.
	\end{assumption}
	
	Because $g$ is quasi-homogeneous of weighted degree $m \geq 1$,
	Assumption~\ref{ass:pos} forces $g$ to be coercive along the weighted scaling:
	for $x \neq 0$, $g(t^{w_1}x_1, \ldots, t^{w_n}x_n) = t^m g(x) \to \infty$ as
	$t \to \infty$. Consequently every sublevel set
	$\{x \in \R^n : g(x) \leq c\}$ is compact, for every $c \geq 0$.
	
	\subsection{The volume problem}\label{sec:setup_problem}
	
	Let $B := [-1,1]^n$ be the centred unit box, and let
	$K$ be as in \ref{eq:K} 
	the unit sublevel set of $g$, which is a compact basic semialgebraic set by
	Assumption~\ref{ass:pos}. {We assume (possibly after scaling) that $K\subset B$.}
	
	\begin{assumption}[Bounding box]\label{ass:box}
		The box $B = [-1,1]^n$ contains the sublevel set $K$, that is $K \subseteq B$;
		equivalently, $g \geq 1$ on $\R^n \setminus \operatorname{int} B$.
	\end{assumption}
	
	Under Assumption~\ref{ass:box}, $K$ coincides with the box-constrained set
	$\{x \in B : g(x) \leq 1\}$, and its boundary is the level set
	$\partial K = \{g = 1\}$. The box then serves as the ambient integration domain,
	and the target of the paper is the Lebesgue volume
	\[
	\vol K \;=\; \int_B \mathbf{1}_K(x)\, dx .
	\]
	The setting of interest is $n$ moderate-to-large (say up to $n = 20$ or more)
	with $m$ small (typically $m \in \{2, 4, 6\}$). The single auxiliary global
	quantity needed by the two algorithms developed below is the maximal value
	\begin{equation}\label{eq:bg}
		\bg \;:=\; \max_{x \in B} g(x)
	\end{equation}
	whose computation will be addressed later on. In fact an upper
	bound on $\bg$ is enough, but it should be as small as possible, for reasons to be explained further. Throughout the paper we assume $\bg > 1$; the degenerate case $\bg = 1$
	forces $K = B$ and $\vol K = 2^n$.

	\paragraph{Running example}
	Take $n = 2$, $w = (1, 2)$, $m = 4$ and $g(x_1, x_2) = x_1^4 + x_2^2$ on
	$B = [-1,1]^2$. Quasi-homogeneity holds, since
	$g(t x_1, t^2 x_2) = t^4 x_1^4 + t^4 x_2^2 = t^4 g(x)$, and $g$ is strictly
	positive away from the origin, so Assumption~\ref{ass:pos} holds; here
	$\ws = 3$ and $\bg = 2$. The sublevel set $K = \{x_1^4 + x_2^2 \leq 1\}$ is
	contained in $B$, touching the boundary at $(\pm 1, 0)$ and $(0, \pm 1)$, so
	Assumption~\ref{ass:box} holds as well. The exact volume is
	\[
	\vol K \;=\; 4\int_0^1 \sqrt{1 - x_1^4}\,dx_1
	\;=\; \frac{\Gamma(\tfrac14)\,\Gamma(\tfrac32)}{\Gamma(\tfrac74)}
	\;\approx\; 3.4961.
	\]
	
	\subsection{Background: moment-SOS volume hierarchies and their limits}\label{sec:bg}
	The volume of a compact basic semialgebraic set can be approximated by the
	moment-SOS   hierarchy as described in \cite{hls2009}. The volume is the value of an infinite-dimensional linear
	program over measures
	whose unique optimum is the pair of Lebesgue measures split across $K$ and its
	complement within $B$. Truncating to moments of degree $\leq 2r$ turns this into a
	semidefinite program: the unknowns are the truncated moment sequences,
	constrained with positive semidefinite moment and
	localizing matrices. The dual semidefinite program is a
	sum-of-squares (SOS) problem constructing a polynomial
	majorant of degree $\leq 2r$ of the discontinuous indicator $\mathbf 1_K$.
	
	As already mentioned, this hierarchy has two well-known limitations: the Gibbs phenomenon due to the approximation of a discontinuous function by polynomials, and the size of the moment matrices, exponential in the ambient dimension $n$.
	Both difficulties are addressed, for quasi-homogeneous $g$, by the
	Stokes-accelerated scalar pushforward of~\cite{lasserre2019}.
	The scalar pushforward moment-SOS reduction was originally proposed in
	\cite{JasourHofmannWilliams2018} for risk estimation in robotics.
	It was then improved in \cite{lasserre2019} in the
	homogeneous polynomial setting by replacing the resulting univariate
	hierarchy with generalized Hankel eigenvalue problems, and by using
	homogeneity moment identities derived from the Stokes divergence theorem. In a more general context, the Stokes-augmented dual formulations of
	\cite[Sec.~3]{stokesgibbs} replace the indicator function by a continuous
	PDE-defined function, accelerating the convergence of the moment-SOS hierarchy.
	
	Convergence rates for the moment-SOS volume hierarchy were first established in \cite{kordahenrion}. For a general compact basic
	semialgebraic set, the upper bounds produced by the hierarchy of \cite{hls2009}
	converge to the volume at the rate $O(1/\log\log r)$ in the
	relaxation order $r$; when the set is the sublevel set of a {single}
	polynomial, this improves to $O(1/\log r)$. These slow
	rates are the quantitative signature of the Gibbs phenomenon incurred by the
	dual, which majorizes the discontinuous indicator $\mathbf 1_K$ by polynomials.
	These bounds were substantially sharpened in \cite{schlosserlazarev}, where,
	combining a recent effective Putinar Positivstellensatz with quantitative
	polynomial approximation, the logarithmic rates are replaced by \emph{algebraic}
	ones: $O(r^{-1/(6nL)})$ for the standard formulation of \cite{hls2009}, and
	$O(r^{-1/(2.5\,nL)})$ (up to an arbitrarily small loss in the exponent) for
	the Stokes-augmented formulation of \cite{stokesgibbs}, whose dual admits a
	smooth optimal solution and hence exhibits no Gibbs phenomenon; here $n$ is the
	ambient dimension and $L \ge 1$ is a {\L}ojasiewicz exponent of the
	constraints. The Stokes constraints thus improve the rate exponent by a factor
	exceeding two, the gain originating from the regularity of the optimal dual
	solution rather than from the optimization itself. In particular, both algebraic
	rates degrade with the ambient dimension $n$.
	
	\subsection{Contribution}\label{sec:contrib}
	
	This paper builds on the scalar pushforward approach to volume
	approximation of $\{x : g(x) \le 1\}$, for a 
	quasi-homogeneous polynomial $g$. Our contribution is twofold.
	
	\begin{enumerate}[leftmargin=*,itemsep=0.7em]
		
		\item \textbf{Efficient implementations that bypass semidefinite
			programming.}
		We describe two numerical algorithms for the univariate relaxations of the
		pushforward approach that avoid semidefinite programming entirely and rely
		exclusively on standard routines of numerical linear algebra.
		The first algorithm \textsc{Cheb} approximates a piecewise constant
		function by a Chebyshev polynomial of degree $2r$, so that each relaxation
		reduces to the computation of Chebyshev coefficients by a fast cosine
		transform. The second algorithm \textsc{Gevp}  extracts two-sided
		volume bounds from a single generalized eigenvalue problem for the moment
		and localizing matrices, of size {exactly $r+1$}, formed from the pushforward
		moments. {In addition, computing all $2r$ entries of this moment matrix can be done
			fully in parallel.}
		Neither algorithm calls an interior-point solver, and both run at
		high relaxation order $r$ at a cost dominated by dense linear algebra. 
		The only stage whose cost depends on the ambient dimension is the computation of
		the box moments of $g$; it is performed once, independently of the relaxation
		order, and is polynomial in $n$ whenever $g$ is sparse or block-separable.
		
		\item \textbf{Rigorous convergence rates that transfer to the moment-SOS
			hierarchy.}
		Both algorithms come with strong convergence guarantees that we prove
		rigorously. \textsc{Cheb} converges at the algebraic rate $O(1/r)$ in the
		relaxation order $r$, the half-degree of the moments and polynomials
		involved. \textsc{Gevp} converges at the
		exponential rate $O(\gamma^{-2r})$, with
		\[
		\gamma \;:=\; \frac{\sqrt{\bg}+1}{\sqrt{\bg}-1} \;>\; 1,
		\]
		where $\bg > 1$ is the auxiliary upper bound on $g$ entering the pushforward
		construction. Furthermore, the univariate
		polynomials returned by either algorithm can be post-processed into
		\emph{admissible} feasible polynomials for the SOS relaxation of the
		moment-SOS volume hierarchy; since these
		explicit polynomials are suboptimal at the rate above, the same $O(1/r)$ and
		$O(\gamma^{-2r})$ rates hold for the hierarchy itself. This is a significant
		improvement over the state-of-the-art convergence rates for the moment-SOS
		hierarchy established in \cite{schlosserlazarev}. To the best of our knowledge, the
		exponential rate $O(\gamma^{-2r})$ is moreover the only one known for the
		moment-SOS hierarchy, beyond the trigonometric polynomial optimization result of \cite{bachrudi}, which requires regularity (curvature) conditions
		on the minimizers.
		
	\end{enumerate}
	
	In short, the scalar pushforward approach, implemented with nothing more than
	Chebyshev approximation and a generalized eigenvalue solver, yields
	two-sided volume bounds that come with rigorous algebraic and exponential
	convergence rates, and these rates carry over to the moment-SOS hierarchy
	itself.
	
	\section{Quasi-homogeneous pushforward representation}\label{sec:pushforward}
	
	This section carries out the reduction on which the whole paper rests: the
	$n$-dimensional volume $\vol K$ of the quasi-homogeneous sublevel set is
	rewritten, \emph{exactly}, as $2^n$ times the mass on $[0,1]$ of the univariate
	pushforward measure of the uniform probability measure on the box $B$
	through the map $g$,  and the moments of the restriction on $[0,1]$
	are pinned down --- up to the single unknown scalar $\vol K$ --- by a weighted
	Stokes identity. The pushforward collapses the ambient dimension into a
	one-dimensional truncated moment problem on $[0,\bg]$, and the weighted Euler
	identity for quasi-homogeneous $g$ is what makes the Stokes step closed-form.
	
	\subsection{Univariate moment problem}\label{sec:pf_univariate}
	
	Let $\mu$ denote the uniform probability measure on $B$,
	$d\mu(x) = 2^{-n}\mathbf{1}_B\,dx$. Its pushforward measure 
	\[
	\nu := g_\# \mu
	\]
	through the map $g$  is such that
	\[
	\nu(A) = \mu(\{x \in \R^n : g(x) \in A\}) = 2^{-n}\,\vol(\{x \in B : g(x) \in A\})
	\]
	for all Borel measurable set $A \subset \R$.
	Equivalently, for every bounded measurable $\varphi : \R \to \R$, it holds
	\begin{equation}\label{eq:cov}
		\int_B \varphi\bigl(g(x)\bigr)\, d\mu(x) \;=\; \int_0^{\bg} \varphi(t)\, d\nu(t)
	\end{equation}
	where $\bar g$ is the maximum of $g$ on B as defined in \eqref{eq:bg}.
	By definition of the pushforward, $\nu$ is a \emph{probability} measure
	supported on $[0,\bg]$ (total mass
	$\nu([0,\bg]) = \mu(\R^n) = 1$); its left endpoint is $0$ because
	$g(0) = 0$ (every monomial of $g$ has positive weighted degree $m \geq 1$, so $g$
	has no constant term) and $0 \in B$, while its right endpoint is $\bg$ by
	definition. Two specializations of~\eqref{eq:cov} drive the rest of the paper.
	
	Taking $\varphi = \mathbf{1}_{[0,1]}$ and using
	$g^{-1}([0,1]) \cap B = \{x \in B : g(x) \leq 1\} = K$ gives the
	\emph{fundamental representation}
	\begin{equation}\label{eq:vol_eq_nu}
		\vol K \;=\; 2^n\,\nu([0,1]).
	\end{equation}
	Computing $\vol K$ is thus equivalent to evaluating the mass that the univariate
	measure $\nu$ places on the sub-interval $[0,1] \subset [0,\bg]$. Taking
	$\varphi(t) = t^k$ gives the moment sequence
	\begin{equation}\label{eq:y_moments}
		y_k \;:=\; \int_0^{\bg} t^k\, d\nu(t) \;=\; \int_B g(x)^k\, d\mu(x)
		\;=\; \frac{1}{2^n}\int_B g(x)^k\, dx,
		\qquad k \in \N,
	\end{equation}
	so that $y_0 = 1$ and the moments of the univariate $\nu$ are precisely the box
	integrals of the powers of $g$, normalized by $\vol B = 2^n$. For polynomial $g$
	and the box $B$ these are available in closed form, as will be seen later on.
	
	A measure on a compact interval is uniquely determined by its moments, and the
	truncated moment sequences of~$\nu$ are characterized by the classical
	\emph{Hausdorff} positive-semidefiniteness conditions on $[0,\bg]$: for every
	$r \geq 0$,
	\begin{equation}\label{eq:hausdorff_cone}
		M_r(\nu) \;\succeq\; 0, \qquad M_{r-1}((\bg - t)\,t\,\nu) \;\succeq\; 0,
	\end{equation}
	where $M_r(\nu)_{ij} = \int t^{i+j}\,d\nu(t)$ is the
	$(r+1)\times(r+1)$ Hankel moment matrix of $\nu$ and
	$M_{r-1}(h\nu)_{ij} = \int h(t)\,t^{i+j}\,d\nu(t)$ its localizing matrix
	relative to $h$; a real sequence $(y_k)$ is the moment sequence of a
	non-negative measure on $[0,\bg]$ if and only if~\eqref{eq:hausdorff_cone} holds
	for all $r$ (see, e.g., \cite[Thm.~3.2(c)]{lasserre2010}). The localizer $(\bg-t)\,t$
	encodes membership in $[0,\bg]$.
	
	\begin{figure}[t]
		\centering
		\includegraphics[width=0.8\linewidth]{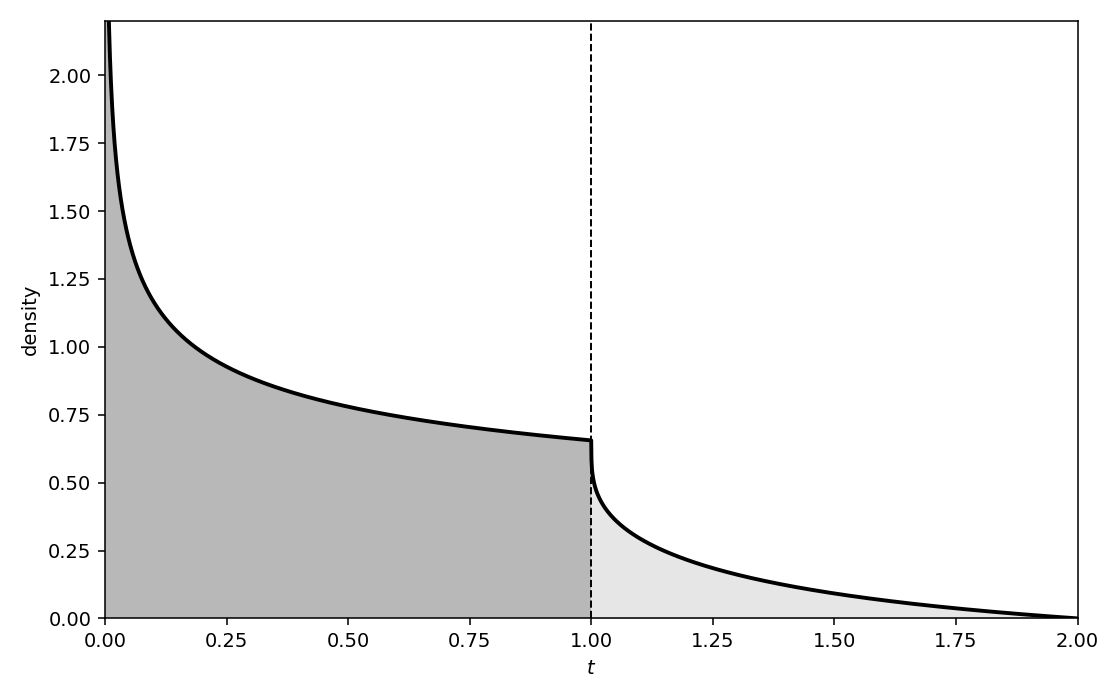}
		\caption{Density (solid red curve) of the pushforward $\nu=g_\#\mu$ of the uniform probability measure
			$\mu$ on $B=[-1,1]^2$ under $g(x)=x_1^4+x_2^2$ with $\bar g=2$. The darker area over
			$(0,1)$ equals the normalized volume of the unit sublevel set
			$\{g\le1\}$; the lighter area over $(1,2)$ carries the complementary mass.
		}
		\label{fig:pushforward}
	\end{figure}
	
	\paragraph{Running example}
	Let $\mu$ be the uniform probability measure on the box $B=[-1,1]^2$,
	$d\mu(dx)=\tfrac14\mathbf 1_B\,dx$, and let $\nu=g_\#\mu$ be its pushforward
	under $g(x)=x_1^4+x_2^2$; since $g(B)=[0,2]$, it holds $\bg=2$ and $\nu$ is supported on $[0,2]$.
	Disintegrating $\mu$ along the fibres of the first coordinate, for fixed
	$x_1$ the law of $x_2^2$ has density $1/(2\sqrt{y})$ on $(0,1)$, so the
	conditional law of $g$ given $x_1$ has density $1/\bigl(2\sqrt{t-x_1^4}\bigr)$
	on $(x_1^4,\,x_1^4+1)$. Averaging over $x_1$ yields the density of $\nu$,
	\[
	\rho(t)=\frac12\int_{\max((t-1)^{1/4},0)}^{\min(1,\,t^{1/4})}
	\frac{dx_1}{\sqrt{t-x_1^4}},\qquad t\in(0,2).
	\]
	For $t\le1$ the lower limit vanishes and the scaling $x_1=t^{1/4}s$ gives the
	pure power law $\rho(t)=\tfrac{\varpi}{4}\,t^{-1/4}$ with
	$\varpi=\Gamma(1/4)^2/(2\sqrt{2\pi})$ the lemniscate constant; for $t>1$ the
	level set $\{g=t\}$ leaves $B$ through the edges $x_2=\pm1$, the lower limit
	becomes $(t-1)^{1/4}$, and $\rho$ is a difference of incomplete elliptic
	integrals of the first kind of modulus $1/\sqrt2$, producing the kink at
	$t=1$ visible in Figure~\ref{fig:pushforward}.
	
	By construction the distribution function of $\nu$ is a normalized sublevel
	volume,
	\[
	\nu\bigl([0,t]\bigr)=\mu\bigl(\{x\in B:g(x)\le t\}\bigr)
	=\frac{\operatorname{vol}\bigl(\{g\le t\}\cap B\bigr)}{\operatorname{vol}B} .
	\]
	For $t\le1$ one has $\{g\le t\}\subset B$, so the clipping by $B$ is inactive
	and $\nu([0,t])$ is the \emph{exact} normalized volume of the sublevel set.
	In particular, the darker shaded area over $(0,1)$ in
	Figure~\ref{fig:pushforward} is the normalized volume of the unit sublevel
	set of $g$,
	\[
	\int_0^1\rho(t)\,dt=\nu\bigl([0,1]\bigr)
	=\frac{\operatorname{vol}\{g\le1\}}{\operatorname{vol}B}
	=\frac{\Gamma(1/4)^2}{6\sqrt{2\pi}}\approx0.8740,
	\]
	so that $\operatorname{vol}\{x:x_1^4+x_2^2\le1\}=4\,\nu([0,1])\approx3.4961$,
	while the lighter area over $(1,2)$ accounts for the remaining mass
	$1-\nu([0,1])\approx0.1260$.

	The representation~\eqref{eq:vol_eq_nu} is exact, but it cannot be exploited from
	the moments~\eqref{eq:y_moments} alone: the obstruction is a genuine
	ill-posedness, recorded next, and it is exactly what is addressed by the Stokes structure introduced next.
	
	\begin{remark}[Ill-posedness of the bare moment problem]\label{rem:illposed}
		The functional $\sigma \mapsto \sigma([0,1])$ is not weak-$\star$ continuous on
		the moment cone, so $\nu([0,1])$ cannot be recovered in a stable way from the
		moments $(y_k)$ alone. On a compact set, convergence of all moments is
		equivalent (by Stone-Weierstrass) to weak-$\star$ convergence, which only
		transmits $\int \varphi\, d\sigma_k \to \int \varphi\, d\sigma$ for
		\emph{continuous} $\varphi$; the indicator $\mathbf{1}_{[0,1]}$ is discontinuous
		at the endpoints. Concretely, with $\sigma_k = \delta_{1+1/k}$ and
		$\sigma = \delta_1$ one has $\int t^j\, d\sigma_k = (1+1/k)^j \to 1 = \int t^j\,
		d\sigma$ for every $j$, yet $\sigma_k([0,1]) = 0$ for all $k$ while
		$\sigma([0,1]) = 1$: a unit of mass appears in the limit without any moment
		detecting it. Restoring well-posedness requires additional structural
		information --- here, the moment constraints that quasi-homogeneity imposes on
		$\nu|_{[0,1]}$.
	\end{remark}

	\subsection{The weighted Euler identity}\label{sec:pf_euler}
	
	The structural information promised by Remark~\ref{rem:illposed} comes from
	differentiating the defining relation of quasi-homogeneity.
	
	\begin{lemma}[Weighted Euler identity]\label{lem:weuler}
		For every $x \in \R^n$,
		\begin{equation}\label{eq:weighted_euler}
			(\wox) \cdot \nabla g(x) \;=\; m\, g(x),
			\qquad (\wox) := (w_1 x_1, \ldots, w_n x_n).
		\end{equation}
	\end{lemma}
	
	\begin{proof}
		Fix $x \in \R^n$ and apply $\frac{d}{dt}$ to the
		identity~\eqref{eq:qh}, $g(t^{w_1}x_1, \ldots, t^{w_n}x_n) = t^m g(x)$, valid
		for $t > 0$. By the chain rule the left-hand side has derivative
		$\sum_{i} w_i\, t^{w_i - 1} x_i\, (\partial_i g)(t^{w_1}x_1, \ldots, t^{w_n}x_n)$
		and the right-hand side has derivative $m\,t^{m-1} g(x)$. Evaluating at $t = 1$
		gives $\sum_i w_i x_i\, \partial_i g(x) = m\, g(x)$, which
		is~\eqref{eq:weighted_euler}. Both sides are polynomials in $x$, so the
		identity extends to all $x \in \R^n$.
	\end{proof}
	
	Two consequences are used repeatedly.
	
	\begin{corollary}[Regularity and power rule]\label{cor:regular}
		Under Assumption \ref{ass:pos}:
		\begin{enumerate}[label=\textup{(\roman*)},leftmargin=*,itemsep=0pt]
			\item every positive value of $g$ is regular; in particular $\{g = 1\}$ is a
			smooth compact hypersurface and the divergence theorem applies on
			$K = \{g \leq 1\}$;
			\item for every integer $k \geq 0$,
			$(\wox) \cdot \nabla(g^k) = k\,m\, g^k$.
		\end{enumerate}
	\end{corollary}
	
	\begin{proof}
		(i) If $\nabla g(x) = 0$ for some $x$, then~\eqref{eq:weighted_euler} gives
		$m\, g(x) = (\wox)\cdot \nabla g(x) = 0$, hence $g(x) = 0$ (as $m \geq 1$), hence
		$x = 0$ by Assumption \ref{ass:pos}. Thus $\nabla g$ does not vanish on
		$\{g > 0\}$, so every positive value of $g$ is regular; coercivity (a
		consequence of Assumption \ref{ass:pos}) makes the level sets compact. (ii) By the
		chain rule and~\eqref{eq:weighted_euler},
		$(\wox)\cdot\nabla(g^k) = k\,g^{k-1}\,(\wox)\cdot\nabla g = k\,g^{k-1}\,m\,g
		= k m\, g^k$.
	\end{proof}
	
	\subsection{Weighted Stokes moment identity}\label{sec:pf_stokes}
	
	We now show that every moment of $g$ over $K$ is a known rational multiple of
	$\vol K$ itself. The mechanism is a divergence identity for a weighted
	radial vector field that vanishes on $\partial K$.
	
	\begin{lemma}[Weighted Stokes moment identity]\label{lem:wstokes}
		For every $k \in \N$,
		\begin{equation}\label{eq:stokes_moment}
			\int_K g(x)^k\, dx \;=\; z_k\, \vol K,
			\qquad z_k \;:=\; \frac{\ws}{\ws + km}.
		\end{equation}
	\end{lemma}
	
	\begin{proof}
		The case $k = 0$ is the identity $\int_K dx = \vol K$, with $z_0 = 1$. Fix
		$k \geq 1$ and consider the weighted radial vector field
		\[
		X_k(x) \;:=\; (\wox)\,\bigl(1 - g(x)^k\bigr),
		\]
		which is of class $C^\infty$. Its divergence is, using
		$\dive(\wox) = \sum_i \partial_i(w_i x_i) = \sum_i w_i = \ws$ and
		Corollary~\ref{cor:regular}(ii),
		\[
		\dive X_k
		\;=\; (1 - g^k)\,\dive(\wox) \;-\; (\wox)\cdot \nabla(g^k)
		\;=\; \ws\,(1 - g^k) - k m\, g^k
		\;=\; \ws - (\ws + km)\, g^k.
		\]
		By Corollary~\ref{cor:regular}(i), $K$ is compact with smooth boundary
		$\partial K = \{g = 1\}$, so the divergence theorem yields
		\[
		\int_K \dive X_k \, dx \;=\; \int_{\partial K} X_k \cdot n_K \, d\sigma .
		\]
		On $\partial K$ one has $g = 1$, hence $g^k = 1$, hence
		$X_k = (\wox)\,(1 - 1) = 0$; the boundary integral vanishes. Therefore
		$\int_K \bigl[\ws - (\ws + km) g^k\bigr]\, dx = 0$, that is,
		$\ws\, \vol K = (\ws + km)\int_K g^k\, dx$, which is~\eqref{eq:stokes_moment}.
	\end{proof}
	
	For unweighted homogeneity $w = (1,\ldots,1)$ one has $\ws = n$ and recovers the
	classical homogeneous identity $z_k = n/(n + km)$ of
	\cite[Lemma~1]{lasserre2013} (the $\alpha = 0$ instance of the level-set moment
	formula), used in the Stokes-accelerated pushforward of
	\cite[Lemma~3.1]{lasserre2019}.
	
	\paragraph{Running example}
	For $g = x_1^4 + x_2^2$ with $w = (1,2)$, $m = 4$ one has $\ws = 3$ and
	$z_j = 3/(3 + 4j)$, so $z_0 = 1$, $z_1 = 3/7$, $z_2 = 3/11$, and so on. The
	weighted Euler identity reads
	$(\wox)\cdot\nabla g = x_1(4x_1^3) + 2x_2(2x_2) = 4x_1^4 + 4x_2^2 = 4g$, in
	agreement with $m = 4$.
	
	\subsection{Reformulation as a moment cone restriction}\label{sec:pf_reform}
	
	Lemma~\ref{lem:wstokes} states that the moments of the restriction $\nu|_{[0,1]}$
	are the numbers $2^{-n} z_j\,\vol K$. There is a canonical measure carrying exactly the
	moments $z_j$.
	
	\begin{lemma}[Reference measure]\label{lem:muw}
		Let $\nu_w$ be the probability measure on $[0,1]$ with density
		\begin{equation}\label{eq:mu_w_density}
			\frac{d\nu_w}{dt}(t) \;=\; \frac{\ws}{m}\, t^{\frac{\ws}{m} - 1},
			\qquad t \in (0,1],
		\end{equation}
		i.e.\ the $\mathrm{Beta}(\frac{\ws}{m},\,1)$ distribution, the law of $X^{m/\ws}$ for $X \sim \mathrm{Unif}[0,1]$. Its moments are
		$\int_0^1 t^k\, d\nu_w(t) = z_k$ for every $k \in \N$.
	\end{lemma}
	
	\begin{proof}
		Direct integration: for $k \geq 0$,
		\[
		\int_0^1 t^k\, \frac{\ws}{m}\, t^{\frac{\ws}{m} - 1}\, dt
		\;=\; \frac{\ws}{m}\int_0^1 t^{\,k + \frac{\ws}{m} - 1}\, dt
		\;=\; \frac{\ws}{m}\cdot \frac{1}{\,k + \frac{\ws}{m}\,}
		\;=\; \frac{\ws}{\ws + km}
		\;=\; z_k,
		\]
		the integral being finite since $k + \frac{\ws}{m} > 0$. Total mass $z_0 = 1$ confirms
		$\nu_w$ is a probability measure.
	\end{proof}
	
	Comparing moments through Lemma~\ref{lem:wstokes} and Lemma~\ref{lem:muw} pins the
	restriction $\nu|_{[0,1]}$ exactly.
	
	\begin{proposition}[Moment cone restriction]\label{prop:nu_restriction}
		The restriction of $\nu$ to $[0,1]$ is a known scalar multiple of $\nu_w$:
		\begin{equation}\label{eq:nu_restriction}
			\nu|_{[0,1]} \;=\; 2^{-n}\,\vol K \cdot \nu_w .
		\end{equation}
		Consequently the residual measure
		$\hat\nu := \nu - 2^{-n}\,\vol K \cdot \nu_w$ on $[0,\bg]$ is non-negative and supported
		in $[1,\bg]$, with moments
		\begin{equation}\label{eq:residual_moments}
			\hat y_k \;:=\; \int_0^{\bg} t^k\, d\hat\nu(t) \;=\; y_k - 2^{-n}\,\vol K\cdot z_k,
			\qquad k \in \N.
		\end{equation}
	\end{proposition}
	
	\begin{proof}
		Both $\nu|_{[0,1]}$ and $2^{-n}\vol K \cdot \nu_w$ are finite measures on the compact
		interval $[0,1]$. For each $j \in \N$, the pushforward
		formula~\eqref{eq:cov} together with Assumption~\ref{ass:box} and Lemma~\ref{lem:wstokes} gives
		\begin{eqnarray*}
			\int_{[0,1]} t^k\, d\nu(t)
			\;=\; \int_{\{x \in B\,:\, g(x) \leq 1\}} g(x)^k\, d\mu(x)
			&=& 2^{-n}\int_K g(x)^k\, dx\\
			&=& 2^{-n} z_k\, \vol K
			\;=\; 2^{-n}\vol K \int_{[0,1]} t^k\, d\nu_w(t),
		\end{eqnarray*}
		using Lemma~\ref{lem:muw} in the last step. The two measures therefore have the
		same moments, and a measure on a compact interval is determined by its moments;
		hence~\eqref{eq:nu_restriction}. Restricting $\hat\nu$ to $[0,1]$ gives
		$\hat\nu|_{[0,1]} = \nu|_{[0,1]} - 2^{-n}\vol K\cdot\nu_w = 0$, while on $(1,\bg]$
		(where $\nu_w$ vanishes) $\hat\nu = \nu|_{(1,\bg]} \geq 0$; thus $\hat\nu \geq 0$
		with $\supp \hat\nu \subseteq [1,\bg]$. Equation~\eqref{eq:residual_moments} is
		linearity of the integral.
	\end{proof}
	
	Equation~\eqref{eq:nu_restriction} is the structural fact that defeats the
	ill-posedness of Remark~\ref{rem:illposed}: on $[0,1]$, the unknown measure
	$\nu|_{[0,1]}$ is a \emph{single scalar} $2^{-n}\vol K$ times the fully known reference
	$\nu_w$. The volume problem can now be stated as a one-parameter truncated moment
	problem on $[1,\bg]$.
	
	\medskip
	\noindent\textit{Volume as a partial moment problem.} {Find the largest
		(resp.\ smallest) scalar $\alpha \geq 0$ such that the truncated sequence
		$(y_k - \alpha\, z_k)_{k=0}^{2r}$ is the moment sequence of a non-negative
		measure on $[1,\bg]$.} By Proposition~\ref{prop:nu_restriction} the true value
	$\alpha = 2^{-n}\vol K = \nu([0,1])$, the normalized volume, is feasible for every
	$r$ (it produces the moments of $\hat\nu$), and the extremal feasible $\alpha$
	bracket $2^{-n}\vol K$; the volume is recovered as $\vol K = 2^n\alpha$.
	
	\begin{figure}[ht]
		\centering
		\includegraphics[width=\textwidth]{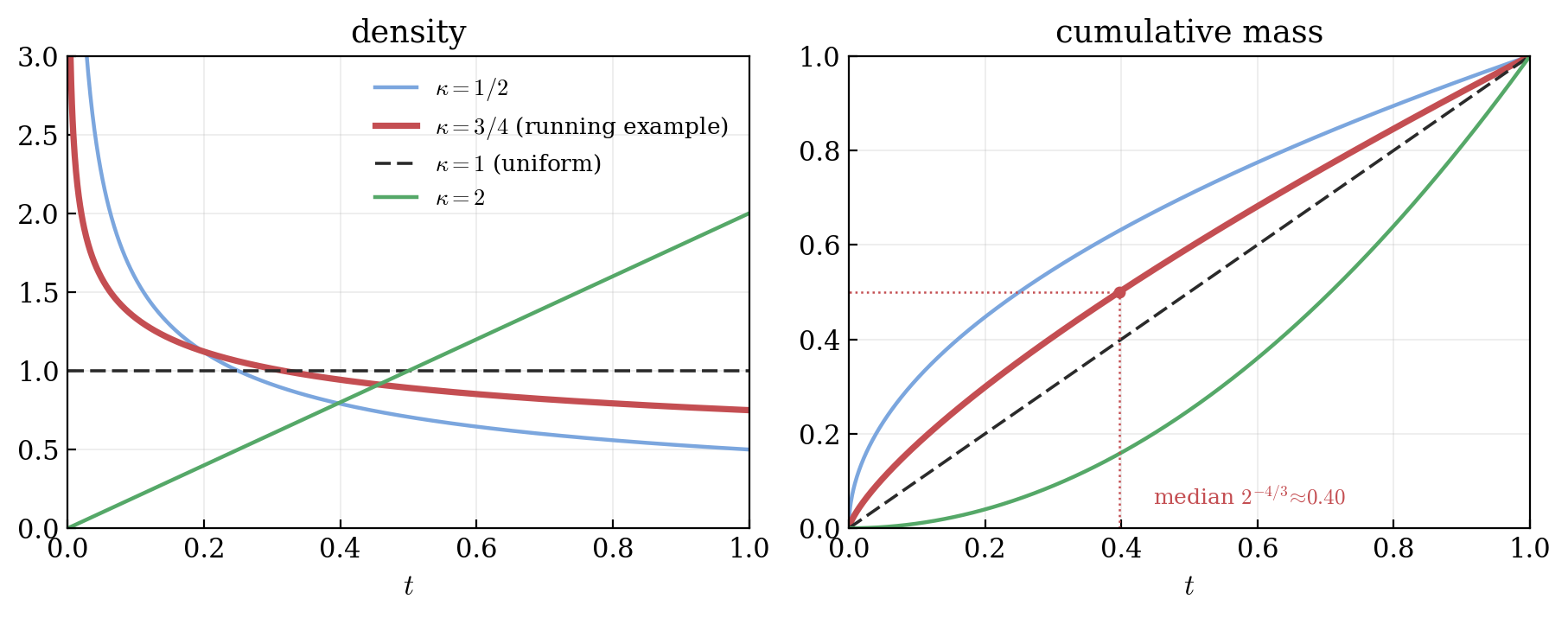}
		\caption{The reference measure $\nu_w = \mathrm{Beta}(\kappa,1)$ on $[0,1]$,
			with $\kappa = \ws/m$. Left plot: density $w_\kappa(t) = \kappa\,
			t^{\kappa-1}$: for $\kappa < 1$ it diverges at $t = 0$ with the integrable
			rate $t^{\kappa-1}$, at $\kappa = 1$ it is uniform, and for $\kappa > 1$ it
			vanishes at the origin. Right plot: cumulative mass
			$\int_0^t w_\kappa(s)\,ds = t^{\kappa}$, showing how the mass concentrates near $0$ as
			$\kappa$ decreases. The running example $\kappa = 3/4$ is highlighted, and its
			median $2^{-4/3} \approx 0.40$ is marked.}
		\label{fig:muw}
	\end{figure}
	
	\begin{remark}[Skewed weights and the singular density]\label{rem:singular_density}
		When $\kappa:=\frac{\ws}{m} < 1$, i.e. skewed weights or large degree $m$, the density~\eqref{eq:mu_w_density} is unbounded
		at $t = 0$. This is however harmless for the moment problem, since the moments
		are well defined and finite for all $k$.
		For the running example we have $\kappa = 3/4$, the density is $w_{3/4}(t) = 3/4\, t^{-1/4}$ and the
		cumulative mass is $\int_0^t w_{3/4}(s)ds = t^{3/4}$; half of the total mass lies in
		$[0,\,2^{-4/3}] \approx [0,\,0.397]$, so the measure concentrates markedly toward
		the origin while remaining finite. Figure~\ref{fig:muw} displays $w_\kappa$ and
		$\int_0^t w_\kappa(t)dt$ over a range of $\kappa$.
	\end{remark}
	
	\subsection{The weighted ReLU identity}\label{sec:pf_relu}
	
	Specializing the Stokes identity to $k = 1$ turns the discontinuous volume
	functional into the integral of a continuous test function, which is the source
	of the Gibbs-free behaviour of both algorithms.
	
	\begin{proposition}[Weighted ReLU identity]\label{prop:relu}
		For quasi-homogeneous positive $g$ with weight sum $\ws$ and degree $m$,
		\begin{equation}\label{eq:relu}
			\vol K
			\;=\; \frac{\ws + m}{m} \int_B (1 - g(x))_+\, dx
			\;=\; \frac{2^n(\ws + m)}{m} \int_0^{\bg}(1-t)_+\, d\nu(t),
		\end{equation}
		where $(t)_+ := \max(0, t)$ is the ReLU function%
		\footnote{The function $(t)_+ := \max(0, t)$ is called the
			\emph{ReLU} (\emph{Rectified Linear Unit}) in deep learning, where the
			activation $t \mapsto t_+$ is the canonical neuron nonlinearity. The related function $(1-t)_+ := \max(0, 1-t)$ is called the \emph{hinge} in machine learning, after the hinge
			loss $L(y,\hat y) = \max(0, 1 - y\hat y)$ of support vector machines: its graph
			has two linear pieces meeting at a corner, like the leaves of a door hinge.
			Classical approximation theory calls the  function $(1-t)_+$  
			a \emph{one-sided ramp}.}.
	\end{proposition}
	
	\begin{proof}
		The $k = 1$ case of Lemma~\ref{lem:wstokes} reads
		$\int_K g\, dx = z_1\, \vol K$ with $z_1 = \ws/(\ws + m)$, hence
		\[
		\int_K (1 - g)\, dx \;=\; (1 - z_1)\, \vol K
		\;=\; \Bigl(1 - \tfrac{\ws}{\ws+m}\Bigr)\vol K
		\;=\; \frac{m}{\ws + m}\, \vol K .
		\]
		Since $g \geq 0$, $(1-g)_+ = (1-g)\,\mathbf{1}_{\{g \leq 1\}}$, and under
		Assumption~\ref{ass:box}, $\{g \leq 1\} = K$, so
		$\int_B (1-g)_+\, dx = \int_K (1-g)\, dx = \tfrac{m}{\ws+m}\,\vol K$.
		Multiplying by $\frac{\ws+m}{m}$ gives the first equality
		in~\eqref{eq:relu}. For the second, the change of
		variables~\eqref{eq:cov} with $\varphi = (1-t)_+$ gives
		$\int_B (1-g)_+\, d\mu = \int_0^{\bg}(1-t)_+\, d\nu$, i.e.\
		$\int_B (1-g)_+\, dx = 2^n\int_0^{\bg}(1-t)_+\, d\nu$ since $d\mu = 2^{-n}dx$.
	\end{proof}
	
	\begin{remark}[Higher-order ReLU identities]\label{rem:relu_k}
		The same argument applied to the field $(\wox)(1 - g^j)$, i.e.\ the case
		$k = j$ of Lemma~\ref{lem:wstokes}, yields a one-parameter family of identities
		\begin{equation}\label{eq:relu_k}
			\vol K \;=\; \frac{\ws + jm}{jm} \int_B \bigl(1 - g(x)^j\bigr)_+\, dx,
			\qquad j \geq 1,
		\end{equation}
		since $\{g^j \leq 1\} = \{g \leq 1\} = K$ for $g \geq 0$. The case $j = 1$ is
		Proposition~\ref{prop:relu}; larger $j$ sharpens the slope of the continuous
		integrand near the level set.
	\end{remark}
	
	The ReLU function
	replaces the discontinuous indicator $\mathbf{1}_{[0,1]}$
	of~\eqref{eq:vol_eq_nu} by a continuous, indeed Lipschitz, test function. This is
	the single source of the Gibbs-free convergence of both algorithms below: the
	Chebyshev algorithm of Section~\ref{sec:cheb} approximates the
	\emph{continuous} integrand $(1-t)_+$ rather than a jump, and the
	generalized-eigenvalue algorithm of Section~\ref{sec:gevp} exploits the
	\emph{moment} content of the same identity through the reference measure
	$\nu_w$. For unweighted homogeneity the constant is $(\ws + m)/m = (n+m)/m$,
	recovering the hinge identity of \cite{lasserre2019}.
	
	\section{Chebyshev ReLU bounds (\textsc{Cheb})}\label{sec:cheb}
	
	The ReLU identity of Proposition~\ref{prop:relu} expresses the volume as the
	integral of the continuous univariate ReLU function $f(t) = (1-t)_+$ against the
	pushforward.
	The \textsc{Cheb} algorithm replaces $f$ by its degree-$r$ Chebyshev partial sum
	on $[0,\bg]$ --- a polynomial with \emph{closed-form} coefficients --- and
	integrates it against $\nu$ term by term. Because both the coefficients and the
	uniform approximation error are known explicitly, this produces a certified
	two-sided bracket for $\vol K$ without solving any semidefinite or linear program,
	converging at the algebraic rate $O(1/r)$.
	
	\subsection{Closed-form Chebyshev expansion of the ReLU function}\label{sec:cheb_coefs}
	
	\paragraph{Reduction to the canonical interval.}
	Consider $f : [0,\bg] \to \R$, $f(t) = (1-t)_+$. The affine map
	\begin{equation}\label{eq:affine}
		t \;=\; \tfrac{\bg}{2}(s+1), \qquad s \;=\; \tfrac{2t}{\bg}-1,
	\end{equation}
	is a bijection $[-1,1] \to [0,\bg]$ preserving polynomial degree, and it pulls $f$
	back to
	\[
	\widetilde f(s) \;:=\; f\bigl(\tfrac{\bg}{2}(s+1)\bigr) \;=\; \tfrac{\bg}{2}\,(s^\star - s)_+,
	\qquad s^\star \;:=\; \tfrac{2}{\bg}-1 \in (-1,1),
	\]
	which is supported on $[-1, s^\star]$, vanishes on $[s^\star, 1]$, and has a single
	interior corner at $s^\star$. Set
	\begin{equation}\label{eq:phidef}
		\varphi \;:=\; \arccos s^\star \in (0,\pi),
	\end{equation}
	so that, by direct computation,
	\begin{equation}\label{eq:phi_ids}
		\cos\varphi = \tfrac{2}{\bg}-1, \qquad
		\sin\varphi = \tfrac{2}{\bg}\sqrt{\bg-1}, \qquad
		\varphi = 2\arctan\sqrt{\bg-1},
	\end{equation}
	and under $s = \cos\theta$ the corner $s^\star$ corresponds to the angle
	$\theta = \varphi$.
	
	\paragraph{Chebyshev expansion and truncation.}
	Let $T_k$ denote the Chebyshev polynomial of the first kind. The $L^2(w)$-expansion of $\widetilde f$ on $[-1,1]$ with weight $w(s) = (1-s^2)^{-1/2}$ has Chebyshev coefficients
	$$c_k = \tfrac{2}{\pi}\int_{-1}^{1} \widetilde f(s)\, T_k(s)\, w(s)\, ds,$$ and we
	define the degree-$r$ truncated Chebyshev approximation of $f$ on $[0,\bg]$,
	\begin{equation}\label{eq:qkC}
		f_r^{\mathrm{ch}}(t) \;:=\; \frac{c_0}{2} + \sum_{k=1}^{r} c_k\, T_k \Bigl(\tfrac{2t}{\bg}-1\Bigr),
		\qquad t \in [0,\bg].
	\end{equation}
	Being the $L^2(w)$-orthogonal projection of $\widetilde f$ onto $\R[s]_r$ pulled
	back through~\eqref{eq:affine}, $f_r^{\mathrm{ch}}$ lies in $\R[t]_r$.
	
	\begin{theorem}[Closed-form coefficients]\label{thm:cheb_coefs}
		With $\varphi = \arccos(2/\bg - 1)$ as in~\eqref{eq:phidef}, it holds
		\begin{align}
			c_0 &= \tfrac{\bg}{\pi}\bigl[(\pi-\varphi)\cos\varphi + \sin\varphi\bigr],
			\label{eq:c0}\\
			c_1 &= -\tfrac{\bg}{4\pi}\bigl[2(\pi-\varphi)+\sin(2\varphi)\bigr],
			\label{eq:c1}\\
			c_k &= \tfrac{\bg}{\pi\, k(k^2-1)}\bigl[\cos\varphi\,\sin(k\varphi) - k\sin\varphi\,\cos(k\varphi)\bigr],
			\qquad k \geq 2. \label{eq:cj}
		\end{align}
	\end{theorem}
	
	\begin{proof}
		Substitute $s = \cos\theta$ in the defining integral. Then
		$ds = -\sin\theta\, d\theta$, $\sqrt{1-s^2} = \sin\theta$,
		$T_k(\cos\theta) = \cos(k\theta)$, and since $\widetilde f$ vanishes on
		$[s^\star,1] \leftrightarrow [0,\varphi]$ (with $s = -1 \leftrightarrow
		\theta = \pi$),
		\[
		c_k \;=\; \frac{2}{\pi}\int_\varphi^\pi \tfrac{\bg}{2}\,(\cos\varphi - \cos\theta)\,\cos(k\theta)\, d\theta
		\;=\; \frac{\bg}{\pi}\int_\varphi^\pi (\cos\varphi - \cos\theta)\,\cos(k\theta)\, d\theta .
		\]
		For $k = 0$, $\int_\varphi^\pi(\cos\varphi - \cos\theta)\,d\theta
		= (\pi-\varphi)\cos\varphi + \sin\varphi$, giving~\eqref{eq:c0}. For $k = 1$,
		using $\cos^2\theta = \tfrac12(1+\cos 2\theta)$ and
		$\cos\varphi\sin\varphi = \tfrac12\sin 2\varphi$ gives~\eqref{eq:c1}. For
		$k \geq 2$, the product-to-sum identity
		$\cos\theta\cos(k\theta) = \tfrac12[\cos((k{+}1)\theta) + \cos((k{-}1)\theta)]$
		together with $\int_\varphi^\pi \cos(l\theta)\, d\theta = -\sin(l\varphi)/l$
		(for integer $l \geq 1$, since $\sin(l\pi) = 0$) gives
		\[
		c_k \;=\; \frac{\bg}{\pi}\Bigl[-\frac{\cos\varphi\,\sin(k\varphi)}{k}
		+ \frac{\sin((k{+}1)\varphi)}{2(k+1)} + \frac{\sin((k{-}1)\varphi)}{2(k-1)}\Bigr].
		\]
		Multiplying by $\pi\, k(k^2-1)/\bg$ and expanding
		$\sin((k{\pm}1)\varphi) = \sin(k\varphi)\cos\varphi \pm \cos(k\varphi)\sin\varphi$,
		the coefficient of $\sin(k\varphi)\cos\varphi$ collapses to $[k(k-1)+k(k+1)]/2 = k^2$ and that of $\cos(k\varphi)\sin\varphi$ to $-k$;
		subtracting the $(k^2-1)\cos\varphi\sin(k\varphi)$ term leaves
		$\cos\varphi\sin(k\varphi) - k\sin\varphi\cos(k\varphi)$, which
		is~\eqref{eq:cj}.
	\end{proof}
	
	\begin{corollary}[Coefficient decay]\label{cor:cheb_decay}
		For $k \geq 2$,
		\[
		|c_k| \;\leq\; \frac{\bg\,(|\cos\varphi| + k\,|\sin\varphi|)}{\pi\, k(k^2-1)}
		\;\leq\; \frac{\bg\,(1+\sin\varphi)}{\pi\,(k^2-1)} \;=\; O(k^{-2}),
		\qquad k \to \infty .
		\]
	\end{corollary}

	\begin{corollary}[Certified uniform error]\label{cor:cheb_eps}
		For every $r \ge 2$,
		\begin{equation}\label{eq:eps_certified}
			\varepsilon_r \;\le\; \bar\varepsilon_r
			\;:=\; \sum_{k>r} |c_k|
			\;\le\; \frac{\bg\,(1+\sin\varphi)}{2\pi}\Bigl(\frac1r + \frac1{r+1}\Bigr).
		\end{equation}
	\end{corollary}
	
	\begin{proof}
		The truncation error is bounded by the tail of the coefficient series,
		$\varepsilon_r = \sup_{[0,\bg]}\bigl|\sum_{k>r} c_k T_k\bigr| \le \sum_{k>r}|c_k|$
		since $|T_k|\le 1$ on $[-1,1]$. By Corollary~\ref{cor:cheb_decay},
		$|c_k| \le \bg(1+\sin\varphi)/[\pi(k^2-1)]$, and the partial fraction
		$\tfrac{1}{k^2-1} = \tfrac12(\tfrac{1}{k-1}-\tfrac{1}{k+1})$ telescopes:
		$\sum_{k>r}\tfrac{1}{k^2-1} = \tfrac12(\tfrac1r + \tfrac1{r+1})$.
	\end{proof}
	
	\begin{figure}[ht]
		\centering
		\includegraphics[width=\textwidth]{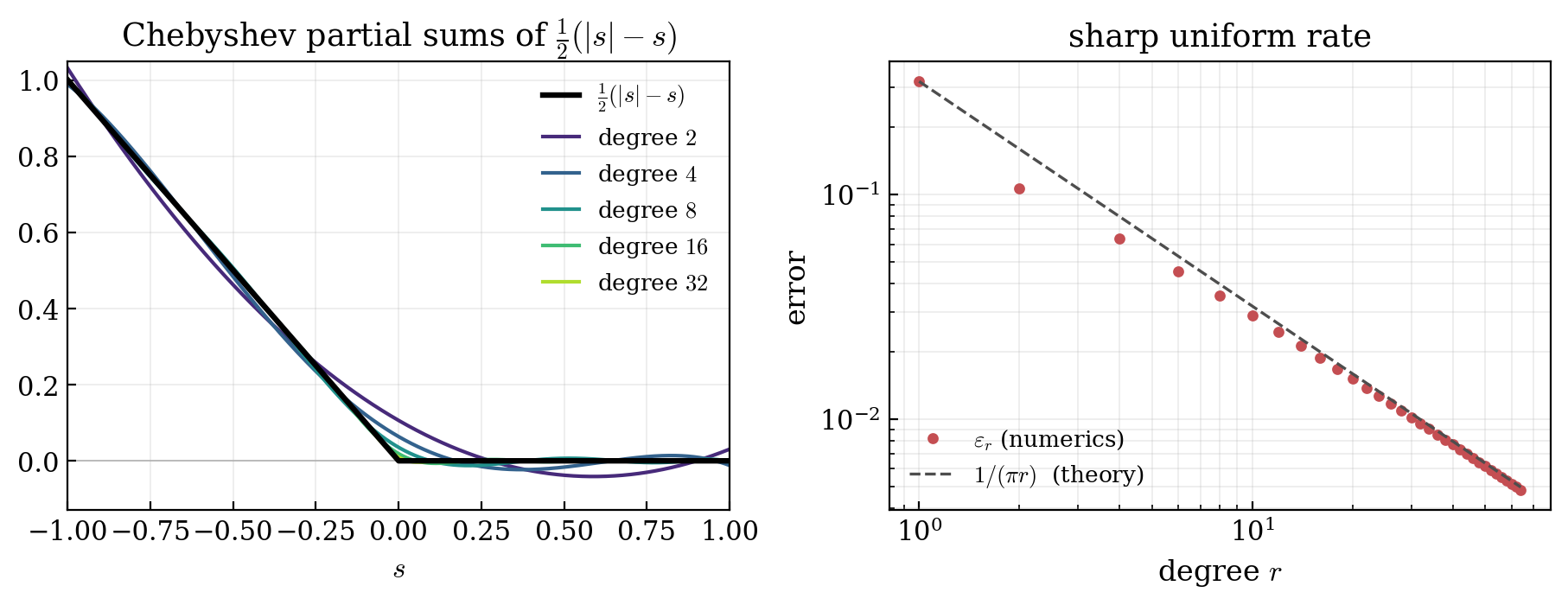}
		\caption{Chebyshev approximation of the ReLU on $[-1,1]$, i.e.\ the
			$\bg = 2$ case in which $f$ pulls back to the ramp
			$\widetilde f(s) = (-s)_+ = \tfrac12(|s| - s)$ with corner at $s = 0$.
			Left plot: the ramp (black) and its degree-$r$
			Chebyshev partial sums $\tfrac{c_0}{2} + \sum_{k=1}^r c_k\, T_k$ for
			$r = 2, 4, 8, 16, 32$, built from the closed-form coefficients $c_0 = 2/\pi$,
			$c_1 = -1/2$, $c_{2l} = (-1)^{l+1}\,\tfrac{2}{\pi(4l^2-1)}$: the partial sums
			converge uniformly, the residual oscillation concentrating at the corner. Right plot: 
			uniform approximation error versus degree on
			log--log axes, against the sharp asymptotic $\varepsilon_r \sim 1/(\pi r)$ of
			Theorem~\ref{thm:cheb_rate}.}
		\label{fig:cheb_relu}
	\end{figure}
	
	\begin{remark}[Consistency at $\bg = 2$; the running example]\label{rem:cheb_T2}
		For $\bg = 2$ one has $s^\star = 0$, $\varphi = \pi/2$, $\cos\varphi = 0$,
		$\sin\varphi = 1$, whence $c_0 = 2/\pi$, $c_1 = -1/2$, and
		$c_k = -2\cos(k\pi/2)/[\pi(k^2-1)]$; thus $c_k = 0$ for odd $k \geq 3$ and
		$c_{2l} = (-1)^{l+1}\,\tfrac{2}{\pi(4l^2-1)}$, recovering the classical
		Chebyshev expansion of $f(s)=\tfrac12(|s| - s)$ on $[-1,1]$, see Figure \ref{fig:cheb_relu}.
	\end{remark}
	
	\paragraph{Sharp uniform error rate.}
	The corner of $f$ is only a first-order singularity, and the truncation error decays at
	the sharp algebraic rate $1/r$, with an explicit constant.
	
	\begin{theorem}[Sharp uniform error]\label{thm:cheb_rate}
		For every fixed $\bg > 1$, the truncation error
		\begin{equation}\label{eq:eps_def}
			\varepsilon_r \;:=\; \sup_{t \in [0,\bg]} |f(t) - f_r^{\mathrm{ch}}(t)|
		\end{equation}
		satisfies
		\begin{equation}\label{eq:relu_rate}
			\lim_{r\to\infty} r\,\varepsilon_r \;=\; \frac{\bg\sin\varphi}{2\pi}
			\;=\; \frac{\sqrt{\bg-1}}{\pi},
			\qquad\text{i.e.}\qquad
			\varepsilon_r \;\sim\; \frac{\sqrt{\bg-1}}{\pi\, r}.
		\end{equation}
	\end{theorem}
	
	The proof rests on a single trigonometric estimate.
	
	\begin{lemma}[Trigonometric remainder]\label{lem:remainder}
		For $\displaystyle R_r(\alpha) := \sum_{k>r}\frac{\cos(k\alpha)}{k^2-1}$,
		$\alpha \in \R$, the series converges absolutely and uniformly, and:
		\begin{enumerate}[label=\textup{(\roman*)},leftmargin=*,itemsep=1pt]
			\item $R_r(0) = \tfrac12\bigl(\tfrac1r + \tfrac1{r+1}\bigr) = \tfrac1r + O(r^{-2})$;
			\item $|R_r(\alpha)| \leq \dfrac{1}{r(r+2)\,|\sin(\alpha/2)|}
			= O(r^{-2})$ for
			$\alpha \not\equiv 0 \pmod{2\pi}$.
		\end{enumerate}
		Moreover $\sup_{\alpha\in\R} |R_r(\alpha)| = R_r(0)$.
	\end{lemma}
	
	\begin{proof}
		The partial fraction $\tfrac{1}{k^2-1} = \tfrac12\bigl(\tfrac{1}{k-1} -
		\tfrac{1}{k+1}\bigr)$ telescopes to (i). 
		For (ii), fix $\alpha\not\equiv 0 \pmod{2\pi}$ and set
		\[
		a_k:=\frac{1}{k^2-1},
		\qquad
		B_n(\alpha):=\sum_{k=r+1}^{n}\cos(k\alpha).
		\]
		By the geometric-series formula,
		\[
		\left|B_n(\alpha)\right|
		\leq
		\left|\sum_{k=r+1}^{n}e^{ik\alpha}\right|
		=
		\left|
		e^{i(r+1)\alpha}
		\frac{1-e^{i(n-r)\alpha}}{1-e^{i\alpha}}
		\right|
		\leq
		\frac{2}{|1-e^{i\alpha}|}
		=
		\frac{1}{|\sin(\alpha/2)|}.
		\]
		Abel summation gives, for every $N>r$,
		\[
		\sum_{k=r+1}^{N}a_k\cos(k\alpha)
		=
		a_N B_N(\alpha)
		+
		\sum_{k=r+1}^{N-1}(a_k-a_{k+1})B_k(\alpha).
		\]
		Since $(a_k)$ is positive and decreasing, one has
		$a_k-a_{k+1}\geq 0$, and therefore
		\[
		\begin{aligned}
			\left|
			\sum_{k=r+1}^{N}a_k\cos(k\alpha)
			\right|
			&\leq
			\frac{1}{|\sin(\alpha/2)|}
			\left(
			a_N+\sum_{k=r+1}^{N-1}(a_k-a_{k+1})
			\right)\\
			&=
			\frac{1}{|\sin(\alpha/2)|}
			\left(a_N+a_{r+1}-a_N\right)\\
			&=
			\frac{a_{r+1}}{|\sin(\alpha/2)|}.
		\end{aligned}
		\]
		Letting $N\to\infty$ and using
		\[
		a_{r+1}
		=
		\frac{1}{(r+1)^2-1}
		=
		\frac{1}{r(r+2)}
		\]
		yields
		\[
		|R_r(\alpha)|
		\leq
		\frac{1}{r(r+2)|\sin(\alpha/2)|}.
		\]
		Thus, for every fixed $\alpha\not\equiv0\pmod{2\pi}$,
		one has $R_r(\alpha)=O(r^{-2})$. Notice that the estimate is not
		uniform as $\alpha\to0$, consistently with part~(i), where
		$R_r(0)$ is of order $r^{-1}$.
		Finally each summand is dominated by $1/(k^2-1)$ with equality at $\alpha = 0$,
		so the supremum equals $R_r(0)$.
	\end{proof}
	
	\begin{proof}[Proof of Theorem~\ref{thm:cheb_rate}]
		The affine map~\eqref{eq:affine} preserves the supremum norm, so
		$\varepsilon_r = \sup_{\theta\in[0,\pi]} |e_r(\theta)|$ with
		$e_r(\theta) := f_r^{\mathrm{ch}} - f$ evaluated at $s = \cos\theta$, i.e.\ $e_r(\theta) = -\sum_{k>r} c_k\cos(k\theta)$ (writing $s = \cos\theta$). For
		$k \geq 2$ split $c_k = \alpha_k + \beta_k$ from~\eqref{eq:cj},
		\[
		\alpha_k \;=\; -\frac{\bg\sin\varphi}{\pi}\,\frac{\cos(k\varphi)}{k^2-1},
		\qquad
		\beta_k \;=\; \frac{\bg\cos\varphi}{\pi}\,\frac{\sin(k\varphi)}{k(k^2-1)},
		\qquad |\beta_k| \leq \frac{\bg}{\pi\, k(k^2-1)} = O(k^{-3}),
		\]
		so that $\sum_{k>r}\beta_k\cos(k\theta) = O(r^{-2})$ uniformly. With
		$\cos(k\varphi)\cos(k\theta) = \tfrac12[\cos(k(\theta{-}\varphi)) +
		\cos(k(\theta{+}\varphi))]$,
		\[
		e_r(\theta) \;=\; \frac{\bg\sin\varphi}{2\pi}\bigl[R_r(\theta-\varphi)
		+ R_r(\theta+\varphi)\bigr] + O(r^{-2}).
		\]
		For $\theta \in [0,\pi]$ the angle $\tfrac{\theta+\varphi}{2}$ ranges in
		$[\tfrac{\varphi}{2}, \tfrac{\pi+\varphi}{2}] \subset (0,\pi)$, on which the sine
		is bounded below, so $R_r(\theta+\varphi) = O(r^{-2})$ uniformly by
		Lemma~\ref{lem:remainder}(ii); and
		$\sup_\theta |R_r(\theta-\varphi)| = \sup_\alpha |R_r(\alpha)| = R_r(0) =
		\tfrac1r + O(r^{-2})$. Hence $\varepsilon_r \leq \tfrac{\bg\sin\varphi}{2\pi}
		R_r(0) + O(r^{-2})$, while the choice $\theta = \varphi$ gives
		$e_r(\varphi) = \tfrac{\bg\sin\varphi}{2\pi} R_r(0) + O(r^{-2})$, so
		$\varepsilon_r \geq e_r(\varphi)$ matches the upper bound. Therefore
		$\varepsilon_r = \tfrac{\bg\sin\varphi}{2\pi r}(1+o(1))$, and
		$\bg\sin\varphi = 2\sqrt{\bg-1}$ from~\eqref{eq:phi_ids}
		gives~\eqref{eq:relu_rate}.
	\end{proof}
	
	\begin{remark}[The error concentrates at the corner; first-order Gibbs]\label{rem:gibbs}
		The proof shows the supremum is attained, asymptotically, at the corner
		$s^\star$ --- equivalently $t = 1$, the level set itself --- the trace of a
		Gibbs-type phenomenon. Because the singularity of $f$ is only first order (a jump in
		$f'$), the partial sums nonetheless converge uniformly, at the rate $1/r$; the
		discontinuous indicator $\mathbf{1}_{[0,1]}$ of~\eqref{eq:vol_eq_nu} would
		instead suffer the full logarithmic Gibbs overshoot. This is the quantitative
		payoff of passing through the ReLU identity.
	\end{remark}
	
	\begin{remark}[Near-optimality and Bernstein's constant]\label{rem:bernstein}
		The closed-form truncation is only a universal constant worse than the best
		polynomial approximation. Writing $E_r(f) := \inf_{p\in\R[t]_r}\sup_{[0,\bg]}|f-p|$,
		the affine reduction gives $E_r(f) = \tfrac{\bg}{4}\,E_r(|s-s^\star|;[-1,1])$,
		and Bernstein's asymptotics for the best approximation of $|s-s^\star|$ on
		$[-1,1]$ yield
		\[
		\lim_{r\to\infty} r\, E_r(f) \;=\; \frac{\beta}{2}\sqrt{\bg-1},
		\qquad \beta = 0.28017\ldots,
		\]
		$\beta$ being Bernstein's constant~\cite[Thm.~25.1]{trefethen2013};	
		the local factor $\sqrt{1-(s^\star)^2} = \sin\varphi$ is the reciprocal
		(up to $1/\pi$) of the density of the Chebyshev equilibrium measure of $[-1,1]$
		at the corner~\cite[Eq.~(12.10)]{trefethen2013}. Comparing with~\eqref{eq:relu_rate},
		\[
		\lim_{r\to\infty}\frac{\sup_{[0,\bg]}|f-f_r^{\mathrm{ch}}|}{E_r(f)} \;=\; \frac{2}{\pi\beta}
		\;\approx\; 2.2723,
		\]
		independently of $\bg$: the closed-form Chebyshev majorant pays only a factor
		$\approx 2.2723$ over the (non-closed-form) Remez optimum, and in particular escapes the $O(\log r)$ Lebesgue-constant
		penalty~\cite[Thms.~15.1 and~15.3]{trefethen2013} that the Chebyshev projection
		incurs on a generic continuous function.
	\end{remark}
	
	\subsection{The \textsc{Cheb} bracket}\label{sec:cheb_bracket}
	
	Inserting $f_r^{\mathrm{ch}}$ into the ReLU identity~\eqref{eq:relu} gives the volume
	estimate
	\begin{equation}\label{eq:Vk}
		V_r \;:=\; \frac{2^n(\ws+m)}{m}\Bigl(\frac{c_0}{2}\, y_0
		+ \sum_{k=1}^{r} c_k\, y_k^{\mathrm{ch}}\Bigr),
		\qquad
		y_k^{\mathrm{ch}} \;:=\; \int_B T_k\Bigl(\tfrac{2g(x)}{\bg}-1\Bigr)\, d\mu(x)
		\;=\; \int_0^{\bg} T_k\Bigl(\tfrac{2t}{\bg}-1\Bigr)\, d\nu(t) ,
	\end{equation}
	where $y_0 = \int_0^{\bg} d\nu = 1$ (the case $T_0 \equiv 1$) and the
	\emph{Chebyshev moments} $y_k^{\mathrm{ch}}$ replace the monomial moments
	$y_k$ of~\eqref{eq:y_moments}.
	
	\paragraph{Computing the Chebyshev moments.}
	The $y_k^{\mathrm{ch}}$ are evaluated by tensor Gauss-Legendre cubature for the
	uniform probability measure $\mu$ on $B$.
	By quasi-homogeneity the per-coordinate degree of $g$ in $x_i$ is at most
	$m/w_i$, so $T_k(2g/\bg-1)$ has per-coordinate degree at most $k\,m/w_i$; taking
	$Q_i = \lceil (r\,m/w_i + 1)/2\rceil$ Gauss-Legendre nodes in coordinate $i$
	integrates $y_0^{\mathrm{ch}}, \ldots, y_r^{\mathrm{ch}}$ \emph{exactly}. On the
	quadrature grid the Chebyshev values are produced by the three-term recurrence
	\begin{equation}\label{eq:cheb_rec}
		T_0(z) = 1, \quad T_1(z) = z, \quad T_{k+1}(z) = 2z\,T_k(z) - T_{k-1}(z),
		\qquad z = \tfrac{2g(x)}{\bg}-1 ,
	\end{equation}
	and since $g(x) \in [0,\bg]$ on $B$ forces $z \in [-1,1]$, every intermediate
	value stays bounded by $1$; correspondingly $|y_k^{\mathrm{ch}}| \leq 1$. This
	is the decisive conditioning advantage of \textsc{Cheb} over a monomial
	evaluation of the same integrals, whose moments $y_k = \int_B g^k\,d\mu$ grow like
	$\bg^{\,k}$ and induce catastrophic cancellation in~\eqref{eq:Vk}.
	
	\paragraph{Certified bracket.}
	By definition of $\varepsilon_r$, shifting $f_r^{\mathrm{ch}}$ by $\pm\varepsilon_r$ produces
	explicit polynomial minorant and majorant of $f$ on $[0,\bg]$,
	\[
	f_r^{\mathrm{ch}}(t) - \varepsilon_r \;\leq\; f(t) \;\leq\; f_r^{\mathrm{ch}}(t) + \varepsilon_r,
	\qquad t \in [0,\bg],
	\]
	and integrating against the pushforward turns these into a two-sided bracket.
	
	\begin{proposition}[\textsc{Cheb} certified bracket]\label{prop:cheb_bracket}
		For every $r \geq 0$,
		\begin{equation}\label{eq:cheb_bracket}
			V_r - \tfrac{\ws+m}{m}\,2^n\,\varepsilon_r
			\;\leq\; \vol K \;\leq\;
			V_r + \tfrac{\ws+m}{m}\,2^n\,\varepsilon_r .
		\end{equation}
	\end{proposition}
	
	\begin{proof}
		By Proposition~\ref{prop:relu}, $\vol K = \tfrac{\ws+m}{m}\int_B f(g(x))\,dx$
		and, by~\eqref{eq:Vk}, $V_r = \tfrac{\ws+m}{m}\int_B f_r^{\mathrm{ch}}(g(x))\,dx$. Since
		$g(x) \in [0,\bg]$ for $x \in B$, the uniform bound $|f - f_r^{\mathrm{ch}}| \leq
		\varepsilon_r$ on $[0,\bg]$ gives
		\[
		|\vol K - V_r| \;\leq\; \frac{\ws+m}{m}\int_B \bigl|f(g(x)) - f_r^{\mathrm{ch}}(g(x))\bigr|\, dx
		\;\leq\; \frac{\ws+m}{m}\,\varepsilon_r\int_B dx
		\;=\; \frac{\ws+m}{m}\,2^n\,\varepsilon_r ,
		\]
		which is~\eqref{eq:cheb_bracket}.
	\end{proof}
	
	\subsection{The \textsc{Cheb} algorithm}
	
	At relaxation order $r$:
	\begin{description}
		\item[$(i)$] compute the coefficients $c_0,\ldots,c_r$ of Theorem~\ref{thm:cheb_coefs}; 
		\item[$(ii)$] compute
		$y_0^{\mathrm{ch}}, \ldots, y_r^{\mathrm{ch}}$ by tensor Gauss-Legendre on $B$
		through the recurrence~\eqref{eq:cheb_rec};
		\item[$(iii)$] form $V_r$ by~\eqref{eq:Vk} and the certified bracket by
		Proposition~\ref{prop:cheb_bracket}, using a certified bound on the uniform-error $\varepsilon_r$ of (\ref{eq:eps_def}): either   the closed-form majorant of Corollary~\ref{cor:cheb_eps} or the exact value obtained by maximizing the
		piecewise-polynomial error $f - f_r^{\mathrm{ch}}$ over $[0,\bg]$.
	\end{description}

	\subsection{Algebraic convergence rate}\label{sec:rate_cheb}
	
	The two ingredients of the \textsc{Cheb} construction --- the certified bracket
	of Proposition~\ref{prop:cheb_bracket} and the sharp uniform error of
	Theorem~\ref{thm:cheb_rate} --- combine to give the convergence rate of the
	method, in closed form. Recall the volume estimate $V_r$ of~\eqref{eq:Vk}, the
	truncation error $\varepsilon_r$ of \eqref{eq:eps_def}, and set
	\begin{equation}\label{eq:cheb_vpm}
		v_r^{\mathrm{ch},\pm} \;:=\; V_r \;\pm\; \frac{\ws+m}{m}\,2^n\,\varepsilon_r ,
	\end{equation}
	the two endpoints of the \textsc{Cheb} bracket.
	
	\begin{theorem}[Algebraic convergence of \textsc{Cheb}]\label{thm:cheb_conv}
		For every $r \geq 0$ the bounds~\eqref{eq:cheb_vpm} satisfy
		$v_r^{\mathrm{ch},-} \leq \vol K \leq v_r^{\mathrm{ch},+}$, and their width
		decreases at the sharp algebraic rate
		\begin{equation}\label{eq:cheb_conv}
			v_r^{\mathrm{ch},+} - v_r^{\mathrm{ch},-}
			\;=\; \frac{2(\ws+m)}{m}\,2^n\,\varepsilon_r,
			\qquad
			\lim_{r\to\infty} r\,\bigl(v_r^{\mathrm{ch},+} - v_r^{\mathrm{ch},-}\bigr)
			\;=\; \frac{2^{\,n+1}(\ws+m)\sqrt{\bg-1}}{\pi\,m} .
		\end{equation}
		In particular the certified accuracy of the point estimate obeys
		\[
		|\,V_r - \vol K\,| \;\leq\; \frac{\ws+m}{m}\,2^n\,\varepsilon_r
		\;=\; O\left(\tfrac1r\right).
		\]
	\end{theorem}
	
	\begin{proof}
		By Proposition~\ref{prop:cheb_bracket} the endpoints~\eqref{eq:cheb_vpm}
		bracket the volume, $v_r^{\mathrm{ch},-} \leq \vol K \leq v_r^{\mathrm{ch},+}$,
		and by their definition the width is exactly twice the half-width
		$\tfrac{\ws+m}{m}2^n\varepsilon_r$. Since $\vol K$ lies in the bracket and
		$V_r$ is its midpoint, $|V_r - \vol K|$ is bounded by the half-width.
		
		Theorem~\ref{thm:cheb_rate} provides the two-sided asymptotic
		$\lim_{r\to\infty} r\,\varepsilon_r = \sqrt{\bg-1}/\pi$; in particular
		$\varepsilon_r = \Theta(1/r)$, so the width $\tfrac{2(\ws+m)}{m}2^n\varepsilon_r$
		and the half-width are $\Theta(1/r)$ and $O(1/r)$ respectively. Multiplying the
		width by $r$ and passing to the limit,
		\[
		\lim_{r\to\infty} r\,\bigl(v_r^{\mathrm{ch},+} - v_r^{\mathrm{ch},-}\bigr)
		\;=\; \frac{2(\ws+m)}{m}\,2^n \lim_{r\to\infty} r\,\varepsilon_r
		\;=\; \frac{2(\ws+m)}{m}\,2^n\,\frac{\sqrt{\bg-1}}{\pi}
		\;=\; \frac{2^{\,n+1}(\ws+m)\sqrt{\bg-1}}{\pi\,m},
		\]
		which is~\eqref{eq:cheb_conv}.
	\end{proof}
	
	\begin{remark}[The rate is intrinsic to the ReLU]\label{rem:cheb_intrinsic}
		The exponent $-1$ is dictated solely by the first-order singularity of the ReLU $f$ at
		$t = 1$, see Remark~\ref{rem:gibbs}: no degree-$r$ polynomial surrogate of $f$ can
		beat the order $1/r$ on this corner, and the best one --- the Remez optimum ---
		improves only the constant, by the universal factor $2/(\pi\beta) \approx 2.27$
		of Remark~\ref{rem:bernstein}. The geometric data $n$, $w$, $m$ of the problem
		enter \eqref{eq:cheb_conv} only through  the multiplicative constant $2^{n+1}(\vert w\vert+m)/m$, never the
		exponent; the algebraic rate is a feature of approximating a non-smooth function
		by polynomials, not of the volume problem.
	\end{remark}

	\paragraph{Running example} 
	For $g = x_1^4 + x_2^2$ ($n = 2$, $\ws = 3$, $m = 4$, $\bg = 2$) the constant
	in~\eqref{eq:cheb_conv} is $2^{3}\cdot 7\cdot 1/(4\pi) = 14/\pi$, so the bracket
	width is $\sim 14/(\pi r) \approx 4.46/r$ and the point estimate satisfies
	$|V_r - \vol K| \lesssim 7/(\pi r) \approx 2.23/r$. At $r = 100$ this is about
	$2\times 10^{-2}$, i.e. two correct digits.

	\section{Generalized eigenvalue bracketing (\textsc{Gevp})}\label{sec:gevp}
	
	By Proposition~\ref{prop:nu_restriction}, the normalized volume
	$2^{-n}\vol K = \nu([0,1])$ is the unique scalar $\alpha$ for which the affine
	moment sequence $(y_k - \alpha\, z_k)_{k\in\N}$
	is the moment sequence of a non-negative measure
	$\hat\nu = \nu - \alpha\,\nu_w$ supported on $[1,\bg]$. The \textsc{Gevp}
	algorithm truncates the moment-cone membership of $\hat\nu$ to order $r$, which
	yields two linear matrix inequalities (LMIs) in the single unknown $\alpha$.
	Because the two  pencils entering these LMIs are sign-definite
	(Lemma~\ref{lem:hmu_pos}), each inequality collapses to a scalar bound on
	$\alpha$, and the extremal feasible values are smallest \emph{generalized
		eigenvalues} of explicit Hankel pencils. Rescaled by $2^n$ to bracket the
	volume itself, these deliver two sequences computed without any semidefinite
	solver --- a single symmetric eigenvalue routine per bound --- which converge to
	$\vol K$ monotonically from above and below. 
	\subsection{Two LMIs}\label{sec:gevp_lmis}
	
	For symmetric $A, B$ of equal size with $B \succ 0$, let
	$\lmin(A,B)$ denote the minimal generalized eigenvalue, i.e.\ the smallest real number
	$\lambda$ such that $\det(A - \lambda B) = 0$; equivalently
	$\lmin(A,B) = \lmin(B^{-1/2} A B^{-1/2})$.
	
	The support of a positive measure in the interval $[1,\bg]$ is encoded by the
	localizing polynomial
	\begin{equation}\label{eq:loc}
		h(t) \;:=\; (t-1)(\bg - t) \;=\; -t^2 + (\bg+1)\,t - \bg,
	\end{equation}
	which is $\geq 0$ on $[1,\bg]$ and $< 0$ on $(0,1)$. The truncated Hausdorff
	conditions for $\hat\nu$ to be a non-negative measure on $[1,\bg]$ at order $r$
	are $M_r(\hat\nu) \succeq 0$ and $M_{r-1}(h\hat\nu) \succeq 0$; substituting
	$\hat\nu = \nu - \alpha\,\nu_w$ and using linearity of the moment and localizing
	maps gives the two necessary LMIs
	\begin{align}
		M_r(\nu) - \alpha\, M_r(\nu_w) &\;\succeq\; 0, \label{eq:upper_lmi}\\
		M_{r-1}(h\,\nu) - \alpha\, M_{r-1}(h\,\nu_w) &\;\succeq\; 0. \label{eq:lower_lmi}
	\end{align}
	The entries of all four matrices are linear in the moment data: $M_r(\nu)_{ij} =
	y_{i+j}$ and $M_r(\nu_w)_{ij} = z_{i+j}$ from~\eqref{eq:y_moments}
	and~\eqref{eq:stokes_moment}, while, writing $h(t) = h_0 + h_1 t + h_2 t^2$ with
	$h_0 = -\bg$, $h_1 = \bg+1$, $h_2 = -1$,
	\[
	M_{r-1}(h\nu)_{ij} = \sum_{k=0}^{2} h_k\, y_{i+j+k},
	\qquad
	M_{r-1}(h\nu_w)_{ij} = \sum_{k=0}^{2} h_k\, z_{i+j+k},
	\]
	so that both localizing matrices use moments up to order $2r$, the same as the
	moment matrices.
	
	\begin{lemma}[Positivity of the reference pencils]\label{lem:hmu_pos}
		For every $r \geq 1$,
		\[
		M_r(\nu_w) \;\succ\; 0
		\qquad\text{and}\qquad
		-\,M_{r-1}(h\,\nu_w) \;\succ\; 0 .
		\]
	\end{lemma}
	
	\begin{proof}
		Let $p \in \R[t]_r$, $p \neq 0$, with coefficient vector $\mathbf p$. By
		definition of the moment matrix,
		\[
		\langle \mathbf p,\, M_r(\nu_w)\,\mathbf p\rangle
		\;=\; \int_0^1 p(t)^2\, d\nu_w(t)
		\;=\; \int_0^1 p(t)^2\, \tfrac{\ws}{m}\, t^{\tfrac{\ws}{m} - 1}\, dt .
		\]
		The integrand is the product of $p(t)^2 \geq 0$ and the density
		$\tfrac{\ws}{m} t^{\tfrac{\ws}{m}-1} > 0$ on $(0,1]$; it is non-negative and vanishes
		only at the (at most $r$) zeros of $p$, hence the integral is strictly
		positive. Thus $M_r(\nu_w) \succ 0$.
		
		For $p \in \R[t]_{r-1}$, $p \neq 0$, using $-h(t) = (1-t)(\bg - t)$,
		\[
		\langle \mathbf p,\, -M_{r-1}(h\nu_w)\,\mathbf p\rangle
		\;=\; \int_0^1 (1-t)(\bg - t)\, p(t)^2\, \tfrac{\ws}{m}\, t^{\tfrac{\ws}{m}-1}\, dt .
		\]
		On $[0,1]$ one has $1 - t \geq 0$ and $\bg - t \geq \bg - 1 > 0$, and the
		density is positive on $(0,1]$; the integrand is therefore non-negative and not
		identically zero (a non-zero $p$ of degree $\leq r-1$ vanishes at finitely many
		points, and $1-t$ vanishes only at $t = 1$). Hence the integral is strictly
		positive, i.e.\ $-M_{r-1}(h\nu_w) \succ 0$.
	\end{proof}
	
	\subsection{Two bounds as smallest generalized eigenvalues}\label{sec:gevp_bounds}
	
	We record the elementary equivalence underlying both bounds: for symmetric
	$X$ and $C \succ 0$,
	\begin{equation}\label{eq:gevp_equiv}
		X - \lambda C \succeq 0 \quad\Longleftrightarrow\quad \lambda \leq \lmin(X,C),
	\end{equation}
	which follows by congruence with $C^{-1/2}$. Lemma~\ref{lem:hmu_pos} lets us
	apply it to each LMI.
	
	For~\eqref{eq:upper_lmi}, since $M_r(\nu_w) \succ 0$, \eqref{eq:gevp_equiv} gives
	$\alpha \leq \lmin\bigl(M_r(\nu), M_r(\nu_w)\bigr)$. For~\eqref{eq:lower_lmi},
	write $C := -M_{r-1}(h\nu_w) \succ 0$; the inequality reads
	$M_{r-1}(h\nu) + \alpha C \succeq 0$, i.e.\ $M_{r-1}(h\nu) - (-\alpha) C \succeq 0$,
	so by~\eqref{eq:gevp_equiv} $-\alpha \leq \lmin\bigl(M_{r-1}(h\nu), C\bigr)$,
	that is $\alpha \geq -\lmin\bigl(M_{r-1}(h\nu), -M_{r-1}(h\nu_w)\bigr)$. The scalar
	pinned by the two LMIs is the normalized volume $\alpha = 2^{-n}\vol K$
	(Proposition~\ref{prop:nu_restriction}); rescaling the extremal eigenvalues by
	$2^n$ so as to bracket the volume itself, we define
	\begin{equation}\label{eq:vr_pm}
		v_r^+ \;:=\; 2^n\,\lmin\bigl(M_r(\nu),\, M_r(\nu_w)\bigr),
		\qquad
		v_r^- \;:=\; -\,2^n\,\lmin\bigl(M_{r-1}(h\nu),\, -M_{r-1}(h\nu_w)\bigr),
	\end{equation}
	so that, for the volume candidate $2^n\alpha$, \eqref{eq:upper_lmi} holds exactly
	for $2^n\alpha \leq v_r^+$ and~\eqref{eq:lower_lmi} exactly for
	$2^n\alpha \geq v_r^-$.
	
	\begin{proposition}[Two-sided bracket]\label{prop:bracket}
		For every $r \geq 1$,
		\[
		v_r^- \;\leq\; \vol K \;\leq\; v_r^+ .
		\]
	\end{proposition}
	
	\begin{proof}
		By Proposition~\ref{prop:nu_restriction}, $\hat\nu = \nu - 2^{-n}\vol K\,\nu_w$
		is a non-negative measure with $\supp \hat\nu \subseteq [1,\bg]$. Consequently
		$M_r(\hat\nu) \succeq 0$ (for any $p$, $\langle \mathbf p, M_r(\hat\nu)\mathbf p\rangle
		= \int p^2\, d\hat\nu \geq 0$), and, since $h \geq 0$ on $[1,\bg] \supseteq
		\supp\hat\nu$, also $M_{r-1}(h\hat\nu) \succeq 0$ (because $\int h\,p^2\, d\hat\nu
		\geq 0$). Thus $\alpha = 2^{-n}\vol K$ satisfies both~\eqref{eq:upper_lmi}
		and~\eqref{eq:lower_lmi}; by the equivalences above,
		$2^{-n}\vol K \leq \lmin(M_r(\nu),M_r(\nu_w)) = 2^{-n}v_r^+$ and
		$2^{-n}\vol K \geq 2^{-n}v_r^-$, i.e.\ $\vol K \leq v_r^+$ and $\vol K \geq v_r^-$.
	\end{proof}
	
	The upper bound $v_r^+$ is the Stokes-augmented moment-SOS upper bound of
	\cite[Thm.~3.3]{lasserre2019}, here generalized to quasi-homogeneous $g$ through
	the reference measure $\nu_w$; the lower bound $v_r^-$ is its localizing
	counterpart.
	
	The convergence proof rests on a standard fact, which we isolate.
	
	\begin{lemma}[Positive-semidefinite moments force a positive measure]\label{lem:moment_pos}
		Let $\sigma$ be a finite signed Borel measure with compact support in $\R$ such
		that $M_r(\sigma) \succeq 0$ for every $r \geq 0$. Then $\sigma \geq 0$.
	\end{lemma}
	
	\begin{proof}
		$M_r(\sigma) \succeq 0$ for all $r$ is equivalent to
		$\int p^2\, d\sigma \geq 0$ for every $p \in \R[t]$. Let $S := \supp\sigma$,
		compact, and let $f \in C(S)$ with $f \geq 0$. For $\varepsilon > 0$ the
		function $\sqrt{f + \varepsilon}$ is continuous on $S$, so by the Weierstrass
		theorem there is a polynomial $p$ with $\sup_S |p - \sqrt{f+\varepsilon}|$ as
		small as desired; then $p^2 \to f + \varepsilon$ uniformly on $S$, whence
		$\int p^2\, d\sigma \to \int (f+\varepsilon)\, d\sigma$ and therefore
		$\int (f+\varepsilon)\, d\sigma \geq 0$. Letting $\varepsilon \to 0$ (and using
		$|\sigma|(S) < \infty$) gives $\int f\, d\sigma \geq 0$. As this holds for every
		non-negative $f \in C(S)$, $\sigma \geq 0$.
	\end{proof}
	
	\begin{theorem}[Monotone convergence]\label{thm:monotone}
		The sequences $(v_r^+)_{r \geq 0}$ and $(v_r^-)_{r \geq 1}$ are respectively
		non-increasing and non-decreasing, and both converge to $\vol K$:
		\[
		v_r^- \;\uparrow\; \vol K \qquad\text{and}\qquad v_r^+ \;\downarrow\; \vol K,
		\qquad r \to \infty.
		\]
	\end{theorem}
	
	\begin{proof}
		\emph{Monotonicity.} If $\alpha$ satisfies the order $(r+1)$
		inequality~\eqref{eq:upper_lmi}, then $M_{r+1}(\nu) - \alpha M_{r+1}(\nu_w)
		\succeq 0$; deleting its last row and column (a principal submatrix of a
		positive-semidefinite matrix is positive-semidefinite) gives $M_r(\nu) - \alpha
		M_r(\nu_w) \succeq 0$. Hence the feasible set shrinks with $r$ and $v_{r+1}^+
		\leq v_r^+$. The identical argument applied to the localizing
		matrices~\eqref{eq:lower_lmi} gives $v_r^- \leq v_{r+1}^-$.
		
		\emph{Limits exist and bracket the volume.} By Proposition~\ref{prop:bracket},
		$v_r^- \leq \vol K \leq v_r^+$ for all $r$. The monotone bounded sequences
		therefore converge, say $v_r^+ \downarrow L^+ \geq \vol K$ and
		$v_r^- \uparrow L^- \leq \vol K$.
		
		\emph{$L^+ = \vol K$.} Since $L^+ \leq v_r^+$ for every $r$, i.e.\
		$2^{-n}L^+ \leq \lmin(M_r(\nu),M_r(\nu_w))$, the
		equivalence~\eqref{eq:gevp_equiv} gives $M_r(\nu - 2^{-n}L^+\nu_w) \succeq 0$ for
		every $r$. The signed measure $\nu - 2^{-n}L^+\nu_w$ has compact support in
		$[0,\bg]$, so Lemma~\ref{lem:moment_pos} yields $\nu - 2^{-n}L^+\nu_w \geq 0$.
		Restricting to $(0,1)$ and using $\nu|_{[0,1]} = 2^{-n}\vol K\,\nu_w$
		(Proposition~\ref{prop:nu_restriction}),
		\[
		(\nu - 2^{-n}L^+\nu_w)\big|_{(0,1)} \;=\; 2^{-n}(\vol K - L^+)\,\nu_w\big|_{(0,1)} \;\geq\; 0 ,
		\]
		and since $\nu_w > 0$ on $(0,1)$ this forces $\vol K - L^+ \geq 0$, i.e.\
		$L^+ \leq \vol K$. With $L^+ \geq \vol K$ we conclude $L^+ = \vol K$.
		
		\emph{$L^- = \vol K$.} Since $v_r^- \uparrow L^-$, we have $L^- \geq v_r^-$ for
		every $r$, i.e.\ $-2^{-n}L^- \leq \lmin(M_{r-1}(h\nu),-M_{r-1}(h\nu_w))$, so
		$\alpha = 2^{-n}L^-$ satisfies~\eqref{eq:lower_lmi} at every order, i.e.\
		$M_{r-1}\bigl(h(\nu - 2^{-n}L^-\nu_w)\bigr) \succeq 0$ for all $r$. The signed measure
		$h\,(\nu - 2^{-n}L^-\nu_w)$ has compact support in $[0,\bg]$, so
		Lemma~\ref{lem:moment_pos} gives $h\,(\nu - 2^{-n}L^-\nu_w) \geq 0$. Restricting to
		$(0,1)$, where $\nu - 2^{-n}L^-\nu_w = 2^{-n}(\vol K - L^-)\nu_w$ and $h < 0$,
		\[
		h\,(\nu - 2^{-n}L^-\nu_w)\big|_{(0,1)} \;=\; 2^{-n}(\vol K - L^-)\, h\,\nu_w\big|_{(0,1)} \;\geq\; 0 ,
		\]
		and since $h < 0$ and $\nu_w > 0$ there, this forces $\vol K - L^- \leq 0$, i.e.\
		$L^- \geq \vol K$. With $L^- \leq \vol K$ we conclude $L^- = \vol K$.
	\end{proof}
	
	\begin{remark}[Matrix sizes]\label{rem:sizes}
		The lower bound $v_r^-$ uses moment data up to order $2r$ --- the same as
		$v_r^+$ --- but through localizing matrices of size $r \times r$ rather than the
		$(r+1)\times(r+1)$ moment matrices of the upper bound.
	\end{remark}
	
	\begin{remark}[No semidefinite solver in the loop]\label{rem:nosdp}
		By construction $v_r^+$ is $2^n$ times the largest $\alpha$ satisfying~\eqref{eq:upper_lmi}
		and $v_r^-$ is $2^n$ times the smallest $\alpha$ satisfying~\eqref{eq:lower_lmi}. Both extremal
		values are $2^n$ times generalized eigenvalues of explicit Hankel pencils whose entries
		depend linearly on the computed data $y_0, \ldots, y_{2r}$ and on the fixed
		rational data $z_0, \ldots, z_{2r}$. No semidefinite program is solved: a single
		symmetric generalized-eigenvalue routine of size $O(r)$ delivers each bound.
	\end{remark}
	
	\subsection{The \textsc{Gevp} algorithm}
	
	At relaxation order $r$:
	\begin{description}
		\item[$(i)$] compute the moments $y_0, \ldots, y_{2r}$
		of~\eqref{eq:y_moments} in closed form; 
		\item[$(ii)$] compute $z_0, \ldots, z_{2r}$ in closed form via $z_j = \ws/(\ws + jm)$;
		\item[$(iii)$]  compute $v_r^+$ and $v_r^-$ via~\eqref{eq:vr_pm}, as $2^n$ times the
		smallest generalized eigenvalues of the two pencils, each by a symmetric eigensolve. 
	\end{description}
	\subsection{Exponential convergence rate}\label{sec:rate_gevp}
	
	Theorem~\ref{thm:monotone} establishes convergence of the bracket but no rate. We
	now show that both bounds converge to $\vol K$ \emph{exponentially}, with an
	explicit base, and that the base is a conformal invariant of the pair of
	intervals $[0,1]$ and $[1,\bg]$. The argument is a Bernstein extremal estimate:
	each gap is the best contrast a degree-$r$ polynomial achieves between its size
	on the support $[1,\bg]$ of $\hat\nu$ and its size on the support $[0,1]$ of
	$\nu_w$, and the asymptotics of such a contrast are governed by the Green
	function of the complement of $[1,\bg]$.
	
	\subsubsection*{The gap as an extremal eigenvalue ratio}
	
	Write $\hat\nu := \nu - 2^{-n}\vol K\,\nu_w$, the non-negative measure on $[1,\bg]$ of
	Proposition~\ref{prop:nu_restriction}, and recall the localizer
	$h(t) = (t-1)(\bg-t)$ of~\eqref{eq:loc}.
	
	\begin{lemma}[Rayleigh-quotient form of the gaps]\label{lem:rayleigh}
		For every $r \geq 1$,
		\begin{align}
			v_r^+ - \vol K
			&= 2^n\min_{p \in \R[t]_r \setminus\{0\}}
			\frac{\displaystyle\int_1^{\bg} p(t)^2\, d\hat\nu(t)}
			{\displaystyle\int_0^1 p(t)^2\, d\nu_w(t)},
			\label{eq:gap_upper_ray}\\[2pt]
			\vol K - v_r^-
			&= 2^n\min_{q \in \R[t]_{r-1} \setminus\{0\}}
			\frac{\displaystyle\int_1^{\bg} h(t)\, q(t)^2\, d\hat\nu(t)}
			{\displaystyle\int_0^1 \bigl(-h(t)\bigr) q(t)^2\, d\nu_w(t)}.
			\label{eq:gap_lower_ray}
		\end{align}
	\end{lemma}
	
	\begin{proof}
		Since $M_r(\nu_w) \succ 0$ (Lemma~\ref{lem:hmu_pos}), the variational
		characterization~\eqref{eq:gevp_equiv} of the smallest generalized eigenvalue
		reads $2^{-n}v_r^+ = \lmin(M_r(\nu),M_r(\nu_w)) =
		\min_{p\neq 0}\langle \mathbf p, M_r(\nu)\mathbf p\rangle/
		\langle \mathbf p, M_r(\nu_w)\mathbf p\rangle$, identifying a coefficient vector
		$\mathbf p = (p_0,\ldots,p_r)$ with $p(t) = \sum_k p_k t^k \in \R[t]_r$. For any
		measure $\sigma$ one has $\langle \mathbf p, M_r(\sigma)\mathbf p\rangle =
		\sum_{i,j} p_i p_j \int t^{i+j}\,d\sigma = \int p(t)^2\, d\sigma(t)$, so
		$2^{-n}v_r^+ = \min_p \int p^2\,d\nu / \int p^2\,d\nu_w$. Substituting
		$\nu = \hat\nu + 2^{-n}\vol K\,\nu_w$ in the numerator and dividing by
		$\int p^2\,d\nu_w > 0$,
		\[
		\frac{\int p^2\, d\nu}{\int p^2\, d\nu_w}
		\;=\; 2^{-n}\vol K + \frac{\int p^2\, d\hat\nu}{\int p^2\, d\nu_w},
		\]
		and pulling the constant out of the minimum gives
		$2^{-n}(v_r^+ - \vol K) = \min_p \int p^2\,d\hat\nu / \int p^2\,d\nu_w$, i.e.\
		$v_r^+ - \vol K = 2^n\min_p \int p^2\,d\hat\nu / \int p^2\,d\nu_w$. Restricting each
		integral to the support of its measure ($\supp\hat\nu \subseteq [1,\bg]$,
		$\supp\nu_w = [0,1]$) yields~\eqref{eq:gap_upper_ray}.
		
		The same three steps applied to the pencil $(M_{r-1}(h\nu), -M_{r-1}(h\nu_w))$
		of~\eqref{eq:vr_pm} --- using $\langle \mathbf q, M_{r-1}(h\sigma)\mathbf q\rangle
		= \int h\, q^2\, d\sigma$, together with $h \geq 0$ on $[1,\bg] \supseteq
		\supp\hat\nu$ and $-h \geq 0$ on $[0,1] = \supp\nu_w$ --- give the lower bound.
		Indeed $-2^{-n}v_r^- = \lmin(M_{r-1}(h\nu), -M_{r-1}(h\nu_w)) = \min_q \int h\,q^2\,d\nu
		/ \int(-h)q^2\,d\nu_w$; substituting $\nu = \hat\nu + 2^{-n}\vol K\,\nu_w$ in the
		numerator gives $\int h\,q^2\,d\nu = \int h\,q^2\,d\hat\nu + 2^{-n}\vol K \int h\,q^2\,d\nu_w
		= \int h\,q^2\,d\hat\nu - 2^{-n}\vol K \int(-h)q^2\,d\nu_w$, so
		$\int h\,q^2\,d\nu/\int(-h)q^2\,d\nu_w = -2^{-n}\vol K + \int h\,q^2\,d\hat\nu/\int(-h)q^2\,d\nu_w$,
		and $-2^{-n}v_r^- = -2^{-n}\vol K + \min_q(\cdots)$ rearranges to
		$\vol K - v_r^- = 2^n\min_q(\cdots)$, i.e.~\eqref{eq:gap_lower_ray}.
	\end{proof}
	
	Both ratios measure how small a degree-$r$ polynomial can be made on $[1,\bg]$
	relative to its size on $[0,1]$. This is the prototypical Bernstein extremal
	problem, whose rate is set by the conformal geometry of $[1,\bg]$.
	
	\subsubsection*{The Bernstein-Walsh lemma and the Green function of $[1,\bg]$}
	
	Let $E \subset \C$ be compact with connected complement
	$\Omega_E := (\C \cup \{\infty\}) \setminus E$. The \emph{Green function} of
	$\Omega_E$ with pole at infinity is the unique non-negative $G_E : \Omega_E \to \R$
	that is harmonic on $\C \setminus E$, vanishes quasi-everywhere on $\partial E$,
	and satisfies $G_E(t) = \log|t| - \log\capac E + o(1)$ as $t \to \infty$, where
	$\capac E$ is the logarithmic capacity~\cite[Sec.~I.4]{safftotik}.
	
	\begin{lemma}[Bernstein-Walsh lemma]\label{thm:bw}
		With $E$ as above, let $p$ be a polynomial of degree $r$. Then for every
		$t \notin E$,
		\begin{equation}\label{eq:bw}
			|p(t)| \;\leq\; \exp\bigl(r\,G_E(t)\bigr)\,\max_{s \in E}|p(s)| .
		\end{equation}
	\end{lemma}
	
	See \cite[Thm.~5.5.7]{ransford} or \cite[Thm.~III.2.1]{safftotik}. The Green
	function of the unit interval is
	$G_{[-1,1]}(x) = \log\bigl|x + \sqrt{x^2-1}\bigr|$, which for real $|x| \geq 1$
	equals $\arccosh|x|$; see \cite[Example~I.3.5]{safftotik} or, in the equivalent
	language of the Bernstein ellipse, \cite[Ch.~12, Eq.~(12.13)]{trefethen2013}. The
	affine map
	\begin{equation}\label{eq:affine_u}
		u(t) \;:=\; \frac{2t - 1 - \bg}{\bg - 1}
	\end{equation}
	sends $[1,\bg]$ onto $[-1,1]$ and fixes $\infty$, so by conformal invariance of
	the Green function~\cite[Thm.~II.4.9]{safftotik},
	$G_{[1,\bg]}(t) = G_{[-1,1]}(u(t))$. For $t \in [0,1]$ one has $u(t) \leq -1$ and
	$|u(t)| = (1 + \bg - 2t)/(\bg-1) \geq 1$, hence
	\begin{equation}\label{eq:greenval}
		G_{[1,\bg]}(t) \;=\; \arccosh\!\left(\frac{1 + \bg - 2t}{\bg-1}\right),
		\qquad t \in [0,1],
	\end{equation}
	which decreases from $G_{[1,\bg]}(0) = \log\gamma$ to $G_{[1,\bg]}(1) = 0$, with
	\begin{equation}\label{eq:gamma_def}
		\gamma \;:=\; \frac{\sqrt{\bg} + 1}{\sqrt{\bg} - 1} \;>\; 1 .
	\end{equation}
	
	Figure~\ref{fig:green_function} plots $G_{[1,\bg]}$ on $[0,1]$ for several $\bg$.
	As $t \to 1^-$ the argument of the $\arccosh$ in~\eqref{eq:greenval} tends to $1$
	and $G_{[1,\bg]}(t) \sim 2\sqrt{(1-t)/(\bg-1)}$, so the Green function reaches the
	support $[1,\bg]$ with a vertical tangent and $|G_{[1,\bg]}'|$ is unbounded as
	$t \to 1^-$. Away from that endpoint, however --- on the interval $[0,\tfrac12]$,
	which is disjoint from $[1,\bg]$ --- $G_{[1,\bg]}$ is real-analytic with bounded
	derivative, so
	\begin{equation}\label{eq:C0_def}
		C_0 \;:=\; \max_{t \in [0,1/2]}\bigl|G_{[1,\bg]}'(t)\bigr| \;<\; \infty,
	\end{equation}
	and the mean value theorem yields the linear minorant
	$G_{[1,\bg]}(t) \ge \log\gamma - C_0\,t$ on $[0,\tfrac12]$ used in both proofs
	below. The value at the origin, $G_{[1,\bg]}(0) = \log\gamma$, controls the entire
	rate; smaller $\bg$ raises it, as the curves of Figure~\ref{fig:green_function}
	show.
	
	\begin{figure}[ht]
		\centering
		\includegraphics[width=0.8\textwidth]{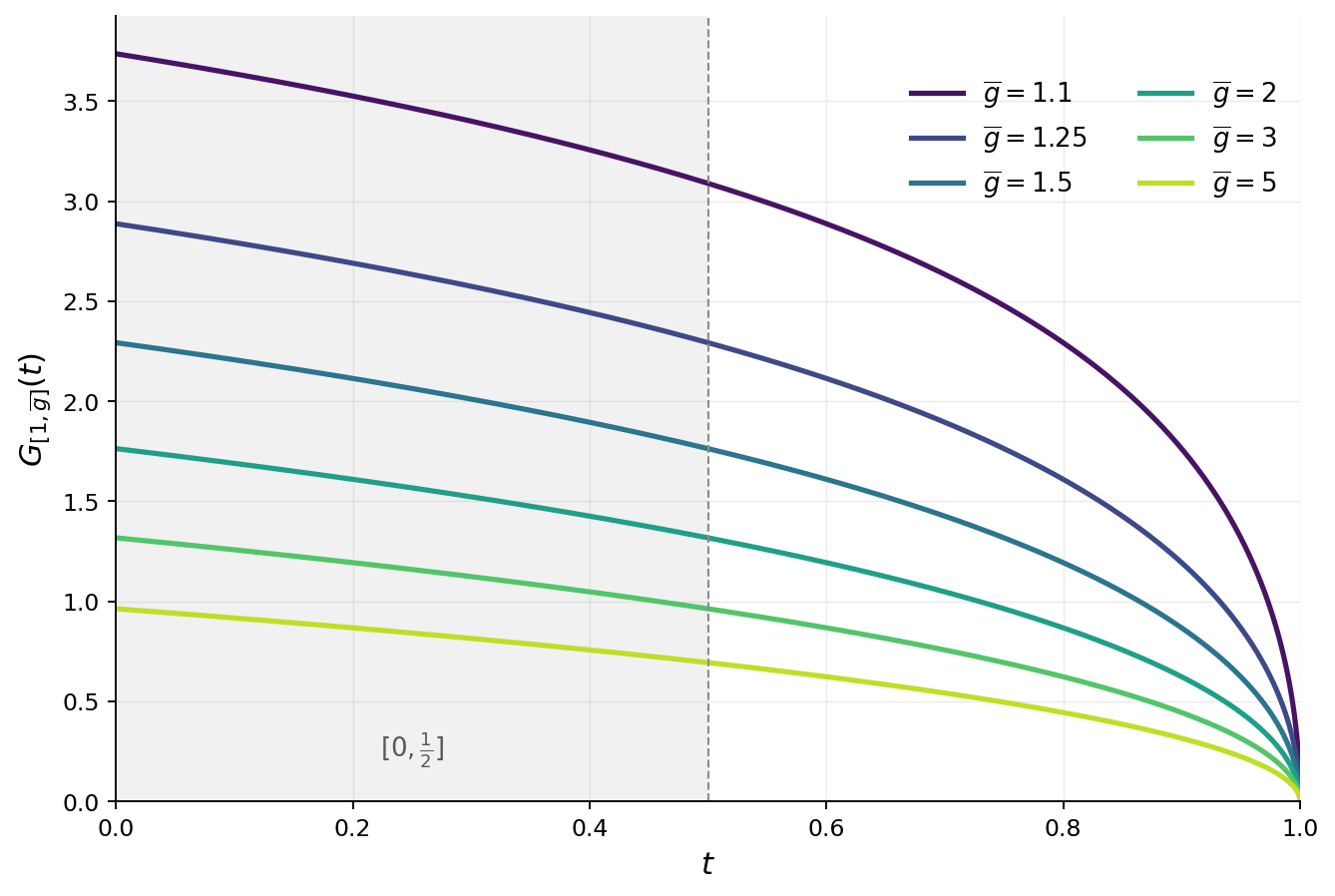}
		\caption{The Green function $G_{[1,\bg]}(t)$ of $\C\cup\{\infty\}\setminus[1,\bg]$
			restricted to $[0,1]$, for several values of $\bg>1$ (Eq.~\eqref{eq:greenval}).
			It vanishes at $t=1$ with a vertical tangent --- $|G_{[1,\bg]}'|$ is unbounded
			as $t\to1^-$ --- and at $t=0$ equals $\log\gamma$, with $\gamma$
			of~\eqref{eq:gamma_def} increasing as $\bg\downarrow1$. On the shaded interval
			$[0,\tfrac12]$, disjoint from $[1,\bg]$, the function is real-analytic with
			finite slope: this is the property underlying the constant
			$C_0=\max_{[0,1/2]}|G_{[1,\bg]}'|<\infty$ of~\eqref{eq:C0_def} in
			Propositions~\ref{prop:rate_upper}--\ref{prop:rate_lower}.}
		\label{fig:green_function}
	\end{figure}
	
	\subsubsection*{Geometric decay of the two gaps}
	
	\begin{proposition}[Upper gap]\label{prop:rate_upper}
		There is a constant $C^+ > 0$, depending on $n$, $m$, $\bg$ and $w$ but not on
		$r$, such that
		\begin{equation}\label{eq:rate_upper}
			v_r^+ - \vol K \;\leq\; C^+\, r^{\frac{\ws}{m}}\, \gamma^{-2r}, \qquad r \geq 1 .
		\end{equation}
	\end{proposition}
	
	\begin{proof}
		Let $T_r$ be the Chebyshev polynomial of the first kind and take, as test
		polynomial in~\eqref{eq:gap_upper_ray}, the shifted Chebyshev polynomial of
		$[1,\bg]$,
		\[
		p(t) := T_r\bigl(u(t)\bigr) \in \R[t]_r ,
		\]
		with $u$ as in~\eqref{eq:affine_u}. Since $u$ maps $[1,\bg]$ onto $[-1,1]$ and
		$|T_r| \leq 1$ there, $|p| \leq 1$ on $[1,\bg]$. Using $T_r(x) = (-1)^r
		\cosh(r\arccosh|x|)$ for $x \leq -1$ together with~\eqref{eq:greenval},
		\begin{equation}\label{eq:p_cosh}
			p(t) = (-1)^r \cosh\bigl(r\,G_{[1,\bg]}(t)\bigr), \qquad t \in [0,1].
		\end{equation}
		
		\emph{Numerator.} As $|p| \leq 1$ on $[1,\bg] \supseteq \supp\hat\nu$ and
		$\hat\nu|_{[1,\bg]} = \nu|_{[1,\bg]}$,
		\[
		\int_1^{\bg} p^2\, d\hat\nu \;\leq\; \hat\nu([1,\bg]) \;=\; \nu([1,\bg])
		\;\leq\; \nu([0,\bg]) \;=\; 1 .
		\]
		
		\emph{Denominator.} By~\eqref{eq:p_cosh} and $\cosh x \geq \tfrac12 e^x$,
		$p(t)^2 \geq \tfrac14 e^{2r G_{[1,\bg]}(t)}$ on $[0,1]$, so, with the density
		$\tfrac{\ws}{m} t^{\frac{\ws}{m}-1}$ of $\nu_w$,
		\[
		I_r := \int_0^1 p^2\, d\nu_w
		\;\geq\; \frac14\,\frac{\ws}{m}\int_0^1 e^{2r G_{[1,\bg]}(t)}\, t^{\frac{\ws}{m}-1}\, dt .
		\]
		On $[0,\tfrac12]$, which is disjoint from $[1,\bg]$, the Green function is
		real-analytic; with $C_0$ of~\eqref{eq:C0_def}, the mean value theorem and
		$G_{[1,\bg]}(0) = \log\gamma$ give $G_{[1,\bg]}(t) \geq \log\gamma - C_0 t$ on
		$[0,\tfrac12]$, so restricting $I_r$ to $[0,\tfrac12]$,
		\[
		I_r \;\geq\; \frac{\gamma^{2r}}{4}\,\frac{\ws}{m}
		\int_0^{1/2} e^{-2r C_0 t}\, t^{\frac{\ws}{m} - 1}\, dt
		\;=\; \frac{\gamma^{2r}}{4}\,\frac{\ws}{m}\,(2r C_0)^{-\frac{\ws}{m}}
		\int_0^{r C_0} e^{-s}\, s^{\frac{\ws}{m} - 1}\, ds ,
		\]
		by the substitution $s = 2 r C_0 t$. For $r \geq 1/C_0$ the last integral is at
		least $C_1 := \int_0^1 e^{-s} s^{\frac{\ws}{m}-1}\, ds > 0$, whence
		\[
		I_r \;\geq\; C_2\, \gamma^{2r}\, r^{-\frac{\ws}{m}},
		\qquad C_2 := \frac{\ws\, C_1}{4 m\,(2 C_0)^{\frac{\ws}{m}}} > 0 .
		\]
		Inserting the test polynomial into~\eqref{eq:gap_upper_ray},
		\[
		v_r^+ - \vol K \;\leq\;
		2^n\,\frac{\int_1^{\bg} p^2\, d\hat\nu}{\int_0^1 p^2\, d\nu_w}
		\;\leq\; \frac{2^n}{I_r}
		\;\leq\; \frac{2^n}{C_2}\, r^{\frac{\ws}{m}}\, \gamma^{-2r}
		\qquad (r \geq 1/C_0).
		\]
		The finitely many degrees $1 \leq r < 1/C_0$ each give a finite left-hand side
		and a strictly positive right-hand side, so~\eqref{eq:rate_upper} holds at each
		of them for a suitable constant; enlarging $C^+$ to the maximum of $2^n/C_2$ and
		these finitely many values gives a single $r$-independent $C^+$.
	\end{proof}
	
	\begin{proposition}[Lower gap]\label{prop:rate_lower}
		There is a constant $C^- > 0$, depending on $n$, $m$, $\bg$ and $w$ but not on
		$r$, such that
		\begin{equation}\label{eq:rate_lower}
			\vol K - v_r^- \;\leq\; C^-\, r^{\frac{\ws}{m}}\, \gamma^{-2r}, \qquad r \geq 1 .
		\end{equation}
	\end{proposition}
	
	\begin{proof}
		Take in~\eqref{eq:gap_lower_ray} the degree-$(r-1)$ shifted Chebyshev polynomial
		$q(t) := T_{r-1}(u(t)) \in \R[t]_{r-1}$, so that $|q| \leq 1$ on $[1,\bg]$ and,
		as in~\eqref{eq:p_cosh},
		$q(t) = (-1)^{r-1}\cosh\bigl((r-1)G_{[1,\bg]}(t)\bigr)$ for $t \in [0,1]$.
		
		\emph{Numerator (above).} On $[1,\bg]$, $0 \leq h(t) \leq \bar h :=
		\max_{[1,\bg]} h = \bigl(\tfrac{\bg-1}{2}\bigr)^2$ (attained at $t = \tfrac{1+\bg}{2}$),
		and $|q| \leq 1$, so
		\[
		\int_1^{\bg} h\, q^2\, d\hat\nu \;\leq\; \bar h\,\hat\nu([1,\bg]) \;\leq\; \bar h .
		\]
		
		\emph{Denominator (below).} On $[0,\tfrac12]$ two bounds hold: first
		$-h(t) = (1-t)(\bg-t) \geq \tfrac12\bigl(\bg-\tfrac12\bigr) =: h^- > 0$, since
		$1-t \geq \tfrac12$ and $\bg - t \geq \bg - \tfrac12$; second
		$q(t)^2 \geq \tfrac14 e^{2(r-1)G_{[1,\bg]}(t)}$. Hence, with the $\nu_w$-density,
		\[
		J_r := \int_0^1 (-h)\, q^2\, d\nu_w
		\;\geq\; \frac{h^-}{4}\,\frac{\ws}{m}
		\int_0^{1/2} e^{2(r-1) G_{[1,\bg]}(t)}\, t^{\frac{\ws}{m}-1}\, dt .
		\]
		Using $G_{[1,\bg]}(t) \geq \log\gamma - C_0 t$ on $[0,\tfrac12]$ as above and,
		for $r \geq 2$, the substitution $s = 2(r-1)C_0 t$,
		\[
		J_r \;\geq\; \frac{h^-}{4}\,\frac{\ws}{m}\,\gamma^{2(r-1)}
		\bigl(2(r-1)C_0\bigr)^{-\frac{\ws}{m}}\int_0^{(r-1)C_0} e^{-s} s^{\frac{\ws}{m}-1}\, ds
		\;\geq\; C_2'\,\gamma^{2(r-1)}\,(r-1)^{-\frac{\ws}{m}}
		\]
		once $r - 1 \geq 1/C_0$, with $C_2' := \tfrac{h^- \ws\, C_1}{4 m\,(2 C_0)^{\frac{\ws}{m}}} > 0$.
		Since $r - 1 < r$ gives $(r-1)^{-\frac{\ws}{m}} \geq r^{-\frac{\ws}{m}}$, we have
		$J_r \geq C_2'\,\gamma^{2(r-1)} r^{-\frac{\ws}{m}}$.
		
		\emph{Conclusion.} Inserting these bounds into~\eqref{eq:gap_lower_ray} and using
		$\gamma^{2(r-1)} = \gamma^{-2}\gamma^{2r}$,
		\[
		\vol K - v_r^- \;\leq\; 2^n\,\frac{\int_1^{\bg} h\, q^2\, d\hat\nu}{\int_0^1 (-h)\,q^2\,d\nu_w}
		\;\leq\; \frac{2^n\,\bar h}{J_r}
		\;\leq\; \frac{\bar h\, 2^n\,\gamma^2}{C_2'}\, r^{\frac{\ws}{m}}\, \gamma^{-2r}
		\qquad (r \geq \max(2, 1 + 1/C_0)),
		\]
		which is~\eqref{eq:rate_lower} with $C^- := \bar h\, 2^n \gamma^2/C_2'$; the
		finitely many smaller degrees are absorbed into $C^-$ as in
		Proposition~\ref{prop:rate_upper}.
	\end{proof}
	
	\begin{theorem}[Exponential convergence of \textsc{Gevp}]\label{thm:rate}
		There is a constant $C > 0$, depending on $n$, $m$, $\bg$ and $w$ but not on
		$r$, such that
		\begin{equation}\label{eq:rate_gap}
			v_r^+ - v_r^- \;\leq\; C\, r^{\frac{\ws}{m}}\, \gamma^{-2r}, \qquad r \geq 1 .
		\end{equation}
	\end{theorem}
	
	\begin{proof}
		Writing $v_r^+ - v_r^- = (v_r^+ - \vol K) + (\vol K - v_r^-)$, with both terms
		non-negative by Proposition~\ref{prop:bracket}, bound the first summand with
		Proposition~\ref{prop:rate_upper} and the second with
		Proposition~\ref{prop:rate_lower}, and set $C := C^+ + C^-$.
	\end{proof}
	
	\begin{figure}[ht]
		\centering
		\includegraphics[width=0.82\textwidth]{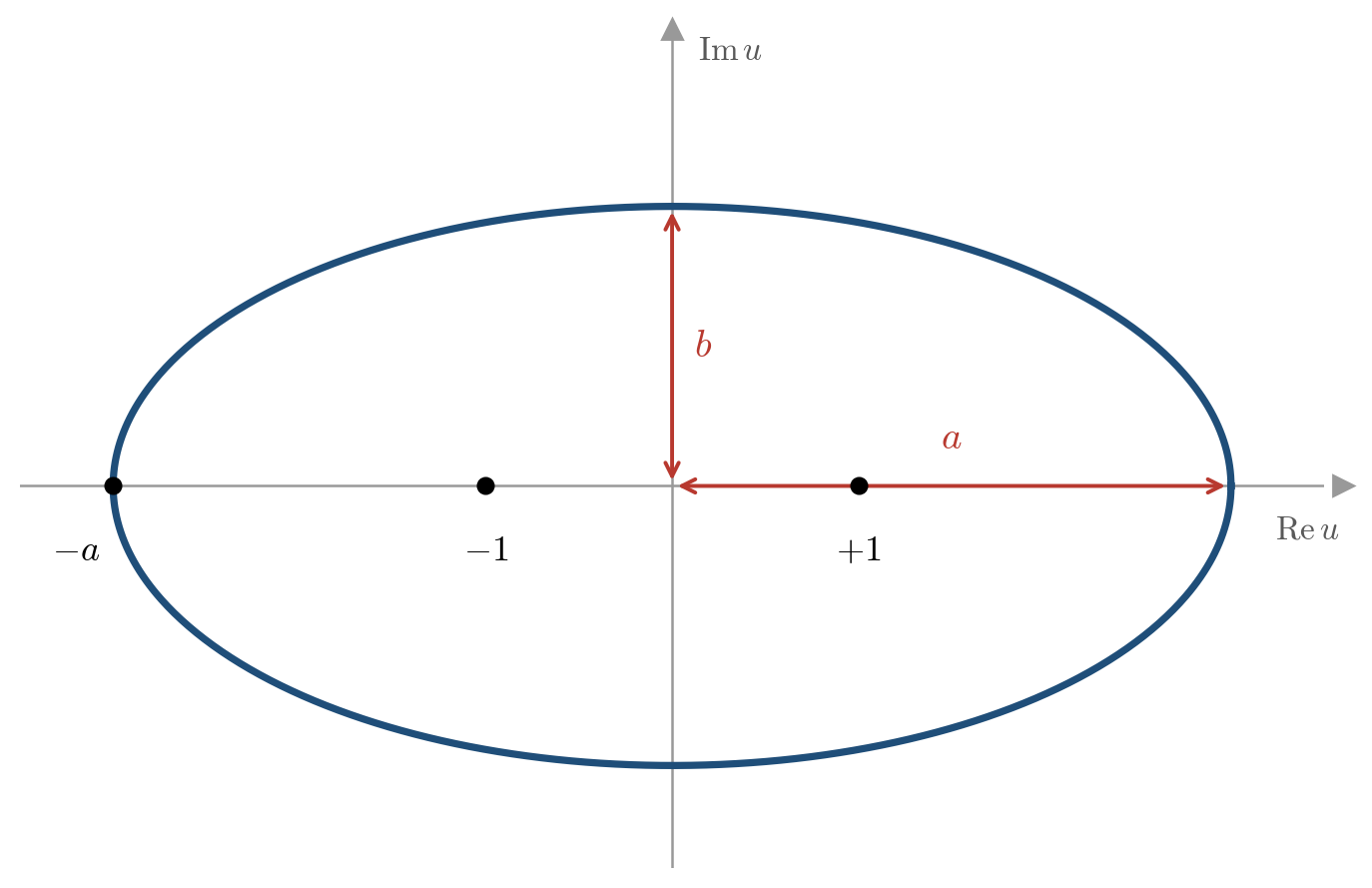}
		\caption{The Bernstein ellipse with foci $\pm1$ for the running
			example $\bg=2$, in the normalized variable $u=(2t-1-\bg)/(\bg-1)$ that sends
			$[1,\bg]$ to $[-1,1]$ (Remark~\ref{rem:gamma_geom}). The endpoints of the
			support $[1,\bg]$ of $\hat\nu$ are the foci $\pm1$, and the origin $t=0$ maps to
			the left vertex $u(0)=-a$. The ellipse has semi-axes $a=\cosh\log\gamma$ and
			$b=\sinh\log\gamma$, hence parameter
			$\gamma=a+b$; for $\bg=2$, $a=3$, $b=2\sqrt2$ and
			$\gamma=3+2\sqrt2\approx5.828$.}
		\label{fig:bernstein}
	\end{figure}
	
	\begin{remark}[Geometric meaning of $\gamma$]\label{rem:gamma_geom}
		The base $\gamma$ of~\eqref{eq:gamma_def} is a conformal invariant of the
		configuration, with the geometric reading of the \emph{Bernstein ellipse} in the complex plane of the normalized variable $u$. Recall that the map
		$u$ of~\eqref{eq:affine_u} normalizes the support $[1,\bg]$ of $\hat\nu$ to
		$[-1,1]$, placing its endpoints at the ellipse foci $\pm1$; the reference interval
		$[0,1]$ carrying $\nu_w$ maps to $[-a,-1]$, and the origin $t = 0$ --- the far
		end of that interval --- maps to
		\[
		u(0) = -\frac{\bg+1}{\bg-1} = -a, \qquad a = \cosh\log\gamma = \tfrac12(\gamma + \gamma^{-1}).
		\]
		Now $-a$ is exactly the left vertex of the ellipse
		with foci $\pm1$ and semi-axes $a = \cosh\log\gamma$ and $b = \sinh\log\gamma$;
		its parameter, the sum of the semi-axes, is
		\[
		a + b \;=\; e^{\log\gamma} \;=\; \gamma.
		\]
		Thus $\gamma$ is simultaneously (i) the exponential of the Green function value
		of $\C\setminus[1,\bg]$ at the origin, and (ii) the parameter of the Bernstein
		ellipse with foci at the normalized endpoints of $[1,\bg]$ that passes through
		the normalized origin. By the Bernstein-Walsh lemma~\eqref{eq:bw}, a degree-$r$
		polynomial bounded by $1$ on $[1,\bg]$ can be as large as $\gamma^r$ at $t = 0$,
		and the Chebyshev test polynomials~\eqref{eq:p_cosh} attain this growth; the
		\emph{square} of such a polynomial, weighed against $\nu_w$ near the origin,
		supplies the denominator $\sim \gamma^{2r}$ of the Rayleigh
		quotients~\eqref{eq:gap_upper_ray}--\eqref{eq:gap_lower_ray}, whence the rate
		$\gamma^{-2r}$. Smaller $\bg$ pushes $u(0)$ farther from $[-1,1]$, enlarging
		the ellipse and $\gamma$, hence accelerating convergence; as $\bg \downarrow 1$
		the ellipse and $\gamma$ blow up, while as $\bg \to \infty$ it collapses onto the
		segment $[-1,1]$ and $\gamma \downarrow 1$.
	\end{remark}
	
	\begin{remark}[Sharpness and the prefactor]\label{rem:rate_sharp}
		As will be seen in the numerical experiments, the exponential base $\gamma^{-2}$ in~\eqref{eq:rate_gap} is sharp. The
		factor $2$ in the exponent is the trace of the squared test polynomials in the
		Rayleigh quotients; the polynomial prefactor $r^{\,\ws/m}$ comes from the
		algebraic vanishing $t^{\ws/m-1}$ of the $\nu_w$-density at the origin
		(Remark~\ref{rem:singular_density}), and is the only place where the weight $w$
		enters the rate --- the base $\gamma$ depends on $\bg$ alone.
	\end{remark}
	
	\begin{remark}[The price of the geometric rate: working precision]\label{rem:rate_precision}
		The same conformal blow-up that produces the rate makes the eigenproblems
		ill-conditioned, and resolving a gap of size $\gamma^{-2r}$ therefore requires
		extended precision. Three effects contribute, additively in decimal digits.
		First, the dynamic range of the data: $y_k \le \bg^{\,k}$, so the entries of
		$M_r(\nu)$ span $2r\log_{10}\bg$ digits. Second, the conditioning of the reference
		pencil: $M_r(\nu_w)$ is the Hankel matrix of a Beta measure on $[0,1]$, of
		Hilbert type, whose condition number grows like $(1+\sqrt2)^{4r}$, i.e.\
		$2r\log_{10}(3+2\sqrt2) = 1.531\,r$ digits. Third, the scale of the eigenvalue
		itself, $\alpha = 2^{-n}\vol K$, which must be resolved against a matrix of unit
		scale. Hence the budget
		\begin{equation}\label{eq:precision_budget}
			P(r) \;=\; \bigl\lceil 2r\log_{10}\bg \bigr\rceil
			\;+\; \bigl\lceil 1.531\, r \bigr\rceil
			\;+\; \Bigl\lceil \log_{10}\frac{2^n}{\vol K} \Bigr\rceil \;+\; O(1)
		\end{equation}
		decimal digits, linear in $r$. Note that it is $\bg$, and not $\gamma$, that
		drives the first and dominant term: large $\bg$ penalizes \textsc{Gevp} twice,
		once through the rate $\gamma \downarrow 1$ and once through the precision --- hence
		the time --- required per order. In practice the maximum of the first two terms
		suffices, the two mechanisms rarely being simultaneously active, and this is what
		the implementation uses.
		
		This is the precise sense in which the geometric \textsc{Gevp} bracket is
		``precision-hungry'' where the algebraic \textsc{Cheb} bracket of
		Section~\ref{sec:rate_cheb} is stable in fixed precision. The two methods occupy
		opposite corners of a rate-robustness trade-off: \textsc{Cheb} converges at the
		algebraic rate $1/r$ but in ordinary double precision, while \textsc{Gevp}
		converges at the geometric rate $\gamma^{-2r}$ at a precision cost linear in $r$.
	\end{remark}
	
	\paragraph{Running example.}
	For $g = x_1^4 + x_2^2$ on $B = [-1,1]^2$, the pushforward $\nu$ is supported on $[0,2]$ and the reference measure is
	$\nu_w = \mathrm{Beta}(\tfrac34, 1)$, whose density is unbounded at the origin
	(Remark~\ref{rem:singular_density}). By Lemma~\ref{lem:hmu_pos} this singularity
	leaves both reference pencils definite, so each bound is delivered by a single
	symmetric eigensolve of size at most $r+1$. The certified bracket
	$v_r^- \le \vol K \le v_r^+$ (Proposition~\ref{prop:bracket}) tightens
	geometrically with the order $r$:
	\[
	\begin{array}{c c c}
		r & [\,v_r^-,\; v_r^+\,] & \text{gap } v_r^+ - v_r^- \\[2pt]
		\hline
		\rule{0pt}{2.6ex}%
		1  & [\,3.4222,\; 3.5254\,] & 1.0\times 10^{-1}\\
		2  & [\,3.49391,\; 3.49731\,] & 3.4\times 10^{-3}\\
		5  & [\,3.49607667,\; 3.49607679\,] & 1.3\times 10^{-7}\\
		10 & [\,3.4960767390562,\; 3.4960767390562\,] & 4.2\times 10^{-15}
	\end{array}
	\]
	Already the $2\times 2$ pencil at $r = 1$ localizes the volume to its first
	digit, and the bracket gains $2\log_{10}\gamma \approx 1.53$ correct digits per
	unit increase of $r$, where $\gamma = (\sqrt2+1)/(\sqrt2-1) = 3 + 2\sqrt2$ is the
	geometric rate established in Section~\ref{sec:rate_gevp}.
	
	\section{Computational complexity}\label{sec:complexity}
	
	The convergence rates of the two algorithms were settled in
	Sections~\ref{sec:rate_cheb} and~\ref{sec:rate_gevp}.
	This section addresses the orthogonal question of how much work a
	single relaxation order costs. We account for two distinct resources: the number
	of arithmetic operations, and the working precision (the number of digits)
	at which those operations must be carried out. The two methods differ as much in
	the second as in the first, and it is the precision axis that ultimately
	separates them.
	
	Both algorithms open with a common preprocessing stage --- estimate the maximal
	value $\bg = \max_B g$ of~\eqref{eq:bg}, then compute the moments of $g$ over the
	box --- and only afterwards diverge: \textsc{Cheb} contracts its moments against a
	vector of closed-form coefficients, while \textsc{Gevp} feeds its moments to two
	generalized eigenvalue problems.
	
	\subsection{Shared preprocessing}\label{sec:cost_shared}
	
	\subsubsection{Estimating the maximum value}
	
	The quantity $\bg = \max_{x \in B} g(x)$ is itself the optimal value of a
	polynomial optimization problem over the box, and would seem to demand a global
	optimizer. In fact the two brackets remain valid for any \emph{upper} estimate of
	$\bg$, so it need not be computed exactly.
	
	\begin{proposition}[Robustness to over-estimation of $\bg$]\label{prop:gbar_upper}
		Let $\bg' \geq \bg$. If the \textsc{Cheb} construction of
		Section~\ref{sec:cheb} and the \textsc{Gevp} construction of
		Section~\ref{sec:gevp} are run with $\bg'$ in place of $\bg$ throughout, the
		resulting brackets still satisfy
		$v_r^{\mathrm{ch},-} \leq \vol K \leq v_r^{\mathrm{ch},+}$ and
		$v_r^- \leq \vol K \leq v_r^+$ for every $r$.
	\end{proposition}
	
	\begin{proof}
		\emph{\textsc{Cheb}.} The ReLU identity~\eqref{eq:relu} and the integral
		$\int_B (1-g)_+\,dx$ do not involve $\bg$. Replacing $\bg$ by $\bg'$ only changes
		the interval $[0,\bg']$ on which $f(t) = (1-t)_+$ is expanded; since
		$g(x) \in [0,\bg] \subseteq [0,\bg']$ on $B$, the uniform bound
		$|f - f_r^{\mathrm{ch}}| \leq \varepsilon_r$ over $[0,\bg']$ still controls the
		integrand on the actual range of $g$, and Proposition~\ref{prop:cheb_bracket}
		holds verbatim. \emph{\textsc{Gevp}.} The pushforward $\nu$ is supported in
		$[0,\bg] \subseteq [0,\bg']$ and the residual $\hat\nu$ in
		$[1,\bg] \subseteq [1,\bg']$. The localizer $h'(t) = (t-1)(\bg'-t)$ is then
		$\geq 0$ on $[1,\bg'] \supseteq \supp\hat\nu$, while
		$-h'(t) = (1-t)(\bg'-t) \geq 0$ on $[0,1]$ because $\bg' - t \geq \bg'-1 > 0$
		there; hence Lemma~\ref{lem:hmu_pos} (positivity of the reference pencils) and
		Proposition~\ref{prop:bracket} (the two-sided bracket) carry over unchanged.
	\end{proof}
	
	By Proposition~\ref{prop:gbar_upper} a certified over-estimate is enough, and the
	cheapest one is immediate: on $B$ every monomial obeys $|x^\alpha| \leq 1$, so
	\begin{equation}\label{eq:gbar_ell1}
		\bg \;\leq\; \sum_\alpha |g_\alpha| \;=:\; \|g\|_1 ,
	\end{equation}
	the $\ell^1$-norm of the coefficient vector, computed in $O(s)$ operations, where
	\[
	s := \#\{\alpha : g_\alpha \neq 0\} \leq \binom{n+m-1}{m}
	\]
	is the number of monomials of $g$.
	For the running example $g = x_1^4 + x_2^2$ this returns $\bg \leq 2$, the exact
	value. A sharper certified bound at the same $O(s)$ cost follows from a single
	interval-arithmetic evaluation of $g$ on $B$; and if a near-optimal $\bg$ is
	wanted, to tighten the convergence constants of
	Sections~\ref{sec:rate_cheb}--\ref{sec:rate_gevp}, the moment-SOS hierarchy on
	the box supplies it through a finite sequence of semidefinite programs whose cost
	is independent of $r$ and therefore amortized over all orders.
	
	\subsubsection{Computing the box moments}
	
	Both algorithms collapse the $n$-dimensional volume onto univariate moments of
	$g$ over $B$, and this is the only stage whose cost grows with the ambient
	dimension. It rests on the fact that monomial moments over the box factorize,
	\begin{equation}\label{eq:box_factor}
		\int_B x^\alpha\, d\mu \;=\; \prod_{i=1}^n l_{\alpha_i},
		\qquad
		l_k := \frac12\int_{-1}^1 t^k\, dt =
		\begin{cases}
			\dfrac{1}{k+1}, & k \text{ even},\\[6pt]
			0, & k \text{ odd},
		\end{cases}
	\end{equation}
	so that any box moment is a product of $n$ one-dimensional moments.
	
	\textsc{Gevp} needs the power moments $y_k = \int_B g^k\, d\mu$ of~\eqref{eq:y_moments}
	for $0 \leq k \leq 2r$. Each is obtained by expanding $g^k$ in the monomial basis
	and applying~\eqref{eq:box_factor} term by term, producing $y_k$ as an exact
	rational. As $g$ is quasi-homogeneous of weighted degree $m$, the power $g^k$ is
	quasi-homogeneous of weighted degree $km$, hence a combination of the
	\[
	\sigma_{km} \;:=\; \#\{\alpha \in \N^n : w\cdot\alpha = km\}
	\;=\; O\!\left(\frac{(km)^{\,n-1}}{(n-1)!\,\prod_i w_i}\right)
	\]
	monomials of that degree. Building $g^1, \dots, g^{2r}$ by repeated multiplication
	and collecting like terms costs
	\begin{equation}\label{eq:gevp_moment_cost}
		O\!\left(s \sum_{k=1}^{2r} \sigma_{(k-1)m}\right)
		\;=\;
		O\!\left(\frac{s\,(2r)^n\, m^{\,n-1}}{n!\,\prod_i w_i}\right)
	\end{equation}
	field operations, after which the $O(n)$-cost factorization~\eqref{eq:box_factor}
	of each surviving monomial is lower order. The bound~\eqref{eq:gevp_moment_cost} is
	the dense worst case; when $g$ is sparse the count is governed instead by
	$\#\{\text{monomials of } g^k\} \leq \binom{k+s-1}{s-1} = O(k^{\,s-1})$, which for
	the running example ($s = 2$) is merely $k+1$, so the entire sequence
	$y_0, \dots, y_{2r}$ costs $O(r^2)$ rational operations.
	
	\textsc{Cheb} needs the Chebyshev moments
	$y_k^{\mathrm{ch}} = \int_B T_k(2g/\bg - 1)\, d\mu$ of~\eqref{eq:Vk} for
	$0 \leq k \leq r$ only --- half the order required by \textsc{Gevp}. These are
	deliberately \emph{not} expanded symbolically, as that would reintroduce the
	$\bg^{\,k}$ growth and catastrophic cancellation that the Chebyshev
	parameterization exists to avoid (Section~\ref{sec:cheb_bracket}); they are
	computed by tensor Gauss-Legendre cubature exact for the integrand. Since
	$T_k(2g/\bg - 1)$ has per-coordinate degree at most $km/w_i \leq rm/w_i$, taking
	\[
	Q_i = \Bigl\lceil \tfrac{rm/w_i + 1}{2} \Bigr\rceil \ \text{nodes in coordinate } i,
	\qquad
	N_{\mathrm{ch}} := \prod_{i=1}^n Q_i
	= O\!\left(\frac{(rm)^n}{2^n \prod_i w_i}\right),
	\]
	integrates $y_0^{\mathrm{ch}}, \dots, y_r^{\mathrm{ch}}$ exactly. At each node, one
	evaluation of $g$ ($O(s)$) feeds the three-term recurrence~\eqref{eq:cheb_rec}
	that produces $T_0, \dots, T_r$ ($O(r)$) and updates the $r+1$ running sums
	($O(r)$), for a total of
	\begin{equation}\label{eq:cheb_moment_cost}
		O\bigl(N_{\mathrm{ch}}\,(s + r)\bigr)
		\;=\;
		O\!\left(\frac{(rm)^n\,(s+r)}{2^n \prod_i w_i}\right)
	\end{equation}
	operations, all in fixed (double) precision: $g(x) \in [0,\bg]$ forces the
	argument into $[-1,1]$ and keeps every Chebyshev value --- and hence
	$|y_k^{\mathrm{ch}}| \leq 1$ --- bounded (Section~\ref{sec:cheb_bracket}). For
	the running example, $Q_1 = \lceil (4r+1)/2\rceil$, $Q_2 = \lceil (2r+1)/2\rceil$,
	a grid of $\approx 2r^2$ nodes and $O(r^3)$ flops; note that the tensor grid,
	unlike the symbolic route~\eqref{eq:gevp_moment_cost}, does not benefit from the
	sparsity of $g$.
	
	\subsubsection{Favorable large dimensional cases}
	
	The estimates \eqref{eq:gevp_moment_cost}--\eqref{eq:cheb_moment_cost} deserve a
	closer look, because they are the only obstruction to running the two algorithms in
	large ambient dimension, and because the asymptotic regime in which they were stated
	is not the one of interest here.
	
	Write
	\[
	N_k \;:=\; \#\{\alpha \in \N^n : (g^k)_\alpha \neq 0\}
	\]
	for the number of monomials of $g^k$. Two bounds apply simultaneously: expanding the
	$k$-th power of an $s$-term polynomial gives $N_k \le \binom{k+s-1}{s-1}$, while
	quasi-homogeneity confines $g^k$ to the weighted-degree-$km$ slice, so
	$N_k \le \sigma_{km}$. Hence
	\begin{equation}\label{eq:Nk_sharp}
		N_k \;\le\; \min\Bigl\{\tbinom{k+s-1}{s-1},\ \sigma_{km}\Bigr\},
		\qquad
		\sigma_{km} = \tbinom{n+km-1}{km} \ \text{ for } w = (1,\ldots,1).
	\end{equation}
	The count \eqref{eq:gevp_moment_cost} reported the regime $km \gg n$, in which
	$\sigma_{km} = O\bigl((km)^{n-1}/(n-1)!\bigr)$; the complementary regime $n \gg km$
	gives $\sigma_{km} = O\bigl(n^{km}/(km)!\bigr)$. The binomial is exponential in
	$\min(n, km)$, and in the regime that motivates the pushforward --- $n$ large, $m$
	small --- it is $n$ that binds. Concretely, a \emph{dense} quadratic form in $n = 20$
	variables requires, at order $r = 30$, the moment $y_{60}$, hence
	$\sigma_{60} = \binom{79}{60} \approx 10^{19}$ monomials: out of reach. The
	one-dimensional collapse achieved in Section~\ref{sec:pushforward} removes the
	exponential dependence from the \emph{optimization} --- positive semidefinite blocks
	of size $\binom{n+r}{n}$ become pencils of size $r+1$ --- but it does not remove it
	from the problem; it confines it entirely to a single preprocessing stage, performed
	once and independent of the relaxation order thereafter. Whether the method scales
	in $n$ is therefore exactly the question of whether that stage is tractable, and the
	answer depends on the structure of $g$, not on its dimension.
	
	\paragraph{Sparse forms.}
	For fixed $s = \#\{\alpha : g_\alpha \neq 0\}$ the first bound in \eqref{eq:Nk_sharp}
	gives $N_k = O(k^{s-1})$, so the whole sequence $y_0, \dots, y_{2r}$ costs
	$O(s\, r^{s})$ rational operations, \emph{polynomially in $n$} (indeed independently
	of $n$, apart from the $O(n)$ cost of each factorization \eqref{eq:box_factor}).
	
	\paragraph{Block-separable forms.}
	Let $I_1, \dots, I_J$ partition $\{1,\dots,n\}$ into blocks of sizes
	$b_1, \dots, b_J$, and let
	\begin{equation}\label{eq:blocksep}
		g(x) \;=\; \sum_{j=1}^J g_j\bigl(x_{I_j}\bigr),
	\end{equation}
	each $g_j$ quasi-homogeneous of the same weighted degree $m$ with respect to
	$w|_{I_j}$ (so that $g$ is quasi-homogeneous of degree $m$). This is the
	correlative-sparsity pattern exploited, in the ambient hierarchy, by
	\cite{TacchiWeisserLasserreHenrion2022}. Here it linearizes the moment computation.
	
	\begin{proposition}[Convolution of block pushforwards]\label{prop:blockmoments}
		Let $\mu_j$ be the uniform probability measure on $[-1,1]^{I_j}$ and
		$\nu_j := (g_j)_\# \mu_j$, with moments $y^{(j)}_k$. Under \eqref{eq:blocksep} the
		pushforward $\nu = g_\#\mu$ is the additive convolution
		$\nu = \nu_1 * \cdots * \nu_J$, so that the exponential generating functions
		multiply,
		\begin{equation}\label{eq:egf}
			\sum_{k \ge 0} y_k \frac{u^k}{k!}
			\;=\; \prod_{j=1}^J \Bigl(\sum_{k \ge 0} y^{(j)}_k \frac{u^k}{k!}\Bigr) ,
		\end{equation}
		and $\bg = \sum_j \bg_j$ with $\bg_j := \max_{[-1,1]^{I_j}} g_j$.
	\end{proposition}
	
	\begin{proof}
		The measure $\mu$ is the product $\bigotimes_j \mu_j$, so the coordinates $x_{I_j}$
		are independent under $\mu$ and the random variables $g_j(x_{I_j})$ are independent
		with laws $\nu_j$. Their sum $g(x)$ has law the convolution of the $\nu_j$, and
		\eqref{eq:egf} is the multiplicativity of the moment generating function of a sum of
		independent variables, read coefficientwise. The formula for $\bg$ is the
		separability of the maximization.
	\end{proof}
	
	\subsection{Cost of \textsc{Cheb}}\label{sec:cost_cheb}
	
	Given the Chebyshev moments, the remaining \textsc{Cheb} work is negligible and
	stays in fixed precision.
	\begin{itemize}[leftmargin=*,itemsep=2pt]
		\item \emph{Coefficients.} The closed forms of Theorem~\ref{thm:cheb_coefs}
		produce $c_0, \dots, c_r$ from a single angle $\varphi = \arccos(2/\bg - 1)$ in
		$O(r)$ operations.
		\item \emph{Uniform error.} The certified $\varepsilon_r$ is either the
		closed-form majorant $\tfrac{\bg(1+\sin\varphi)}{2\pi}(\tfrac1r + \tfrac1{r+1})$
		of Corollary~\ref{cor:cheb_decay}, an $O(1)$ evaluation, or the exact value
		obtained by locating the extrema of the piecewise degree-$r$ error polynomial
		$f - f_r^{\mathrm{ch}}$ on $[0,\bg]$ --- the roots of a degree-$(r-1)$
		derivative, found in $O(r^3)$ double-precision operations.
		\item \emph{Estimate and bracket.} The point value $V_r$ of~\eqref{eq:Vk} is the
		length-$(r+1)$ inner product
		$\tfrac{2^n(\ws+m)}{m}\bigl(\tfrac{c_0}{2} y_0 + \sum_{j \geq 1} c_j y_j^{\mathrm{ch}}\bigr)$,
		$O(r)$ operations, and the certified bracket of
		Proposition~\ref{prop:cheb_bracket} is then $O(1)$.
	\end{itemize}
	The cubature~\eqref{eq:cheb_moment_cost} dominates: at order $r$, \textsc{Cheb}
	costs $O\bigl((rm)^n (s+r)/(2^n \prod_i w_i)\bigr)$ arithmetic operations,
	entirely in fixed precision, with no linear-algebra solve of any kind.
	
	\subsection{Cost of \textsc{Gevp}}\label{sec:cost_gevp}
	
	Given the power moments, \textsc{Gevp} performs genuine linear algebra, at a
	precision that grows with $r$.
	\begin{itemize}[leftmargin=*,itemsep=2pt]
		\item \emph{Reference moments and assembly.} The rationals $z_0, \dots, z_{2r}$
		of~\eqref{eq:stokes_moment} cost $O(r)$. The four matrices of~\eqref{eq:vr_pm}
		are Hankel (the moment matrices $M_r(\nu), M_r(\nu_w)$) and Hankel-from-a-three-term
		combination (the localizing matrices, entries $\sum_{l=0}^2 h_l\,y_{i+j+l}$ and
		$\sum_{l=0}^2 h_l\,z_{i+j+l}$); each is determined by $O(r)$ numbers, and filling
		the dense $(r+1)\times(r+1)$ and $r \times r$ arrays is $O(r^2)$.
		\item \emph{Two eigenproblems.} Each bound is $2^n$ times the smallest generalized
		eigenvalue of a symmetric-definite pencil (Lemma~\ref{lem:hmu_pos},
		Remark~\ref{rem:nosdp}): a Cholesky factorization of the definite member, a
		reduction to standard symmetric form, and a symmetric eigensolve --- $O(r^3)$
		field operations for the size-$(r+1)$ upper pencil and likewise for the size-$r$
		lower pencil. 
		\item \emph{Precision.} This is the decisive cost. The Hankel pencils are ill-conditioned, at the rate quantified in
		Remark~\ref{rem:rate_precision}: resolving a gap of size $\gamma^{-2r}$ demands
		the working precision $P = P(r)$ of \eqref{eq:precision_budget}, linear in $r$
		with slope $2\log_{10}\bg + 1.531$.
		Every field operation then costs $M(P)$, the price of a $P$-digit multiplication
		--- $O(P^2)$ schoolbook, $O(P \log P)$ with fast multiplication --- with
		$P = \Theta(r)$.
	\end{itemize}
	The two eigensolves therefore cost $O\bigl(r^3 M(P)\bigr)$ bit operations, i.e.\
	$O(r^5)$ with schoolbook and $\widetilde O(r^4)$ with fast multiplication; the
	exact moments must be carried at the same $P$ digits. (The Hankel structure
	admits fast and superfast structured eigensolvers that lower the $O(r^3)$
	field-operation count, but not the precision factor $M(P)$.) For fixed $n$ and
	growing $r$ this precision-driven term dominates the cost; for growing $n$ the
	moment cost~\eqref{eq:gevp_moment_cost} dominates instead.
	
	\subsection{Comparison}\label{sec:cost_compare}
	
	Table~\ref{tab:complexity} collects the per-order costs.
	
	\begin{table}[ht]
		\centering
		\small
		\renewcommand{\arraystretch}{1.35}
		\begin{tabular}{@{}p{0.26\textwidth} p{0.32\textwidth} p{0.34\textwidth}@{}}
			\toprule
			& \textsc{Cheb} & \textsc{Gevp}\\
			\midrule
			moments used
			& $y_k^{\mathrm{ch}}$, \ $0 \leq k \leq r$
			& $y_k$, \ $0 \leq k \leq 2r$\\
			moment route
			& tensor Gauss-Legendre cubature
			& exact monomial expansion (rational)\\
			moment cost
			& $O\bigl((rm)^n (s{+}r)/(2^n\!\prod_i w_i)\bigr)$
			& $O\bigl(s\,(2r)^n m^{n-1}/(n!\!\prod_i w_i)\bigr)$\\
			arithmetic precision
			& fixed (double)
			& $P(r)$ of \eqref{eq:precision_budget}, $\Theta(r)$ digits \\
			downstream algebra
			& one length-$(r{+}1)$ inner product
			& two sym.-def.\ GEVPs, sizes $r{+}1$, $r$\\
			downstream cost
			& $O(r)$
			& $O(r^3)$ ops $= O\bigl(r^3 M(P)\bigr)$ bits\\
			eigensolver
			& none
			& one symmetric eigensolve per bound\\
			$n$ enters via
			& box moments only
			& box moments only\\
			\bottomrule
		\end{tabular}
		\caption{Per-order computational cost of the two algorithms at relaxation
			order $r$. Here $s$ is the number of
			monomials of $g$, and $M(P)$ the cost of a $P$-digit multiplication.}
		\label{tab:complexity}
	\end{table}
	
	\begin{remark}[Where the two methods differ]\label{rem:cost_summary}
		Three contrasts stand out. \emph{(i) Order of the moments:} \textsc{Cheb}
		consumes moments to order $r$, \textsc{Gevp} to order $2r$. \emph{(ii)
			Precision:} \textsc{Cheb} is a fixed-precision method --- the Chebyshev
		parameterization keeps the Chebyshev moments in $[-1,1]$ and the estimate
		$O(2^n)$ --- whereas \textsc{Gevp} is a
		variable-precision method whose digit budget grows linearly in $r$, the
		arithmetic price of the geometric rate (Remark~\ref{rem:rate_precision}).
		\emph{(iii) Linear algebra:} beyond the shared cubature/expansion, \textsc{Cheb}
		forms a single inner product while \textsc{Gevp} solves two dense
		symmetric-definite eigenproblems --- yet \emph{neither} method invokes a
		semidefinite-programming solver, the upper bound $v_r^+$ of~\eqref{eq:vr_pm}
		being delivered, as recorded in Section~\ref{sec:gevp_bounds}, by an eigenvalue
		routine alone. The shared, and only dimension-dependent, expense is the
		box-moment computation; the entire downstream is univariate.
	\end{remark}
	
	Combined with the convergence rates of Sections~\ref{sec:rate_cheb}
	and~\ref{sec:rate_gevp}, this per-order accounting completes the
	rate-robustness picture anticipated in Remark~\ref{rem:rate_precision}:
	\textsc{Cheb} spends little per order and in fixed precision, but needs many
	orders; \textsc{Gevp} spends much per order and at a precision growing with the
	order, but needs few. Which is preferable is therefore dictated by the target
	accuracy and by the cost of extended-precision arithmetic on the platform at
	hand, not by either count in isolation.
	
	\section{Volume approximation with the moment-SOS hierarchy}\label{sec:momsos}
	
	The two algorithms of Sections~\ref{sec:cheb} and~\ref{sec:gevp} were designed to
	\emph{bypass} semidefinite programming. It is nonetheless illuminating to place
	the pushforward in the classical moment-SOS framework of \cite{hls2009}, for three
	reasons: it exhibits the exact mechanism by which the Stokes structure of
	Section~\ref{sec:pushforward} removes the Gibbs phenomenon; it produces a
	\emph{continuous} analytic dual certificate, the weighted ReLU of
	Proposition~\ref{prop:relu}, in closed form; and it is the formulation against
	which \textsc{Cheb} and \textsc{Gevp} should be compared.
	
	\subsection{Primal-dual formulation on the pushforward}\label{sec:momsos_lp}
	
	The restriction $\nu|_{[0,1]}$ is the largest non-negative measure supported on
	$[0,1]$ and dominated by $\nu$. Writing $\sigma$ for the unknown restriction and
	$\hat\sigma := \nu - \sigma$ for the complementary slack measure on $[0,\bg]$, this
	characterization is the infinite-dimensional primal linear program
	\begin{equation}\label{eq:momsos_primal}
		\nu([0,1]) \;=\;
		\sup_{\sigma,\,\hat\sigma}\ \int d\sigma
		\quad\text{s.t.}\quad
		\sigma + \hat\sigma = \nu,\quad
		\sigma \geq 0 \ \text{on}\ [0,1],\quad
		\hat\sigma \geq 0 \ \text{on}\ [0,\bg],
	\end{equation}
	whose optimal solution is $\sigma^\star = \nu|_{[0,1]}$,
	$\hat\sigma^\star = \nu|_{(1,\bg]}$, and whose value is the normalized volume
	$\nu([0,1]) = 2^{-n}\vol K$. Lagrange duality gives the dual program on continuous
	functions
	\begin{equation}\label{eq:momsos_dual}
		\nu([0,1]) \;=\;
		\inf_{p}\ \int p\, d\nu
		\quad\text{s.t.}\quad
		p \geq 1 \ \text{on}\ [0,1],\quad
		p \geq 0 \ \text{on}\ [0,\bg].
	\end{equation}
	The constraints force $p \geq \mathbf 1_{[0,1]}$ on $[0,\bg]$, so any feasible $p$
	yields the upper bound $\int p\, d\nu \geq \nu([0,1])$, and the formal optimum is
	the \emph{discontinuous} indicator $p^\star = \mathbf 1_{[0,1]}$. The order-$r$
	relaxation replaces $p$ by a polynomial of degree $\leq 2r$; majorizing a jump by
	polynomials forces the Gibbs overshoot, and this is the root cause of the slow,
	oscillatory convergence of the bare hierarchy (left panel of
	Figure~\ref{fig:gibbs}).
	
	\subsection{Stokes constraints and the analytic dual certificate}\label{sec:momsos_stokes}
	
	Quasi-homogeneity furnishes extra information on the restriction $\sigma$ that the
	bare program~\eqref{eq:momsos_primal} ignores. By
	Proposition~\ref{prop:nu_restriction}, $\nu|_{[0,1]} = 2^{-n}\vol K\,\nu_w$, so the
	restriction carries the \emph{exact} reference moments $z_k = \ws/(\ws+km)$ of
	$\nu_w$ up to the common scalar; equivalently, for the functions
	\begin{equation}\label{eq:stokes_fun}
		a_k(t) \;:=\; t^k - z_k, \qquad k \geq 1,
	\end{equation}
	one has the linear \emph{Stokes constraints} $\int a_k\, d\sigma = 0$ (this is
	Lemma~\ref{lem:wstokes} read on the pushforward, since
	$\int a_k\, d\nu_w = z_k - z_k = 0$). Adjoining them to~\eqref{eq:momsos_primal}
	leaves the optimum unchanged but tightens every relaxation, and introduces in the
	dual a free multiplier $q = \sum_{k\geq 1} q_k\, a_k$:
	\begin{equation}\label{eq:stokes_dual}
		\nu([0,1]) \;=\;
		\inf_{p,\,q}\ \int p\, d\nu
		\quad\text{s.t.}\quad
		p + q \geq 1 \ \text{on}\ [0,1],\quad
		p \geq 0 \ \text{on}\ [0,\bg],\quad
		q = \sum_{k\geq 1} q_k\, a_k .
	\end{equation}
	The multiplier $q$ integrates to zero against the optimal restriction
	($\int q\, d\sigma^\star = 0$), so it does not enter the objective; its sole role
	is to relax the constraint on $[0,1]$, allowing $p$ to vanish where
	$\mathbf 1_{[0,1]}$ would force a jump. The dual no longer has to approximate a
	discontinuous function --- and a single multiplier exhibits the optimal
	\emph{continuous} certificate explicitly.
	
	\begin{proposition}[Analytic dual certificate]\label{prop:dual_certificate}
		Set
		\[
		p^\star(t) := \frac{\ws+m}{m}\,(1-t)_+
		\]
		and
		\[
		q^\star(t) := \frac{\ws+m}{m}\bigl(t - \frac{\ws}{\ws+m}\bigr), \quad \text{i.e.} \quad 
		q_1^\star := \frac{\ws+m}{m}, \quad q_k^\star := 0 \ (k\geq 2).
		\]
		Then $(p^\star, q^\star)$ is optimal for~\eqref{eq:stokes_dual}: it is feasible,
		with $p^\star + q^\star = 1$ on $[0,1]$ and $p^\star \geq 0$ on $[0,\bg]$, and its
		value is $\int p^\star\, d\nu = \nu([0,1]) = 2^{-n}\vol K$.
	\end{proposition}
	
	\begin{proof}
		For $t \in [0,1]$,
		$(1-t)_+ = 1-t$, hence
		\[
		p^\star(t) + q^\star(t)
		= \frac{\ws+m}{m}\Bigl[(1-t) + \bigl(t - \frac{\ws}{\ws+m}\bigr)\Bigr]
		= \frac{\ws+m}{m}\Bigl(1 - \frac{\ws}{\ws+m}\Bigr)
		= 1,
		\]
		so the first constraint holds with equality; $p^\star \geq 0$ on $[0,\bg]$ is
		clear. Feasibility for the primal at $\sigma^\star = \nu|_{[0,1]}$ and weak
		duality give $\int p^\star\, d\nu \geq \nu([0,1])$, while the weighted ReLU
		identity~\eqref{eq:relu}, in the normalized form
		$\nu([0,1]) = 2^{-n}\vol K = \tfrac{\ws+m}{m}\int_0^{\bg}(1-t)_+\, d\nu$, gives
		$\int p^\star\, d\nu = \tfrac{\ws+m}{m}\int (1-t)_+\, d\nu = \nu([0,1])$. Hence
		$(p^\star,q^\star)$ attains the dual value and is optimal.
	\end{proof}
	
	Thus the Stokes-augmented dual replaces the discontinuous target
	$\mathbf 1_{[0,1]}$ of~\eqref{eq:momsos_dual} by the \emph{continuous, piecewise
		affine} weighted ReLU $\tfrac{\ws+m}{m}(1-t)_+$ of Proposition~\ref{prop:relu}. 
	
	\subsection{Finite-dimensional SOS relaxation}\label{sec:momsos_sos}
	
	In practice one fixes a relaxation order $r$ and computes finitely many moments
	$y_0, \ldots, y_{2r}$ of $\nu$, cf.~\eqref{eq:y_moments}. The order-$r$ dual is a
	univariate polynomial majorization problem,
	\begin{equation}\label{eq:sos_dual_finite}
		d^{\mathrm{sos}}_r:=\min_{p \in \R[t]_{2r}}\ \int p\, d\nu
		\quad\text{s.t.}\quad
		p \geq f \ \text{on}\ [0,1],\quad
		p \geq 0 \ \text{on}\ [0,\bg],
	\end{equation}
	where the target is $f \equiv 1$ for the bare hierarchy~\eqref{eq:momsos_dual} and
	$f(t) = \tfrac{\ws+m}{m}(1-t)$ for the Stokes hierarchy~\eqref{eq:stokes_dual}
	(having substituted the optimal degree-one multiplier $q^\star$ of
	Proposition~\ref{prop:dual_certificate}, which is admissible at every order
	$r \geq 1$); in the latter case the two constraints combine into
	$p \geq \tfrac{\ws+m}{m}(1-t)_+$ on $[0,\bg]$. In one variable the interval
	positivity constraints are represented \emph{exactly} by sum-of-squares
	certificates of Markov-Luk\'acs type: a degree-$2r$ polynomial satisfies
	$p \geq 0$ on $[0,\bg]$ iff $p = \sigma_0 + t(\bg-t)\,\sigma_1$, and $p - f \geq 0$
	on $[0,1]$ iff $p - f = \tau_0 + t(1-t)\,\tau_1$, with $\sigma_0,\tau_0$ and
	$\sigma_1,\tau_1$ sums of squares of degree $\leq 2r$ and $\leq 2r-2$ respectively.
	Problem~\eqref{eq:sos_dual_finite} is therefore a semidefinite program in the four
	Gram matrices.
	
	To keep the program well conditioned we express every polynomial in the shifted
	Chebyshev basis $T_k(2t/\bg - 1)$ of $[0,\bg]$ rather than in the monomial basis:
	the Gram-matrix-to-coefficient map is assembled from the linearization
	$T_iT_j = \tfrac12(T_{i+j}+T_{|i-j|})$, the localizer $t(\bg-t)$ has the closed
	form $\tfrac{\bg^2}{8}\,(T_0 - T_2)$ in the variable $s = 2t/\bg-1$ (and $t(1-t)$ is
	the elementary degree-two polynomial $-\tfrac12 T_0 - T_1 - \tfrac12 T_2$ when
	$\bg = 2$), and the objective is $\int p\, d\nu = \sum_k p_k\, C_k$ with Chebyshev
	moments
	$C_k = \int_0^{\bg} T_k(2t/\bg-1)\, d\nu = \int_B T_k(2g/\bg-1)\, d\mu \in [-1,1]$
	computed by the exact Gauss-Legendre cubature of Section~\ref{sec:cheb}. This is
	the same conditioning advantage that the Chebyshev parameterization confers
	on \textsc{Cheb}.
	
	Figure~\ref{fig:gibbs} shows the two degree-$20$ optimal majorants for the running
	example $g = x_1^4 + x_2^2$, with $\bg = 2$, $\tfrac{\ws+m}{m} = 7/4$. The bare
	hierarchy~\eqref{eq:momsos_dual} majorizes the indicator $\mathbf 1_{[0,1]}$: the
	optimal polynomial overshoots above $1$ at the level $t = 1$ and rings on
	$[1,\bg]$ --- the Gibbs phenomenon --- and the resulting bound is loose,
	$\vol K \leq 3.695$. The Stokes hierarchy~\eqref{eq:stokes_dual} instead majorizes
	the continuous ReLU $\tfrac{7}{4}(1-t)_+$: the optimal polynomial tracks it without
	oscillation and returns the far tighter bound $\vol K \leq 3.506$, against the exact
	value $\vol K = 3.4961$.
	
	\begin{figure}[ht]
		\centering
		\includegraphics[width=\textwidth]{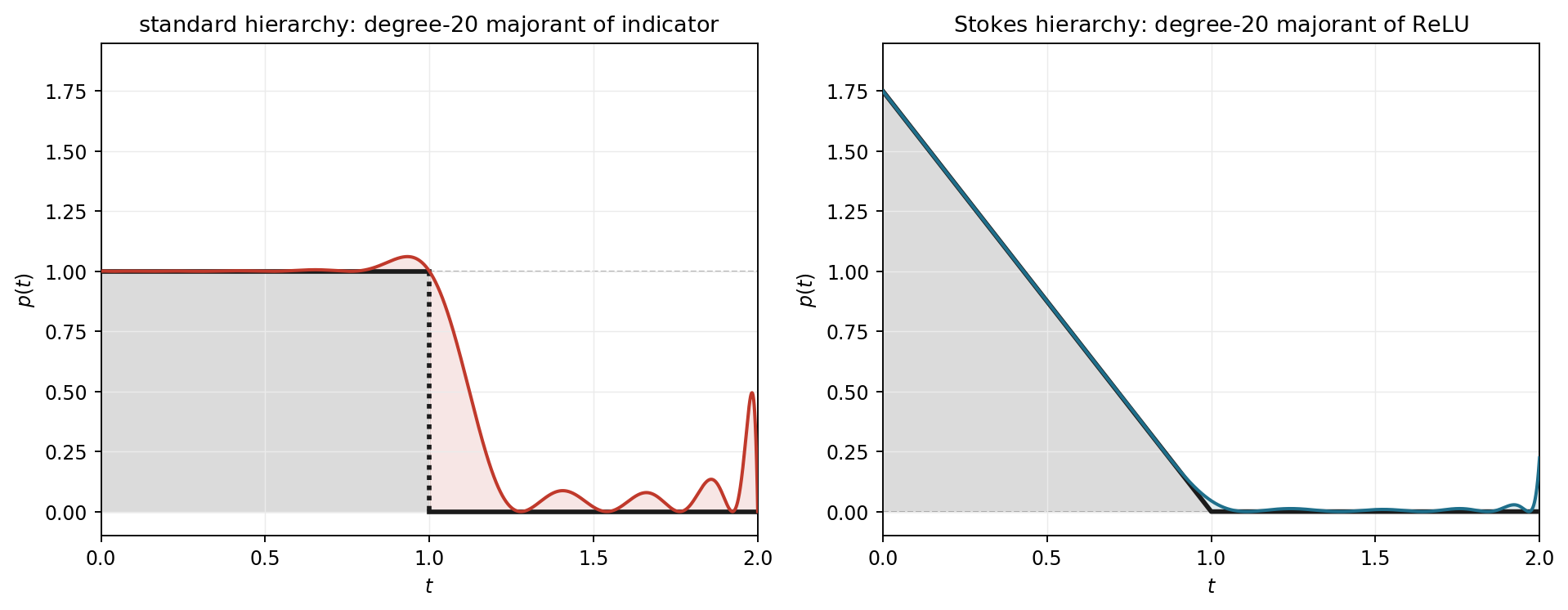}
		\caption{Degree-$20$ optimal SOS majorants for the moment-SOS dual on the
			pushforward $\nu$ of the running example, solved in the shifted
			Chebyshev basis. \emph{Left:} the bare hierarchy~\eqref{eq:momsos_dual}
			majorizes the discontinuous indicator $\mathbf 1_{[0,1]}$ (grey), with a
			marked Gibbs overshoot at the level set $t = 1$ and oscillations on
			$[1,\bg]$; the corresponding upper bound is $\vol K \leq 3.695$.
			\emph{Right:} the Stokes-augmented hierarchy~\eqref{eq:stokes_dual} majorizes
			the continuous weighted ReLU $\tfrac74(1-t)_+$
			(grey) of Proposition~\ref{prop:dual_certificate}, exhibiting no Gibbs
			phenomenon and the much tighter bound $\vol K \leq 3.506$ (exact value
			$\vol K = 3.4961$). The residual upswing near $t = \bg$ reflects that $\nu$
			has vanishing density there, leaving the majorant essentially unconstrained.}
		\label{fig:gibbs}
	\end{figure}
	
	\subsection{Closed-form dual certificates from \textsc{Cheb}}\label{sec:momsos_cheb_dual}
	
	\textsc{Cheb} and the Stokes-accelerated
	dual~\eqref{eq:stokes_dual} approximate the \emph{same} continuous function, the
	weighted ReLU of Proposition~\ref{prop:relu}. We now make this precise: the
	\emph{upper} endpoint $v_r^{\mathrm{ch},+}$ of the \textsc{Cheb}
	bracket~\eqref{eq:cheb_vpm} is exactly the objective value of an explicit,
	closed-form polynomial that is dual-feasible for the order-$r$ relaxation, together
	with explicit SOS certificates. Thus the \textsc{Cheb} upper bound is a
	genuine moment-SOS dual bound, produced without any semidefinite solver.
	
	Throughout we fix the analytic multiplier $q^\star$ of
	Proposition~\ref{prop:dual_certificate}. With $q = q^\star$, and using
	$1 - q^\star(t) = \tfrac{\ws+m}{m}(1-t)$, the two dual constraints
	of~\eqref{eq:stokes_dual} read
	\begin{equation}\label{eq:cheb_dual_reduced}
		p \ge \tfrac{\ws+m}{m}(1-t)\ \text{on}\ [0,1],
		\qquad
		p \ge 0\ \text{on}\ [0,\bg].
	\end{equation}
	Since $\tfrac{\ws+m}{m}(1-t)\ge 0$ on $[0,1]$, the pair~\eqref{eq:cheb_dual_reduced}
	is equivalent to the single requirement that $p$ majorize the continuous
	weighted ReLU on the whole interval,
	\begin{equation}\label{eq:cheb_dual_majorize}
		p(t) \ge \tfrac{\ws+m}{m}(1-t)_+, \qquad t\in[0,\bg].
	\end{equation}
	
	\begin{lemma}\label{lem:markov_lukacs}
		Let $a<b$, set $g(t):=(t-a)(b-t)$, and let $p\in\R[t]$ with $\deg p=2r$.
		Then $p\ge 0$ on $[a,b]$ if and only if
		\begin{equation}\label{eq:ML}
			p \;=\; \sigma_0 + g\,\sigma_1,
			\qquad \sigma_0\in\Sigma[t]_{2r},\quad \sigma_1\in\Sigma[t]_{2(r-1)}
		\end{equation}
		where $\Sigma[t]_{2r}$ is the cone of polynomial SOS of degree up to $2r$.
	\end{lemma}
	
	Representation \eqref{eq:ML} is classical \cite[Theorem~2.6(a)]{lasserre2010} and is commonly called the
	\emph{Markov-Luk\'acs theorem}. 
	
	\paragraph{Proof.}
	The necessary direction is immediate: $g\ge 0$ on $[a,b]$ and $\sigma_0,\sigma_1\ge 0$
	everywhere, so $p\ge 0$ on $[a,b]$. For the converse we construct \eqref{eq:ML}.
	
	\emph{Step 0 (reduction to $[-1,1]$).} The affine map
	$t=\tfrac{a+b}{2}+\tfrac{b-a}{2}\,s$ carries $[-1,1]$ onto $[a,b]$ and
	$g(t)=\tfrac{(b-a)^2}{4}(1-s^2)$. The SOS structure and the degree bounds are
	preserved, and transporting back only multiplies $\sigma_1$ by the positive
	constant $4/(b-a)^2$; hence we may assume $[a,b]=[-1,1]$ and produce
	$p=\sigma_0+(1-s^2)\sigma_1$.
	
	\emph{Step 1 (trigonometric lift).} Put $s=\cos\theta$. Since $p\ge 0$ on
	$[-1,1]$, the function $T(\theta):=p(\cos\theta)$ is a nonnegative trigonometric
	polynomial, even in $\theta$, of degree $2r$,
	\[
	T(\theta)=\sum_{k=-2r}^{2r} c_k\,e^{ik\theta},
	\qquad c_{-k}=c_k\in\R,\qquad T\ge 0,
	\]
	whose coefficients are, up to normalisation, the Chebyshev coefficients of $p$.
	
	\emph{Step 2 (Fej\'er-Riesz factorisation).} Multiplying the Laurent polynomial
	$\sum_k c_k z^k$ by $z^{2r}$ gives the palindromic polynomial
	$R(z)=\sum_{m=0}^{4r} c_{m-2r}\,z^m$ of degree $4r$. Its real coefficients make
	its zero set invariant under $z\mapsto\bar z$, palindromy makes it invariant
	under $z\mapsto 1/z$, and $T\ge 0$ forces every zero on the unit circle to have
	even multiplicity. Collecting, from each conjugate-reciprocal orbit, the
	representatives in the closed unit disk (unit-circle zeros with half their
	multiplicity) yields $G(z)=\sum_{k=0}^{2r} g_k z^k$ with
	$T(\theta)=\lvert G(e^{i\theta})\rvert^2$. Since the chosen zero multiset is
	closed under conjugation, $G$ may be taken with \emph{real} coefficients,
	$g_k\in\R$.
	
	\emph{Step 3 (recentring).} Define the symmetric (linear-phase) factor
	\[
	\Phi(\theta):=e^{-ir\theta}\,G(e^{i\theta})
	=\sum_{k=0}^{2r} g_k\,e^{\,i(k-r)\theta}
	=\sum_{j=-r}^{r}\gamma_j\,e^{ij\theta},
	\qquad \gamma_j:=g_{j+r}\in\R,
	\]
	so that $\lvert\Phi(\theta)\rvert^2=T(\theta)$ and, since $\gamma_j\in\R$,
	$\Phi(-\theta)=\overline{\Phi(\theta)}$. Hence $\operatorname{Re}\Phi$ is even and
	$\operatorname{Im}\Phi$ is odd in $\theta$, and
	\[
	\operatorname{Re}\Phi=A(\cos\theta),\qquad
	\operatorname{Im}\Phi=\sin\theta\,B(\cos\theta),
	\qquad \deg A\le r,\ \deg B\le r-1,
	\]
	using $\cos j\theta=T_j(\cos\theta)$ and
	$\sin j\theta=\sin\theta\,U_{j-1}(\cos\theta)$ ($T_j,U_\ell$ the Chebyshev
	polynomials of the first and second kind). Both sides being polynomials in $s$
	that agree on $[-1,1]$,
	\[
	p(s)=\lvert\Phi\rvert^2=A(s)^2+\sin^2\!\theta\,B(s)^2
	=A(s)^2+(1-s^2)\,B(s)^2,
	\]
	which is \eqref{eq:ML} with $\sigma_0=A^2$ and $\sigma_1=B^2$ single squares,
	$\deg\sigma_0\le 2r$, $\deg\sigma_1\le 2r-2$.
	
	\emph{Step 4 (transport).} Substituting $s=\tfrac{2t-a-b}{b-a}$ returns the
	$[-1,1]$ identity to $p=\sigma_0+g\,\sigma_1$ on $[a,b]$ with the asserted degree
	bounds.
	
	Representation \eqref{eq:ML} is obtained by numerical linear algebra alone; no
	semidefinite solver is required. After the affine reduction and the substitution
	$s=\cos\theta$, the only nontrivial step is the Fej\'er-Riesz factorisation of
	Step~2, i.e.\ computing the zeros of the degree-$4r$ palindromic polynomial $R$ and
	keeping those in the closed unit disk. For moderate degrees ($r$ up to a few
	hundred) this is a backward-stable companion-matrix eigenvalue computation, at cost
	$O(r^3)$; dedicated polynomial root-finders bring this to roughly $O(r^2)$ and reach
	much higher degrees. When $r$ is large one bypasses the explicit zeros and iterates
	the spectral factor directly.
	
	\begin{theorem}[\textsc{Cheb} dual certificate]\label{thm:cheb_dual}
		Let $r\ge 1$. Let $f_r^{\mathrm{ch}}\in\R[t]_r$ be the truncated Chebyshev
		approximant~\eqref{eq:qkC} of the ReLU $f(t)=(1-t)_+$ on $[0,\bg]$, with the
		closed-form coefficients~\eqref{eq:c0}--\eqref{eq:cj}, and let $\varepsilon_r$ be
		any certified upper bound on the uniform error~\eqref{eq:eps_def} (for instance
		the closed form of Corollary~\ref{cor:cheb_decay}). Define the degree-$r$
		polynomial
		\begin{equation}\label{eq:cheb_dual_poly}
			p_r(t) := \frac{\ws+m}{m}\bigl(f_r^{\mathrm{ch}}(t) + \varepsilon_r\bigr).
		\end{equation}
		Then, with the multiplier $q^\star$ of Proposition~\ref{prop:dual_certificate}:
		\begin{enumerate}
			\item[(i)] \emph{(Feasibility.)} $p_r$ satisfies~\eqref{eq:cheb_dual_majorize},
			equivalently~\eqref{eq:cheb_dual_reduced}; hence $(p_r,q^\star)$ is feasible
			for the order-$r$ Stokes relaxation~\eqref{eq:sos_dual_finite}.
			\item[(ii)] \emph{(SOS certificates.)} Both inequalities
			in~\eqref{eq:cheb_dual_reduced} admit degree $2r$  
			SOS certificates   (Lemma~\ref{lem:markov_lukacs}).
			\item[(iii)] \emph{(Value.)} The objective value of $p_r$ reproduces the
			upper \textsc{Cheb} endpoint,
			\begin{equation}\label{eq:cheb_dual_value}
				2^n\!\int_0^{\bg} p_r\,d\nu
				\;=\; V_r + 2^n\frac{\ws+m}{m}\,\varepsilon_r
				\;=\; v_r^{\mathrm{ch},+},
			\end{equation}
			with $V_r$ and $v_r^{\mathrm{ch},+}$ as in~\eqref{eq:Vk}--\eqref{eq:cheb_vpm}.
			Consequently $p_r$ certifies $\vol K \le v_r^{\mathrm{ch},+}$, and the optimal
			value $d_r^{\mathrm{sos}}$ of the order-$r$ Stokes
			dual~\eqref{eq:sos_dual_finite} is sandwiched as
			\begin{equation}\label{eq:cheb_dual_sandwich}
				\vol K \;\le\; 2^n d_r^{\mathrm{sos}}\;\le\; v_r^{\mathrm{ch},+}.
			\end{equation}
		\end{enumerate}
	\end{theorem}
	
	\begin{proof}
		\emph{(i)} By Proposition~\ref{prop:dual_certificate},
		$q^\star(t)=\tfrac{\ws+m}{m}\bigl(t-\tfrac{\ws}{\ws+m}\bigr)$, so
		$1-q^\star(t)=\tfrac{\ws+m}{m}(1-t)$ and the dual constraints
		of~\eqref{eq:stokes_dual} with $q=q^\star$ are exactly~\eqref{eq:cheb_dual_reduced};
		their equivalence with~\eqref{eq:cheb_dual_majorize} was noted above. By the very
		definition of the uniform error~\eqref{eq:eps_def} and the choice
		$\varepsilon_r\ge\sup_{[0,\bg]}\lvert f-f_r^{\mathrm{ch}}\rvert$, one has
		$f(t)-f_r^{\mathrm{ch}}(t)\le\varepsilon_r$ for all $t\in[0,\bg]$, that is
		\[
		f_r^{\mathrm{ch}}(t)+\varepsilon_r \;\ge\; f(t) \;=\;(1-t)_+\;\ge\;0,
		\qquad t\in[0,\bg].
		\]
		Multiplying by $\tfrac{\ws+m}{m}>0$ gives
		$p_r(t)\ge\tfrac{\ws+m}{m}(1-t)_+ = h(t)\ge 0$ on $[0,\bg]$, which
		is~\eqref{eq:cheb_dual_majorize}. This proves~(i).
		
		\emph{(ii)} The polynomials $p_r$ and $p_r-\tfrac{\ws+m}{m}(1-t)$ have degree
		$r$ and are nonnegative on $[0,\bg]$ resp.\ $[0,1]$ by~(i)
		(on $[0,1]$, $(1-t)_+=1-t$, so $p_r-\tfrac{\ws+m}{m}(1-t)=\tfrac{\ws+m}{m}
		(\varepsilon_r-(f-f_r^{\mathrm{ch}}))\ge 0$). The representations are then exactly
		Lemma~\ref{lem:markov_lukacs} applied on $[0,\bg]$ resp.\ $[0,1]$.
		
		\emph{(iii)} As $\nu$ is a probability measure, $\int_0^{\bg}d\nu=1$, so
		\[
		\int_0^{\bg}p_r\,d\nu
		=\frac{\ws+m}{m}\Bigl(\int_0^{\bg}f_r^{\mathrm{ch}}\,d\nu+\varepsilon_r\Bigr).
		\]
		By~\eqref{eq:qkC} and $\int_0^{\bg}T_k(2t/\bg-1)\,d\nu=y_k^{\mathrm{ch}}$ with
		$y_0=1$, we have
		$\int_0^{\bg}f_r^{\mathrm{ch}}\,d\nu=\tfrac{c_0}{2}y_0+\sum_{k=1}^r c_k\,
		y_k^{\mathrm{ch}}$, hence $2^n\tfrac{\ws+m}{m}\int_0^{\bg}f_r^{\mathrm{ch}}\,d\nu
		=V_r$ by~\eqref{eq:Vk}. Therefore
		$2^n\int_0^{\bg}p_r\,d\nu=V_r+2^n\tfrac{\ws+m}{m}\varepsilon_r
		=v_r^{\mathrm{ch},+}$, which is~\eqref{eq:cheb_dual_value}. Feasibility~(i) and
		weak duality for~\eqref{eq:stokes_dual} give
		$2^{-n}\vol K=\nu([0,1])\le\int_0^{\bg}p_r\,d\nu$, i.e.\
		$\vol K\le v_r^{\mathrm{ch},+}$. Finally, $p_r$ being feasible for the order-$r$
		program forces $d_r^{\mathrm {sos }}\le\int_0^{\bg}p_r\,d\nu=2^{-n}v_r^{\mathrm{ch},+}$,
		while $2^n d_r^{\mathrm{sos}}\ge\vol K$ by weak duality; this
		is~\eqref{eq:cheb_dual_sandwich}.
	\end{proof}
	
	\begin{remark}[Rate-optimality and suboptimality]\label{rem:cheb_dual_gap}
		Combining~\eqref{eq:cheb_dual_sandwich} with the \textsc{Cheb}
		bracket~\eqref{eq:cheb_vpm} and Proposition~\ref{prop:cheb_bracket},
		\[
		0\;\le\; 2^n d_r^{\mathrm{sos}}-\vol K
		\;\le\; v_r^{\mathrm{ch},+}-\vol K
		\;\le\; 2^{\,n+1}\frac{\ws+m}{m}\,\varepsilon_r
		\;\in\;O(1/r),
		\]
		so the closed-form certificate $p_r$ attains the same $O(1/r)$ rate as the
		(generally tighter) semidefinite optimum: it is rate-optimal within the order-$r$
		Stokes hierarchy, suboptimal only in the constant. By
		Remark~\ref{rem:bernstein} the truncated-series slack $\varepsilon_r$ exceeds the
		best degree-$r$ uniform majorant of the ReLU by only a universal factor, so the
		constant is itself near-optimal. For the running example
		($\bg=2$, $\tfrac{\ws+m}{m}=\tfrac74$) one finds $\varepsilon_{20}\approx 0.0152$
		and $v_{20}^{\mathrm{ch},+}\approx 3.602$, against the degree-$20$ semidefinite
		optimum $2^n d_{20}^{\mathrm S}\approx 3.506$ of Figure~\ref{fig:gibbs} and the
		exact value $\vol K=3.4961$; the bound $v_r^{\mathrm{ch},+}-\vol K$ halves as $r$
		doubles ($\approx 0.203,\,0.106,\,0.054$ at $r=10,20,40$).
	\end{remark}
	
	\begin{remark}[No semidefinite solve, fixed precision]\label{rem:cheb_dual_nosdp}
		The certificate $p_r$ in~\eqref{eq:cheb_dual_poly}, its Markov-Luk\'acs
		multipliers, and the bound~\eqref{eq:cheb_dual_value} are obtained in closed form
		from the Chebyshev coefficients~\eqref{eq:c0}--\eqref{eq:cj}, the moments
		$y_k^{\mathrm{ch}}$, and the error bound $\varepsilon_r$ --- with no semidefinite
		optimization and in fixed precision. This is the dual-certificate counterpart of
		the conditioning advantage that motivated \textsc{Cheb} in
		Section~\ref{sec:cheb}: solving~\eqref{eq:sos_dual_finite} directly would return
		the tighter value $2^n d_r^{\mathrm{sos}}$ at the cost of a semidefinite program in four Gram
		matrices, whereas post-processing yields a certified bound of the same order for
		free.
	\end{remark}
	
	\begin{remark}[The matching lower endpoint]\label{rem:cheb_dual_lower}
		Symmetrically, the \emph{minorant}
		$\underline p_r := \tfrac{\ws+m}{m}\bigl(f_r^{\mathrm{ch}}-\varepsilon_r\bigr)$
		satisfies $\underline p_r\le h$ on $[0,\bg]$, whence, by the weighted ReLU
		identity~\eqref{eq:relu} in the normalized form
		$\int_0^{\bg}h\,d\nu=2^{-n}\vol K$,
		\[
		2^n\!\int_0^{\bg}\underline p_r\,d\nu
		= V_r - 2^n\tfrac{\ws+m}{m}\varepsilon_r = v_r^{\mathrm{ch},-}\le\vol K .
		\]
		This recovers the lower \textsc{Cheb} endpoint. It is not a dual certificate: it
		is a direct minorant bound through the exact identity~\eqref{eq:relu}, the
		primal-side counterpart of $p_r$. The genuine primal moment certificate, with the
		Stokes relations imposed exactly, is described in the next section.
	\end{remark}
	
	\subsection{Closed-form dual certificates from \textsc{Gevp}}\label{sec:momsos_gevp_dual}
	
	The certificate of Section~\ref{sec:momsos_cheb_dual} froze the single analytic
	multiplier $q^\star$ of Proposition~\ref{prop:dual_certificate} and majorized the
	resulting ReLU; the price was the algebraic rate $O(1/r)$. We now let \emph{all}
	Stokes multipliers $q_1,\dots,q_{2r}$ act, which is exactly what \textsc{Gevp}
	does. The order-$r$ Stokes constraints~\eqref{eq:stokes_fun} pin down every moment
	of the restriction up to a single scalar: imposing $\int a_k\,d\sigma = 0$ for
	$k=1,\dots,2r$ forces the order-$r$ moment data of $\sigma$ to coincide with those
	of $\alpha\,\nu_w$, where $\alpha=\sigma([0,1])$ and $\nu_w$ is the reference
	measure with moments $z_k=\ws/(\ws+km)$ (Proposition~\ref{prop:nu_restriction}).
	Writing $y=(y_k)$ for the moments of $\nu$ and $z=(z_k)$ for those of $\nu_w$, the
	order-$r$ Stokes primal collapses to the scalar semidefinite program
	\begin{equation}\label{eq:gevp_primal}
		\alpha_r^+ \;:=\; \max\Bigl\{\alpha\ge 0 \ :\
		M_r(y)\succeq \alpha M_r(z),\ \ M_{r-1}(t(\bg-t)y) \succeq \alpha M_{r-1}(t(\bg-t)z)\Bigr\}
		\;\ge\; \nu([0,1]),
	\end{equation}
	where $M_r(\cdot)$ is the moment matrix and
	$M_{r-1}(t(\bg-t)\cdot)$ the localizing matrix. Both inequalities are
	affine in $\alpha$ with positive-definite $\alpha$-coefficient, so the maximal
	feasible $\alpha$ is the smallest generalized eigenvalue of the corresponding
	matrix pencils --- the generalized eigenvalue problem solved in
	Section~\ref{sec:gevp}, eq.~\eqref{eq:vr_pm} --- and $v_r^+=2^n\alpha_r^+$. We show
	that its dual is, again, a single explicit SOS polynomial, and that this
	polynomial certifies the geometric upper bound.
	
	The certificate rests on the following one-sided domination, a direct consequence
	of the Stokes structure $\nu|_{[0,1]}=\nu([0,1])\,\nu_w$ of
	Proposition~\ref{prop:nu_restriction}.
	
	\begin{lemma}[Stokes domination]\label{lem:stokes_domination}
		For every $\pi\in\R[t]$ with $\pi\ge 0$ on $[0,\bg]$,
		\begin{equation}\label{eq:stokes_domination}
			\int_0^{\bg}\pi\,d\nu \;\ge\; \nu([0,1])\int_0^{1}\pi\,d\nu_w .
		\end{equation}
		If moreover $\int_0^1\pi\,d\nu_w>0$, then
		$\vol K \le 2^n\,\dfrac{\int_0^{\bg}\pi\,d\nu}{\int_0^1\pi\,d\nu_w}$.
	\end{lemma}
	
	\begin{proof}
		By Proposition~\ref{prop:nu_restriction}, $\nu|_{[0,1]}=\nu([0,1])\,\nu_w$, and
		$\nu_w$ is supported on $[0,1]$. Hence
		\[
		\int_0^{\bg}\pi\,d\nu
		=\int_{[0,1]}\pi\,d\nu+\int_{(1,\bg]}\pi\,d\nu
		=\nu([0,1])\int_0^{1}\pi\,d\nu_w+\int_{(1,\bg]}\pi\,d\nu .
		\]
		As $\pi\ge 0$ on $[0,\bg]\supseteq(1,\bg]$ and $\nu\ge 0$, the last integral is
		nonnegative, proving~\eqref{eq:stokes_domination}. Dividing by
		$\int_0^1\pi\,d\nu_w>0$ and using $\vol K=2^n\nu([0,1])$ (eq.~\eqref{eq:vol_eq_nu})
		gives the bound.
	\end{proof}
	
	\begin{theorem}[\textsc{Gevp} dual certificate]\label{thm:gevp_dual}
		Let $r\ge 1$ and let $Q_r := \Sigma[t]_{2r}$ be the cone of sums of squares of degree at most $2r$,
		whose dual cone is exactly $\{y : M_r(y)\succeq 0\}$. Then:
		\begin{enumerate}
			\item[(i)] The value $\alpha_r^+$ of~\eqref{eq:gevp_primal} is the generalized
			Rayleigh minimum
			\begin{equation}\label{eq:gevp_rayleigh}
				\alpha_r^+
				=\min\Bigl\{\frac{\int_0^{\bg}\pi\,d\nu}{\int_0^1\pi\,d\nu_w}
				\ :\ \pi\in Q_r,\ \int_0^1\pi\,d\nu_w>0\Bigr\},
			\end{equation}
			attained by the single square $\pi_r^+ = (p_r^+)^2$, where $p_r^+ \in \R[t]_r$
			is the generalized eigenvector of \eqref{eq:vr_pm} associated with
			$\lmin\bigl(M_r(\nu), M_r(\nu_w)\bigr)$.
			\item[(ii)] \emph{(Admissibility.)} Every $\pi\in Q_r$ with
			$\int_0^1\pi\,d\nu_w>0$ certifies the upper bound
			$\vol K\le 2^n\int\pi\,d\nu/\int\pi\,d\nu_w$ through
			Lemma~\ref{lem:stokes_domination}; the squared eigenfunction $\pi_r^+$ is the
			tightest such degree-$2r$ certificate, and it attains
			\begin{equation}\label{eq:gevp_value}
				2^n\,\frac{\int_0^{\bg}\pi_r^+\,d\nu}{\int_0^1\pi_r^+\,d\nu_w}
				\;=\;2^n\alpha_r^+\;=\;v_r^+\;\ge\;\vol K .
			\end{equation}
			\item[(iii)] \emph{(Optimality and rate.)} $v_r^+=2^n\delta_r$, where $\delta_r$
			is the optimal value of the full Stokes dual~\eqref{eq:stokes_dual} at order
			$r$, with \emph{all} multipliers $q_1,\dots,q_{2r}$ free; and the bound
			converges geometrically,
			\begin{equation}\label{eq:gevp_rate}
				0\;\le\; v_r^+-\vol K \;=\;O(\gamma^{-2r}),
				\qquad \gamma=\frac{\sqrt{\bg}+1}{\sqrt{\bg}-1},
			\end{equation}
			by the analysis of Section~\ref{sec:gevp}.
		\end{enumerate}
	\end{theorem}
	
	\begin{proof}
		\emph{(i)} Parametrizing $\pi=\sigma_0+t(\bg-t)\sigma_1\in Q_r$ by the Gram
		matrices $G_0\succeq 0$, $G_1\succeq 0$ of $\sigma_0,\sigma_1$, one has
		$\int\pi\,d\nu=\langle G_0,M_r(y)\rangle+\langle G_1,M_{r-1}(t(\bg -t)y)\rangle$
		and likewise for $\nu_w$ with $z$ in place of $y$. Thus the right-hand side
		of~\eqref{eq:gevp_rayleigh} is the minimum of a ratio of two linear forms in
		$(G_0,G_1)\succeq 0$; its value is the largest $\alpha$ for which
		$\langle G, M_r(y - \alpha z)\rangle \ge 0$ holds for every $G \succeq 0$, that is,
		the largest $\alpha$ with $M_r(y) - \alpha M_r(z) \succeq 0$, which
		is~\eqref{eq:gevp_primal}. The
		minimum is attained on an extreme ray, giving an optimizer $\pi_r^+\in Q_r$;
		being nonnegative on $[0,\bg]$, it admits the stated Markov-Luk\'acs form by
		Lemma~\ref{lem:markov_lukacs}, with the squares read off the generalized
		eigenvector of~\eqref{eq:vr_pm}.
		
		\emph{(ii)} Immediate from Lemma~\ref{lem:stokes_domination} applied to each
		$\pi\in Q_r$, and from~\eqref{eq:gevp_rayleigh} for the optimal $\pi_r^+$, which
		gives $2^n\alpha_r^+=v_r^+\ge\vol K$.
		
		\emph{(iii)} Problem~\eqref{eq:gevp_primal} is the order-$r$ Stokes primal:
		eliminating $\sigma$ via the order-$r$ Stokes
		constraints~\eqref{eq:stokes_fun} (which force $M_r(\sigma)=\alpha M_r(z)$ and
		$M_{r-1}(t(\bg -t)\sigma)=\alpha M_{r-1}(t(\bg -t)z)$, hence reduce the constraints
		$\sigma\ge 0$ on $[0,1]$ to $\alpha\ge 0$) leaves the slack constraints
		$\nu-\alpha\nu_w\ge 0$ on $[0,\bg]$, i.e.~\eqref{eq:gevp_primal}. Slater's
		condition holds (take $\alpha$ slightly below $\nu([0,1])$), so strong duality
		identifies its value with that of its Lagrangian dual, which is precisely the
		order-$r$ instance of~\eqref{eq:stokes_dual} with all multipliers free; thus
		$\alpha_r^+=\delta_r$ and $v_r^+=2^n\delta_r$. The geometric
		rate~\eqref{eq:gevp_rate} is established in Section~\ref{sec:gevp}.
	\end{proof}
	
	\begin{remark}[Suboptimal admissible certificates]\label{rem:gevp_suboptimal}
		Lemma~\ref{lem:stokes_domination} turns \emph{any} polynomial $\pi$ nonnegative
		on $[0,\bg]$ with positive reference mass into a dual-admissible certificate
		$\vol K\le 2^n\int\pi\,d\nu/\int\pi\,d\nu_w$; \textsc{Gevp} merely selects the
		minimizer $\pi_r^+$ over $Q_r$. Cheaper, suboptimal choices are legitimate: a
		lower-degree square, a fixed test polynomial, or the squared partial sum
		$(f^{\mathrm{ch}}_{\lfloor r/2\rfloor})^2$ built from the \textsc{Cheb}
		coefficients all give valid --- if looser --- upper bounds, without any
		eigenvalue computation. The certificate $\pi_r^+$ is optimal within $Q_r$ but
		still suboptimal with respect to $\vol K$ at finite order.
	\end{remark}
	
	\begin{remark}[Companion lower bound]\label{rem:gevp_lower}
		The lower bracket endpoint $v_r^-\le\vol K$ of Section~\ref{sec:gevp},
		eq.~\eqref{eq:vr_pm}, is certified symmetrically. Since
		$\int_{(1,\bg]}\pi\,d\nu=\int\pi\,d\nu-\nu([0,1])\int\pi\,d\nu_w$ for every
		$\pi$, a polynomial $\psi$ nonnegative on $[1,\bg]\supseteq\supp(\nu|_{(1,\bg]})$
		with \emph{negative} reference mass $\int\psi\,d\nu_w<0$ yields, after dividing
		the inequality $\int\psi\,d\nu-\nu([0,1])\int\psi\,d\nu_w\ge 0$ by that negative
		quantity, the reverse bound $\vol K\ge 2^n\int\psi\,d\nu/\int\psi\,d\nu_w$. The
		reference localizer being sign-indefinite on $[0,1]$, this lower problem is not a
		standard definite pencil; we refer to Section~\ref{sec:gevp} for its precise
		generalized-eigenvalue form and its (also geometric) convergence. Unlike
		$\pi_r^+$, this is a primal-side certificate.
	\end{remark}
	
	\begin{remark}[\textsc{Cheb} versus \textsc{Gevp}]\label{rem:cheb_vs_gevp}
		The two post-processings certify the volume from above by a single
		SOS polynomial, but through complementary uses of the Stokes
		multipliers. \textsc{Cheb} (Theorem~\ref{thm:cheb_dual}) keeps only $q^\star$ and
		majorizes the ReLU, giving the closed-form  bound $v_r^{\mathrm{ch},+}$
		at the algebraic rate $O(1/r)$. \textsc{Gevp} activates all $q_k$ and returns the
		full Stokes-dual value $v_r^+$ via one generalized eigenproblem, with the
		geometric rate~\eqref{eq:gevp_rate}; its certificate is the squared
		eigenfunction $\pi_r^+\in Q_r$. For the running example
		($\bg=2$, $\gamma=3+2\sqrt2$) the gap $v_r^+-\vol K$ contracts by the factor
		$\gamma^{-2}\approx 0.029$ per order ($\approx 2.9\!\times\!10^{-2},\,
		1.2\!\times\!10^{-3},\,4.6\!\times\!10^{-5},\,1.6\!\times\!10^{-6}$ at
		$r=1,2,3,4$), reaching machine precision near $r=6$: at the degree-$20$ budget of
		Figure~\ref{fig:gibbs}, where the fixed-multiplier majorant gives $3.602$ and the
		majorization semidefinite program gives $3.506$, \textsc{Gevp} returns $\vol K=3.4961$ to full
		accuracy.
	\end{remark}
	
	\subsection{Lifting to the original moment-SOS hierarchy}\label{sec:momsos_lift}
	
	The certificates of Sections~\ref{sec:momsos_cheb_dual}
	and~\ref{sec:momsos_gevp_dual} are univariate. We now check that they lift to the original hierarchy of \cite{hls2009} on $B = [-1,1]^n$, so that the algebraic and
	geometric rates are rates for that hierarchy as well. Write
	\[
	Q(B) \;:=\; \Bigl\{ \textstyle\sum_{i=0}^n \theta_i\, b_i \ :\
	\theta_i \in \Sigma[x],\ b_0 = 1,\ b_i = 1-x_i^2 \Bigr\}
	\]
	for the quadratic module generated by the box constraints, which is Archimedean.
	
	\begin{proposition}[Composition]\label{prop:lift}
		Let $\pi \in \R[t]_{2r}$ be non-negative on $[0,\bg]$ and let $\bg' > \bg$. Then
		\[
		\pi \circ g \;\in\; Q(B) \;+\; \Sigma[x]\cdot g ,
		\]
		and more precisely there are $d_0 \in \N$ depending on $g$, $\bg'$ and $n$ but not
		on $r$, and sums of squares $\theta_i$, such that $\pi \circ g$ admits a
		representation in $Q(B) + \Sigma[x]\cdot g$ of degree at most $2rm + d_0$.
	\end{proposition}
	
	\begin{proof}
		By Lemma~\ref{lem:markov_lukacs} on $[0,\bg']$, $\pi = \sigma_0 + t(\bg'-t)\sigma_1$
		with $\sigma_0, \sigma_1$ sums of squares of degree $\le 2r$ and $\le 2r-2$,
		whence
		\[
		\pi \circ g \;=\; \sigma_0(g) \;+\; g\,\bigl(\bg' - g\bigr)\,\sigma_1(g).
		\]
		The compositions $\sigma_i(g)$ are sums of squares in $x$, of degree $\le 2rm$.
		It remains to place $\bg' - g$ in $Q(B)$. Since $\bg' > \bg = \max_B g$, the
		polynomial $\bg' - g$ is strictly positive on $B$, and $Q(B)$ is Archimedean, so
		Putinar's Positivstellensatz gives $\bg' - g \in Q(B)$ with some representation
		degree $d_0$, a quantity fixed by $g$, $\bg'$ and $n$ and independent of $r$.
		Multiplying that representation by the sum of squares $g\,\sigma_1(g)$ --- note
		$g \ge 0$ on $\R^n$, and $g$ itself need not be a sum of squares, whence the term
		$\Sigma[x]\cdot g$ --- yields the claim, with total degree at most
		$2rm + d_0$.
	\end{proof}
	
	\begin{corollary}[Rates for the original hierarchy]\label{cor:lift_rates}
		Let $u_d$ denote the optimal value of the degree-$d$ SOS relaxation of the original
		volume hierarchy of \cite{hls2009}, augmented with the Stokes constraints of
		\cite{stokesgibbs}, on the box $B$. Then, with $d = 2rm + d_0$,
		\[
		0 \;\le\; u_d - \vol K \;\le\; v_r^+ - \vol K \;=\; O\bigl(\gamma^{-2r}\bigr)
		\;=\; O\bigl(\gamma^{-(d-d_0)/m}\bigr),
		\]
		and likewise $u_d - \vol K = O(m/d)$ using the closed-form certificate $p_r$ of
		Theorem~\ref{thm:cheb_dual} in place of $\pi_r^+$.
	\end{corollary}
	
	\begin{proof}
		Apply Proposition~\ref{prop:lift} to the optimal $\pi_r^+ \in Q_r$ of
		Theorem~\ref{thm:gevp_dual} (resp.\ to $p_r$ of Theorem~\ref{thm:cheb_dual}). The
		resulting $\pi_r^+ \circ g$ is a feasible point of the degree-$d$ original
		relaxation whose objective value is, by the change of variables \eqref{eq:cov},
		the univariate value already bounded in Theorem~\ref{thm:gevp_dual}(iii)
		(resp.\ Theorem~\ref{thm:cheb_dual}(iii)). Optimality of $u_d$ gives the first
		inequality.
	\end{proof}
	
	Corollary~\ref{cor:lift_rates} is the transfer announced in
	Section~\ref{sec:contrib}: it replaces the algebraic rate $O(d^{-1/(2.5 n \bar L)})$
	of \cite{schlosserlazarev} by a geometric one, at the cost of the quasi-homogeneity
	hypothesis. The exponent degrades with the degree $m$ of $g$, but --- in contrast
	with \cite{kordahenrion,schlosserlazarev} --- not with the ambient dimension $n$,
	which enters only the additive constant $d_0$.
	
	\section{Numerical illustrations}\label{sec:numerics}
	
	Two families of experiments are reported:
	\begin{enumerate}[leftmargin=*,itemsep=0pt]
		\item \textbf{Non-convex homogeneous quartic and sextic}:
		$g_m = 2^m\bigl(x_1^m + x_2^m - c\,x_1^{m/2}x_2^{m/2}\bigr)$ for $m = 4, 6$
		($c = 1.925$) on $[-1,1]^2$, with reference values from polar integration;
		isolates the dependence of both algorithms on $\bg$.
		\item \textbf{Balls}: $g(x) = \sum_i x_i^p$, $p$ even, $m = p$,
		$w = (1,\dots,1)$, $\bg = n$, exact value
		$\vol K = 2^n\Gamma(1+1/p)^n/\Gamma(1+n/p)$; the grid
		$p \in \{2,4,6\}$, $n \in \{2,3,4\}$ at fixed order, then $p = 2$ with $n$ up to
		$64$ under a wall-clock budget. Separability makes the moments free
		(Proposition~\ref{prop:blockmoments}), so $n$ enters only through $\bg$.
	\end{enumerate}
	
	The computations were performed under Linux using Python~3.13.5, NumPy~2.3.5, SciPy~1.17.0, \texttt{mpmath}~1.3.0, and Matplotlib~3.10.8; the \textsc{Cheb} computations were carried out in IEEE double precision, whereas the \textsc{Gevp} computations used adaptive multiprecision arithmetic through \texttt{mpmath}.
	
	\subsection{Two non-convex examples}
	\label{sec:nonconvex}
	
	\newcommand{\bv}{b}                       
	\newcommand{\Kell}{\operatorname{K}}      
	To illustrate the pushforward construction on non-convex sets we take the homogeneous
	quartic and sextic of \cite[Fig.~1]{lasserre2013},
	\begin{equation}\label{eq:nc_forms}
		G_4(x)=x_1^4+x_2^4-c\,x_1^2x_2^2,
		\qquad
		G_6(x)=x_1^6+x_2^6-c\,x_1^3x_2^3,
		\qquad c=1.925 .
	\end{equation}
	Writing $u=x_1^{m/2},\,v=x_2^{m/2}$, both equal the binary quadratic $u^2-c\,uv+v^2$,
	positive definite for $c\in(0,2)$; hence $G_4,G_6>0$ on $\R^2\setminus\{0\}$
	(Assumption~\ref{ass:pos}), and the unit sublevel sets are compact and markedly
	non-convex, see Figure~\ref{fig:nc_sublevel}. They reach $(2-c)^{-1/m}\approx1.91$ and
	$1.54$ along a diagonal, hence are \emph{not} contained in $B=[-1,1]^2$. We therefore
	work with the rescaled forms $g_m:=2^m G_m$, whose unit sublevel set
	$K=\{g_m\le1\}=\tfrac12\{G_m\le1\}$ sits strictly inside $B$, so that
	\begin{equation}\label{eq:nc_volrel}
		\vol K=\vol\{g_m\le1\}=2^{-n}\vol\{G_m\le1\},
		\qquad
		\bg_4=\max_B g_4=2^4=16,
		\quad
		\bg_6=2^6(2+c)=251.2 .
	\end{equation}
	
	\begin{figure}[ht]
		\centering
		\includegraphics[width=0.49\linewidth]{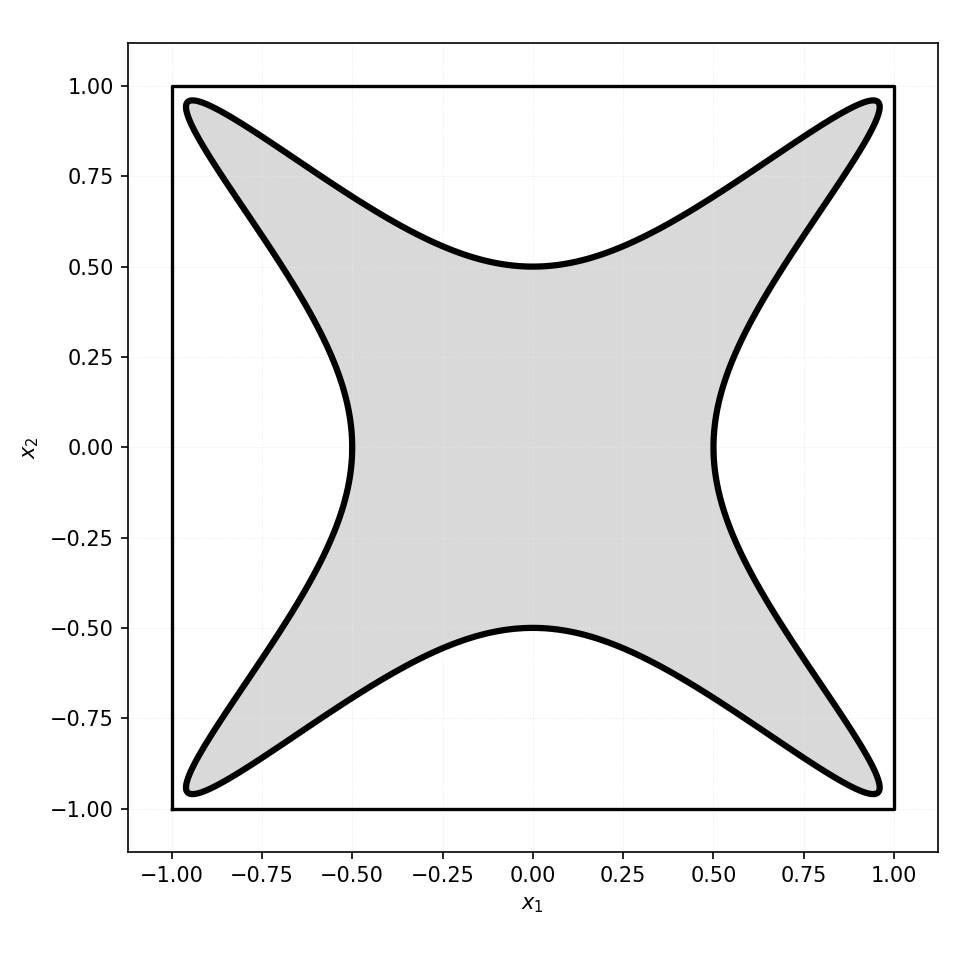}\hfill
		\includegraphics[width=0.49\linewidth]{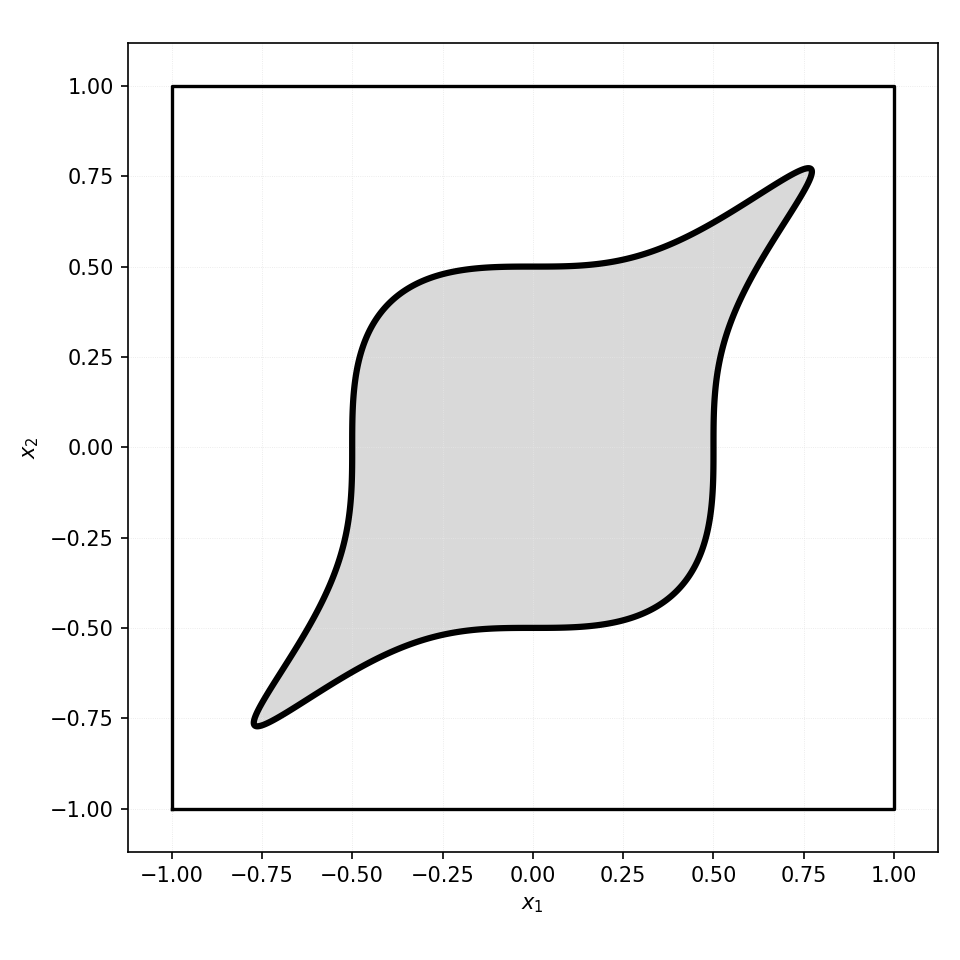}
		\caption{The unit sublevel sets $K=\{g_m\le1\}=\tfrac12\{G_m\le1\}\subset B=[-1,1]^2$
			of the rescaled forms, whose volumes are approximated below: the quartic (left) and
			the sextic (right), both strictly inside the bounding box.}
		\label{fig:nc_sublevel}
	\end{figure}
	
	\subsubsection{The pushforward density in polar form}\label{sec:nc_radial}
	
	Let $\mu$ be the uniform probability measure on $B$ and $\nu=(g_m)_\#\mu$. Write
	$x=r(\cos\theta,\sin\theta)$. With the \emph{circle profile}
	$\phi(\theta):=g(\cos\theta,\sin\theta)>0$, the \emph{box radius}
	$r_B(\theta):=1/\max(|\cos\theta|,|\sin\theta|)$ and the \emph{boundary value}
	$\bv(\theta):=r_B(\theta)^m\phi(\theta)=g|_{\partial B}$, the density takes a clean polar form.
	
	\begin{lemma}[Radial density]\label{lem:nc_radial}
		The pushforward measure $\nu=g_\#\mu$ is
		absolutely continuous on $(0,\bg)$ with density
		\begin{equation}\label{eq:nc_density}
			\rho(t)=\frac{t^{\,2/m-1}}{4m}\,\beta(t),
			\qquad
			\beta(t):=\int_{\{\theta:\,\bv(\theta)>t\}}\phi(\theta)^{-2/m}\,d\theta .
		\end{equation}
		For $t\le\bv_{\min}:=\min_\theta\bv$ the truncation is inactive, $\beta(t)\equiv
		\beta_0:=\int_0^{2\pi}\phi^{-2/m}\,d\theta$, and $\rho$ is the pure power law
		$\rho(t)=\tfrac{\beta_0}{4m}\,t^{2/m-1}$.
	\end{lemma}
	
	Indeed $\{g=t\}$ has polar radius $(t/\phi)^{1/m}$, so $\{g\le t\}\cap B$ is the star domain of
	radius $\min((t/\phi)^{1/m},r_B)$; integrating $\tfrac12 r^2$ over $\theta$ and differentiating
	gives~\eqref{eq:nc_density}, a direction contributing only while $(t/\phi)^{1/m}\le r_B$, i.e.\
	$\bv(\theta)>t$. The homogeneous power $t^{2/m-1}$ is carried by the prefactor; the whole effect
	of the box sits in the monotone angular mass $\beta$, decreasing from $\beta_0$ to $0$ as
	$\{g=t\}$ is clipped by and finally leaves $B$.
	
	\begin{corollary}\label{cor:nc_beta}
		On $[0,\bv_{\min}]$ the measure $\nu$ is $\mathrm{Beta}(\tfrac2m,1)$ up to scale. Hence when
		$\bv_{\min}>1$ (as for $g_m$), the restriction
		$\nu|_{[0,1]}=2^{-n}\vol K\,\nu_w$ is \emph{exactly} the reference measure $\nu_w$ of
		Lemma~\ref{lem:muw}, in agreement with Proposition~\ref{prop:nu_restriction}; in
		particular $\vol K=\tfrac12\beta_0$ and
		$\nu([0,1])=\tfrac18\beta_0=2^{-n}\vol K$, the dark-grey areas in
		Figures~\ref{fig:nc_density_g4} and~\ref{fig:nc_density_g6}.
	\end{corollary}
	
	\paragraph{Singularities.}
	By Lemma~\ref{lem:nc_radial} the only singularities of $\rho$ on $(0,\bg)$ come from $\beta$,
	which changes regularity exactly at the critical values of $\bv=g|_{\partial B}$ (at a regular
	value $\beta'(t)=-\sum_{\bv(\theta)=t}\phi^{-2/m}/|\bv'(\theta)|$, by the angular coarea formula).
	For a form on the square the local types are four: a nondegenerate edge tangency (interior
	extremum of $\bv$) gives a \emph{square-root cusp}; a box vertex, where $\bv$ has a corner,
	gives a \emph{kink}; a degenerate inflection ($\bv'=\bv''=0\ne\bv'''$) gives a
	\emph{cube-root cusp}; and at $t=\bg$ the density vanishes like $\sqrt{\bg-t}$ (smooth
	maximiser) or like $\bg-t$ (corner maximiser). For homogeneous $g$ the only interior critical
	point is the origin, so apart from the power-law blow-up at $0$ the entire singular set of
	$\nu$ is generated by $g|_{\partial B}$.
	
	\subsubsection{The quartic}\label{sec:nc_g4}
	
	Here $\phi_4(\theta)=2^4\bigl(1-k^2\sin^2 2\theta\bigr)$ with $k=\tfrac12\sqrt{2+c}\in(0,1)$,
	so the angular mass is closed-form,
	$\beta_0=\int_0^{2\pi}\phi_4^{-1/2}\,d\theta=2^{-2}\cdot4\Kell(k)=\Kell(k)$ where $\Kell$ the
	complete elliptic integral of the first kind, and
	\begin{equation}\label{eq:nc_rho4}
		\rho(t)\sim\tfrac{\Kell(k)}{16}\,t^{-1/2}\ (t\to0^+),
		\quad
		\vol K=\vol\{g_4\le1\}=\tfrac12\Kell(k)=1.69291,
		\quad
		\nu([0,1])=\tfrac18\Kell(k)=0.42323
	\end{equation}
	As $c\uparrow2$, $k\uparrow1$ and $\Kell(k)\to\infty$, the diagonal of the sublevel set
	escaping to infinity. The boundary profile
	$\bv_4(\theta)=g_4(1,s)=2^4(1+s^4-cs^2)$ ($s=\tan\theta$) has critical values
	\begin{center}
		\renewcommand{\arraystretch}{1.2}
		\begin{tabular}{llll}
			\toprule
			boundary point(s) & value of $g_4$ & type of $\bv_4$ & mult.\\
			\midrule
			axis midpoints $(\pm1,0),(0,\pm1)$ & $2^4=16=\bg$ & nondeg.\ max & $4$\\
			edge tangencies $s=\pm\sqrt{c/2}$ & $2^4(1-c^2/4)=1.1775$ & nondeg.\ min & $8$\\
			vertices $(\pm1,\pm1)$ & $2^4(2-c)=1.2$ & corner & $4$\\
			\bottomrule
		\end{tabular}
	\end{center}
	so on $(0,\bg=16)$ the density $\rho_4(t)=\tfrac{t^{-1/2}}{16}\beta_4(t)$ has a square-root cusp
	at $t_1=1.1775$, a kink at $t_2=1.2$ --- nearly coincident, separated by
	$2^4(1-c/2)^2=0.0225$ --- and a square-root vanishing at $\bg=16$ ($\beta_4\equiv\Kell(k)$
	for $t\le t_1$, a complete-minus-incomplete elliptic integral above). See
	Figure~\ref{fig:nc_density_g4}.
	
	\begin{figure}[ht]
		\centering
		\includegraphics[width=0.8\linewidth]{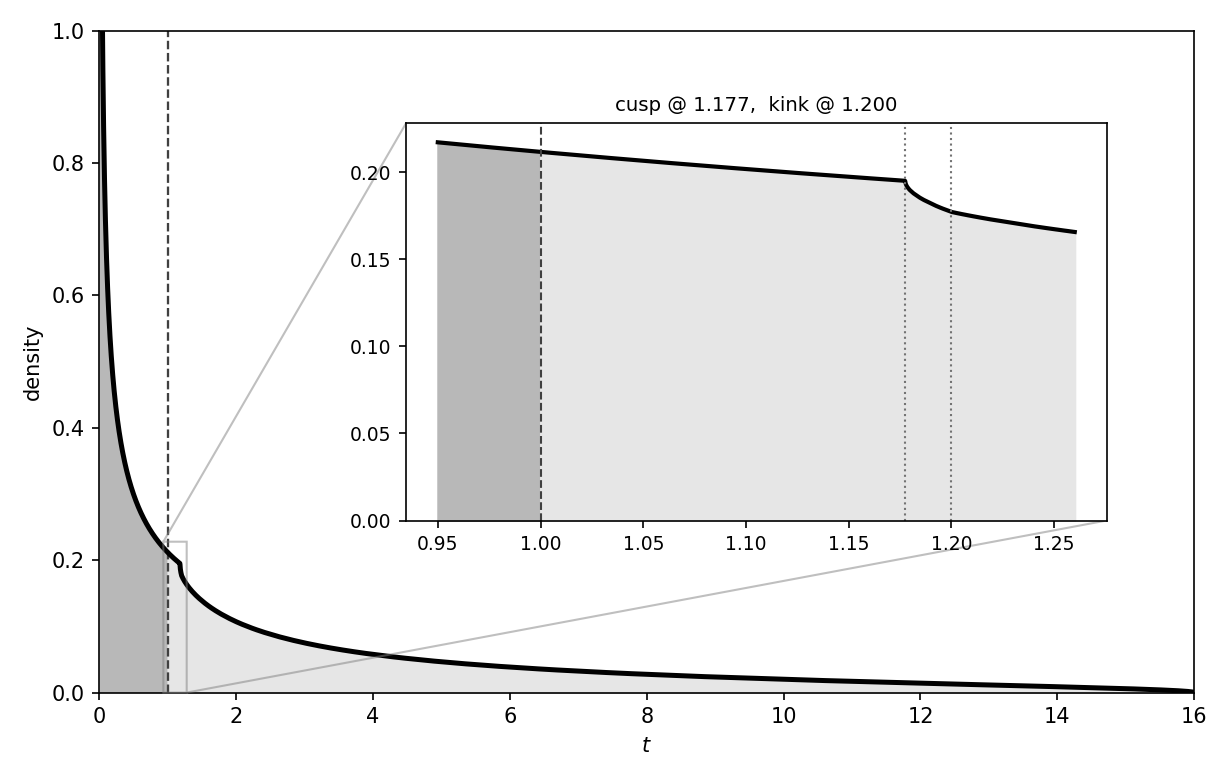}
		\caption{Pushforward density $\rho_4$ of $\nu=(g_4)_\#\mu$ on $[0,\bg_4]=[0,16]$. The
			dark-grey area over $[0,1]$ equals $\nu([0,1])=2^{-n}\vol K=0.4232$; the light-grey
			area over $(1,\bg)$ is the remaining mass. The zoom window resolves the $t=1$
			boundary together with the near-coincident square-root cusp ($t_1=1.1775$) and kink
			($t_2=1.2$); the density blows up like $\tfrac{1}{16}\Kell(k)\,t^{-1/2}$ at $0$ and
			vanishes like $\sqrt{\bg-t}$ at $\bg=16$.}
		\label{fig:nc_density_g4}
	\end{figure}
	
	\subsubsection{The sextic}\label{sec:nc_g6}
	
	Now $\phi_6(\theta)=2^6\bigl(1-\tfrac34\sin^2 2\theta-\tfrac c8\sin^3 2\theta\bigr)$; the odd
	term breaks the reflection symmetry, and
	$\beta_0=\int_0^{2\pi}\phi_6^{-1/3}\,d\theta=2^{-2}\cdot8.67874=2.16969$ is a hyperelliptic
	period with no elementary closed form, giving
	\[
	\rho(t)\sim0.09040\,t^{-2/3}\ (t\to0^+),
	\quad
	\vol K=\vol\{g_6\le1\}=\tfrac12\beta_0=1.08484,
	\quad
	\nu([0,1])=\tfrac18\beta_0=0.27121.
	\]
	The boundary profile
	$\bv_6(\theta)=g_6(1,s)=2^6(1+s^6-cs^3)$ has $\bv_6'\propto s^2(2s^3-c)$, so the axis point
	$s=0$ is a degenerate inflection ($\bv_6'=\bv_6''=0\ne\bv_6'''$), and the two corners of each
	edge become inequivalent:
	\begin{center}
		\renewcommand{\arraystretch}{1.2}\setlength{\tabcolsep}{4pt}
		\begin{tabular}{llll}
			\toprule
			boundary point(s) & value of $g_6$ & type of $\bv_6$ & mult.\\
			\midrule
			anti-diag.\ vertices $(1,-1),(-1,1)$ & $2^6(2+c)=251.2=\bg$ & corner (max) & $2$\\
			axis midpoints $(\pm1,0),(0,\pm1)$ & $2^6=64$ & degenerate inflection & $4$\\
			diagonal vertices $(1,1),(-1,-1)$ & $2^6(2-c)=4.8$ & corner & $2$\\
			edge tangencies $s=(c/2)^{1/3}$ & $2^6(1-c^2/4)=4.71$ & nondeg.\ min & $4$\\
			\bottomrule
		\end{tabular}
	\end{center}
	Hence $\rho_6(t)=\tfrac{t^{-2/3}}{24}\beta_6(t)$ has a square-root cusp at $t_1=4.71$, a kink at
	$t_2=4.8$, a \emph{cube-root cusp} at $t_3=2^6=64$ (the degenerate inflection --- the genuinely
	new singularity, absent for $g_4$, with a two-sided vertical tangent), and a \emph{linear}
	vanishing at $\bg=251.2$ (a corner maximiser, versus the square-root vanishing of $g_4$). See
	Figure~\ref{fig:nc_density_g6}.
	
	\begin{figure}[ht]
		\centering
		\includegraphics[width=0.8\linewidth]{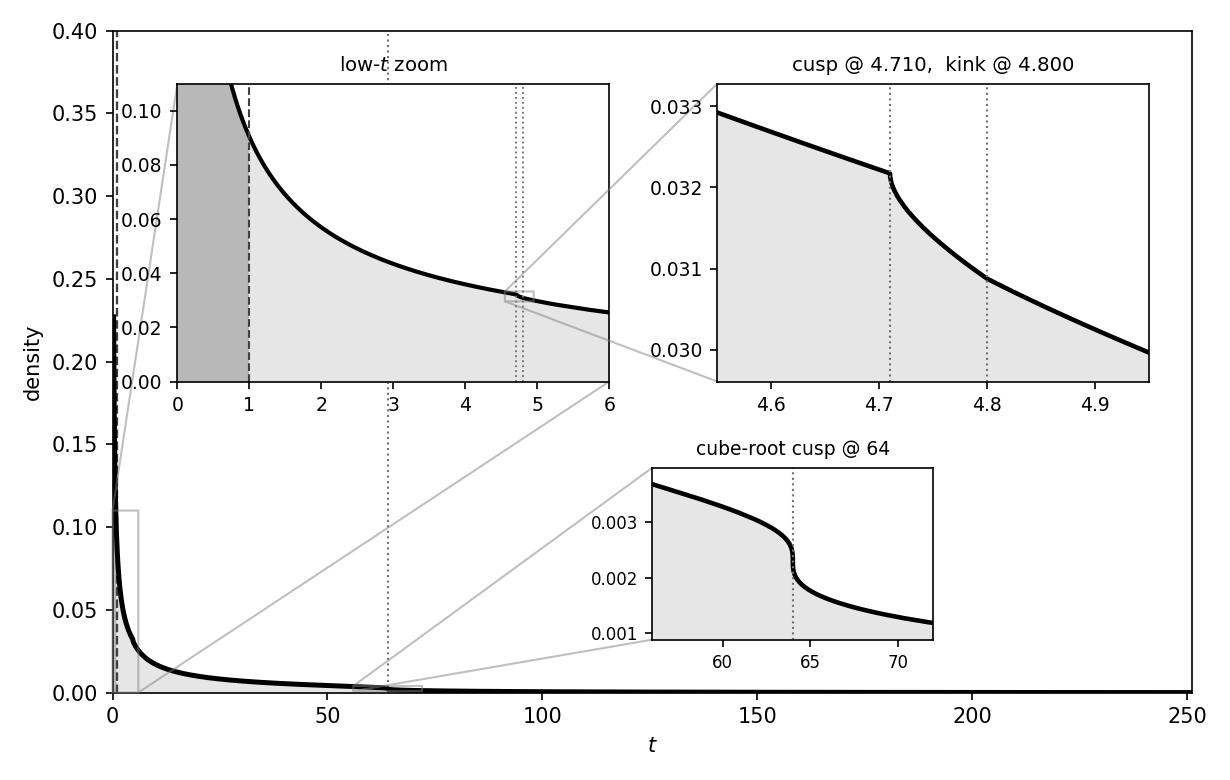}
		\caption{Pushforward density $\rho_6$ of $\nu=(g_6)_\#\mu$ on $[0,\bg_6]=[0,251.2]$.
			Dark grey over $[0,1]$: $\nu([0,1])=2^{-n}\vol K=0.2712$; light grey over
			$(1,\bg)$: the remaining mass. Three chained zoom windows: a low-$t$ window with the
			$[0,1]$ area and the $0.0904\,t^{-2/3}$ blow-up; a window resolving the square-root
			cusp ($t_1=4.71$) and the kink ($t_2=4.8$); and a window at the cube-root cusp
			$t_3=2^6=64$, whose two-sided vertical tangent is the degenerate singularity absent
			for $g_4$. The density vanishes linearly at $\bg=251.2$.}
		\label{fig:nc_density_g6}
	\end{figure}
	
	\subsubsection{Volume approximation}\label{sec:nc_numerics}
	
	\begin{remark}[The rate is governed by $\bg$]\label{rem:nc_rates}
		Both forms are homogeneous ($w=(1,1)$), so the machinery of
		Sections~\ref{sec:pf_stokes}--\ref{sec:gevp} applies to $g_m$. By
		Corollary~\ref{cor:nc_beta} the restriction $\nu|_{[0,1]}$ is exactly the Beta reference
		($\bv_{\min}=1.1775$ and $4.71$, both $>1$), and all the singularities above lie in the
		residual support $[1,\bg]$. The \textsc{Gevp} rate of
		Theorem~\ref{thm:rate} depends on $\bg$ alone, through
		$\gamma=(\sqrt{\bg}+1)/(\sqrt{\bg}-1)$, not on that singularity structure: $\gamma_4=\tfrac53$
		($\gamma_4^{-2}=0.360$) and $\gamma_6=1.135$ ($\gamma_6^{-2}=0.777$); the \textsc{Cheb} error
		is $\varepsilon_r\sim\sqrt{\bg-1}/(\pi r)$, i.e.\ $\approx1.23/r$ and $\approx5.03/r$. The
		examples thus isolate the \emph{$\bg$-dependence}: the loose box inflates $\bg$, slowing
		\textsc{Gevp} ($\gamma\downarrow1$) and widening the \textsc{Cheb} bracket, separately from the
		loss of analyticity a non-quasi-homogeneous $g$ would add.
	\end{remark}
	
	We run both algorithms on $g_4,g_6$ against the exact volumes $\vol K_4=1.6929082$ and
	$\vol K_6=1.0848422$. The moments $y_k=2^{-n}\int_B g_m^k\,dx$ are computed exactly (rational
	multinomial expansion), the Chebyshev moments by tensor Gauss-Legendre, and the \textsc{Gevp}
	generalized eigenvalues in multiprecision.
	
	\textsc{Cheb} (Figure~\ref{fig:nc_cheb}): the central estimate converges quickly, while the
	certified half-width $\tfrac{2^n(\ws+m)}{m}\varepsilon_r$ of the bracket \eqref{eq:cheb_vpm} decays at the guaranteed rate $O(1/r)$
	in plain double precision.
	
	\textsc{Gevp} (Figure~\ref{fig:nc_gevp}): the two generalized eigenvalues~\eqref{eq:vr_pm}
	give a monotone bracket whose gap closes geometrically at the rate $\gamma^{-2r}$, reaching
	width $1.3\cdot10^{-12}$ at $r=30$ for $g_4$ and $9.5\cdot10^{-6}$ at $r=55$ for $g_6$, at a
	working precision growing linearly in $r$.
	
	Table~\ref{tab:nc_compare} contrasts the two: the geometric \textsc{Gevp} bracket is far
	tighter at equal order but precision-hungry, whereas \textsc{Cheb} stays in double precision;
	both degrade as $\bg$ grows, \textsc{Gevp} through its rate and \textsc{Cheb} through its
	constant.
	
	\begin{figure}[ht]
		\centering
		\includegraphics[width=0.49\linewidth]{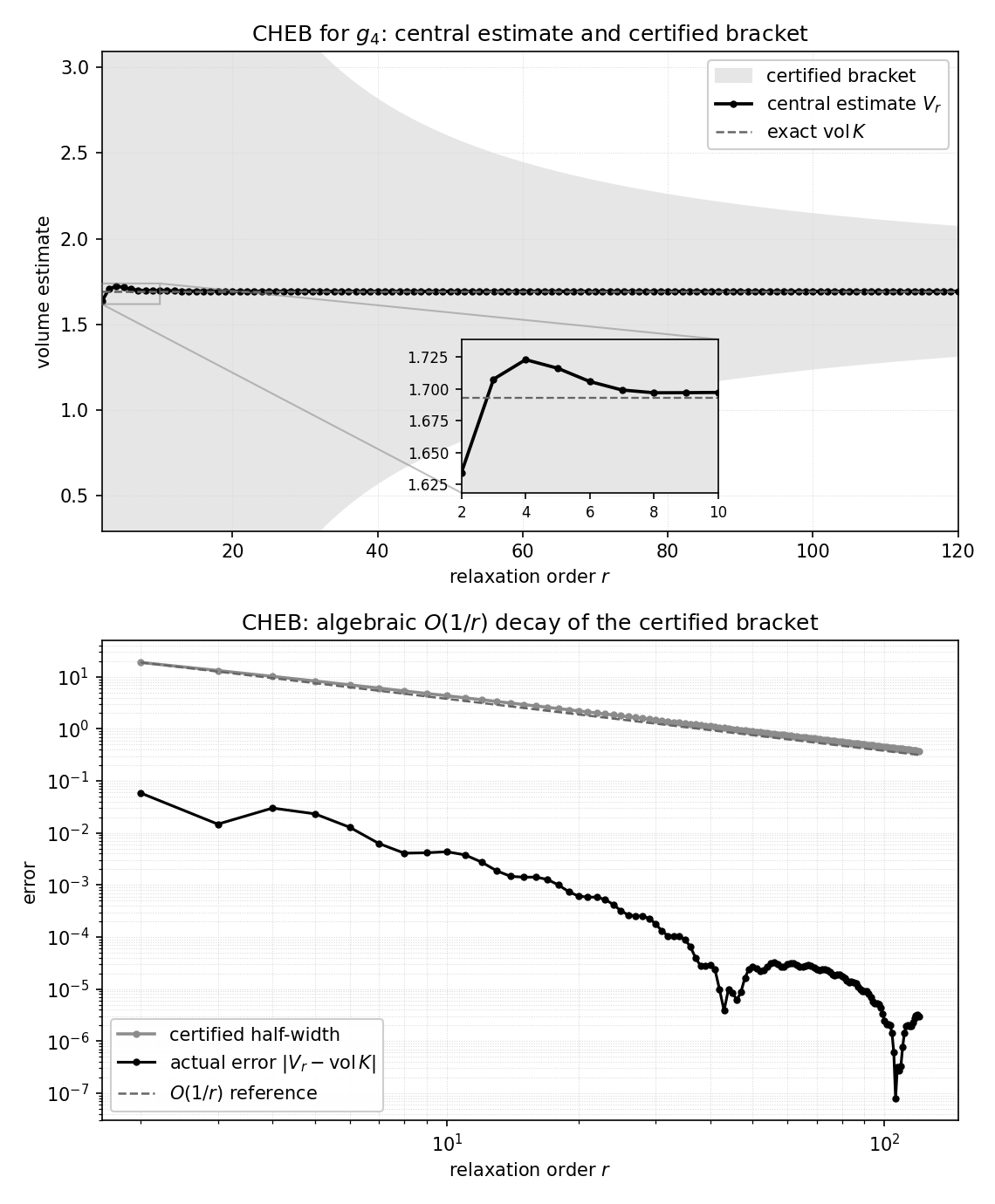}\hfill
		\includegraphics[width=0.49\linewidth]{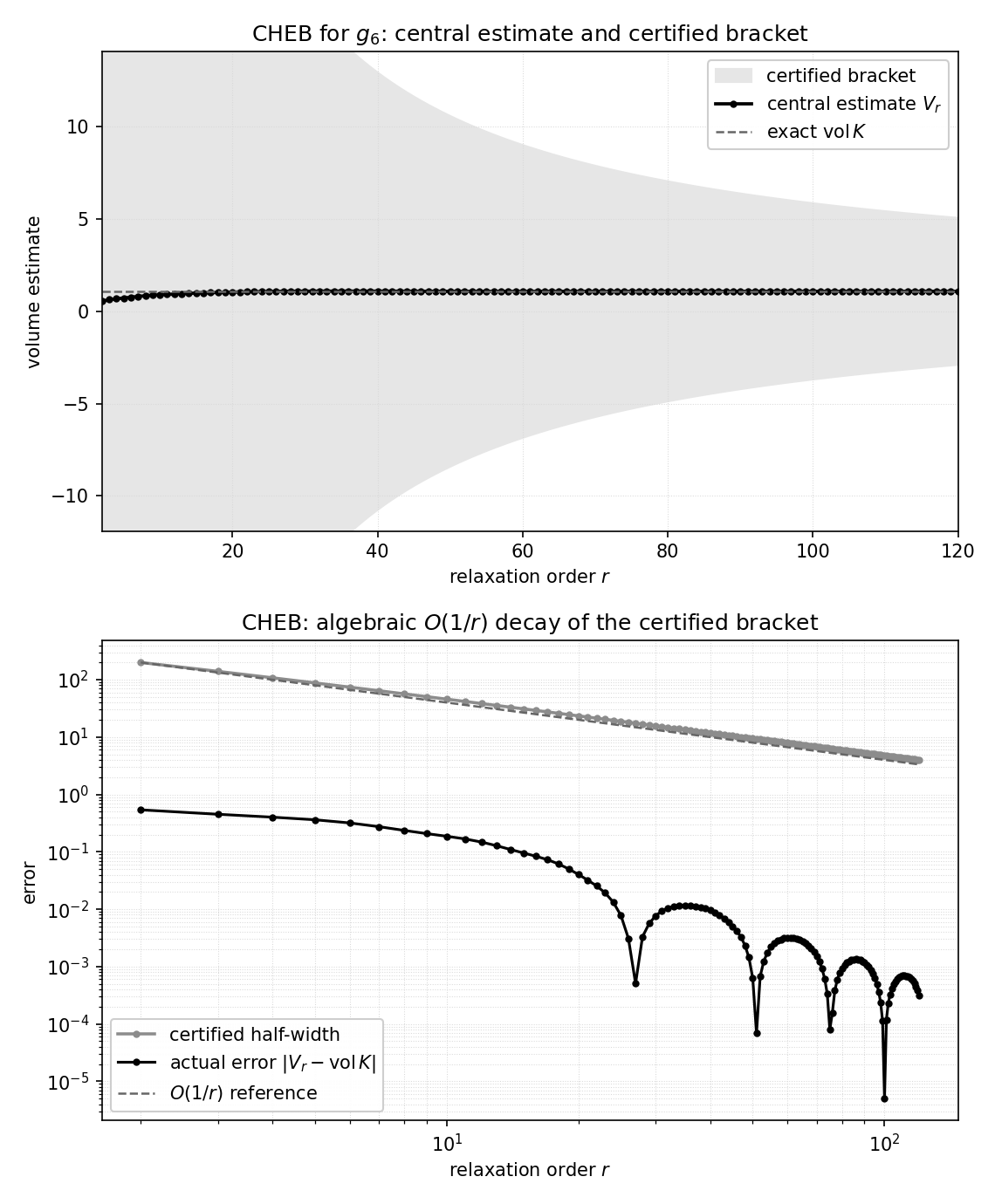}
		\caption{\textsc{Cheb} volume approximation, $g_4$ (left) and $g_6$ (right). Top: central
			estimate and certified bracket against the exact volume (dashed). Bottom
			(log--log): the certified half-width follows the $O(1/r)$ law (slope $-1$), while the
			actual error decays faster. The $g_6$ bracket is wider by $\approx4.1$ but obeys the same
			rate.}
		\label{fig:nc_cheb}
	\end{figure}
	
	\begin{figure}[ht]
		\centering
		\includegraphics[width=0.49\linewidth]{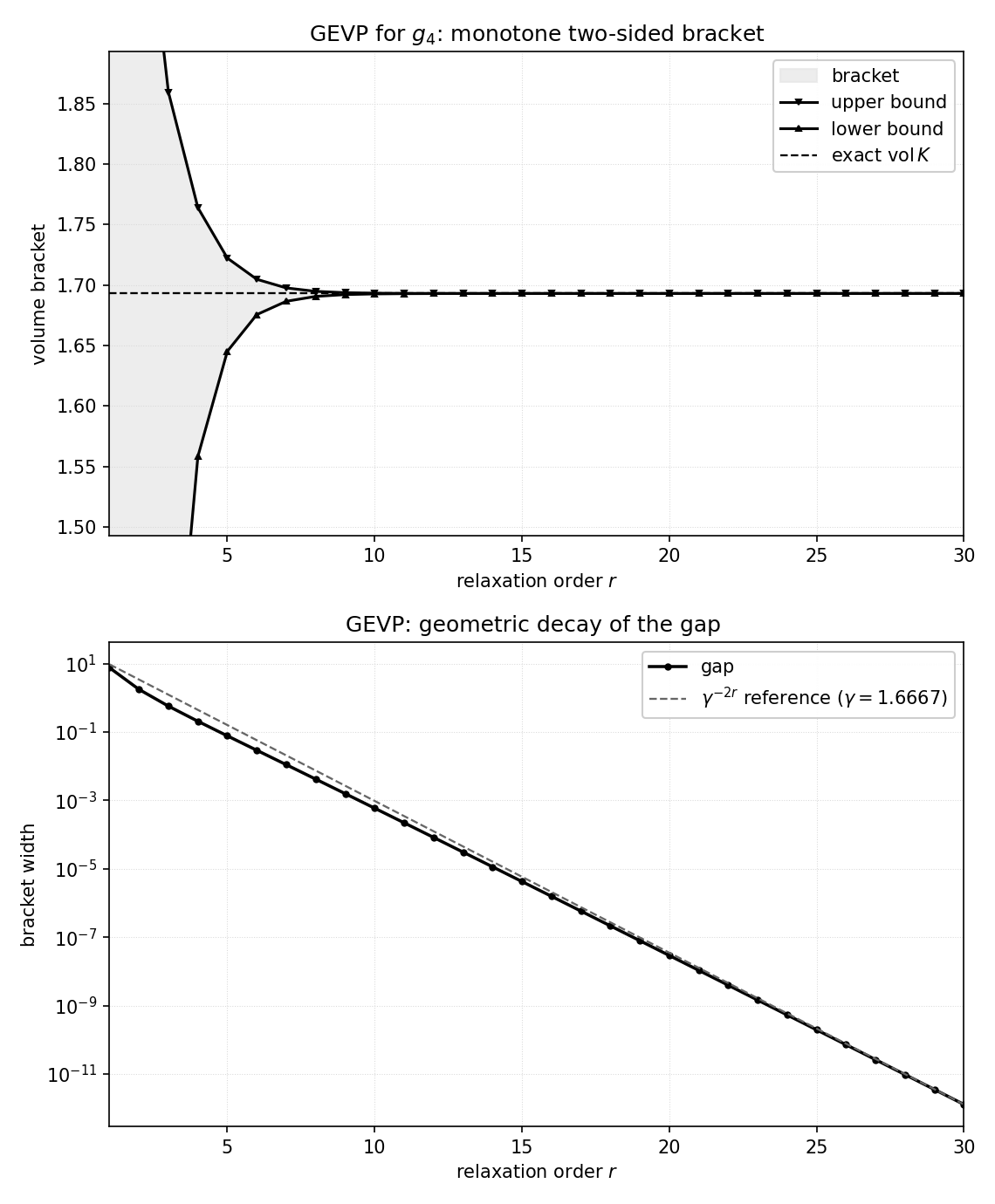}\hfill
		\includegraphics[width=0.49\linewidth]{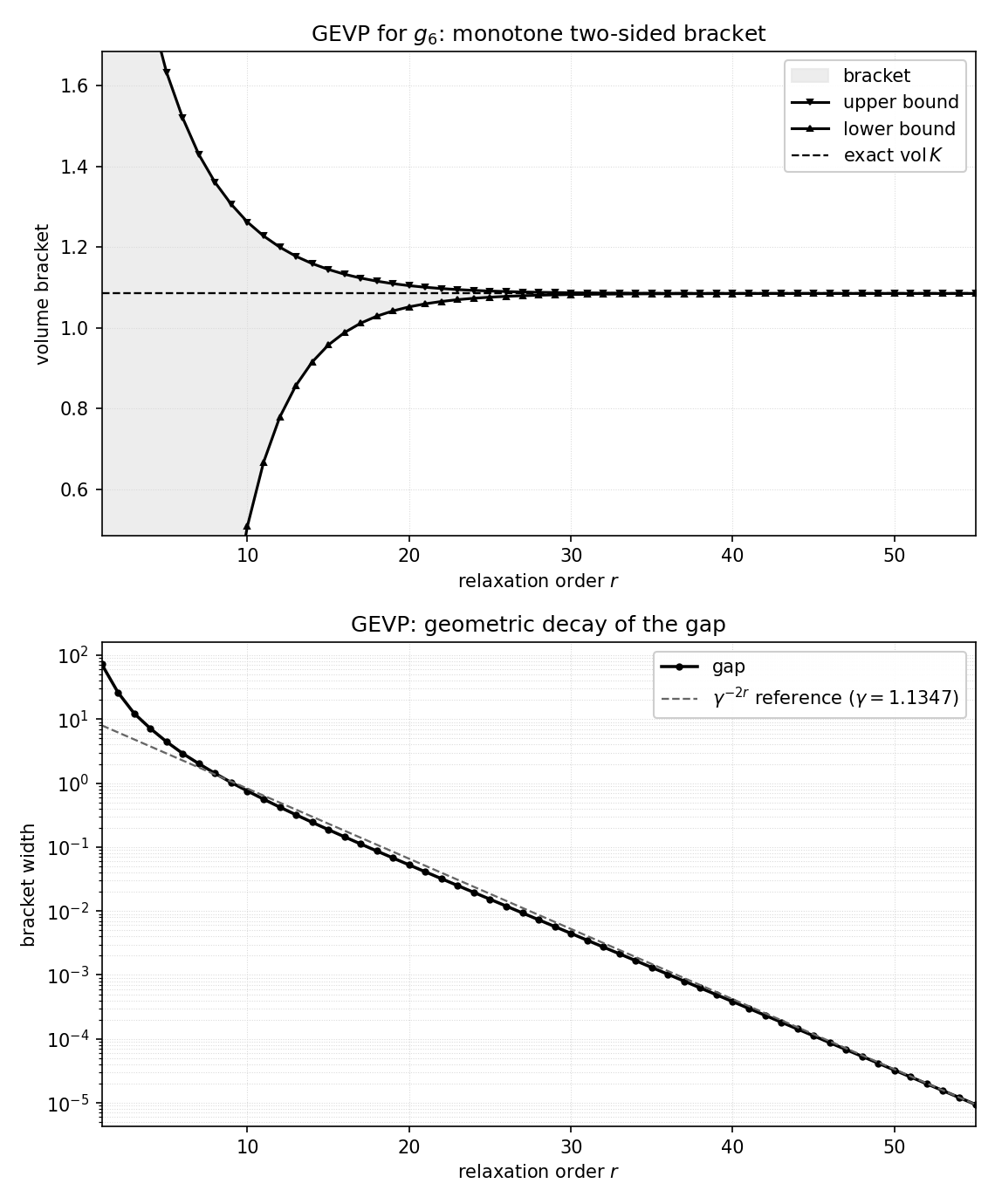}
		\caption{\textsc{Gevp} volume approximation, $g_4$ (left) and $g_6$ (right). Top: the
			monotone bracket $[v_r^-,v_r^+]$ squeezing onto the exact volume. Bottom (semi-log): the
			gap closes geometrically, the dashed reference having the predicted slope
			$\log_{10}\gamma^{-2}$ ($0.360$ for $g_4$, $0.777$ for $g_6$); the slow $g_6$ rate is the
			signature of the large $\bg_6=251.2$.}
		\label{fig:nc_gevp}
	\end{figure}
	
	\begin{table}[ht]
		\centering
		\renewcommand{\arraystretch}{1.2}
		\begin{tabular}{r cc cc}
			\toprule
			& \multicolumn{2}{c}{$g_4$\ \ ($\bg=16$, $\gamma=5/3$)}
			& \multicolumn{2}{c}{$g_6$\ \ ($\bg=251.2$, $\gamma=1.135$)}\\
			\cmidrule(lr){2-3}\cmidrule(lr){4-5}
			$r$ & \textsc{Cheb} width & \textsc{Gevp} width
			& \textsc{Cheb} width & \textsc{Gevp} width\\
			\midrule
			$5$  & $1.7\cdot10^{1}$ & $7.7\cdot10^{-2}$ & $1.8\cdot10^{2}$ & $4.4\cdot10^{0}$\\
			$10$ & $8.7\cdot10^{0}$ & $5.9\cdot10^{-4}$ & $9.2\cdot10^{1}$ & $7.5\cdot10^{-1}$\\
			$20$ & $4.4\cdot10^{0}$ & $2.9\cdot10^{-8}$ & $4.7\cdot10^{1}$ & $5.3\cdot10^{-2}$\\
			$30$ & $3.0\cdot10^{0}$ & $1.3\cdot10^{-12}$& $3.2\cdot10^{1}$ & $4.5\cdot10^{-3}$\\
			\bottomrule
		\end{tabular}
		\caption{Certified bracket \emph{widths} (upper minus lower bound) at matched orders.
			\textsc{Cheb} widths are $2\cdot\tfrac{2^n(\ws+m)}{m}\varepsilon_r$ (double precision);
			\textsc{Gevp} widths are $v_r^+-v_r^-$ (multiprecision). The geometric \textsc{Gevp}
			bracket overtakes the algebraic \textsc{Cheb} one by many orders of magnitude, the gain
			shrinking but persisting as $\bg$ grows from $16$ to $251.2$.}
		\label{tab:nc_compare}
	\end{table}
	
	\paragraph{Computational cost.}
	The two columns of Table~\ref{tab:nc_compare} also differ by orders of magnitude in cost. \textsc{Cheb} runs entirely in
	hardware double precision, its only nontrivial work being the tensor Gauss-Legendre
	cubature for the Chebyshev moments ($Q=241$, resp.\ $361$, nodes per axis at $r=120$): no
	displayed order costs more than $0.15$\,s, and the complete sweeps $r\le120$, moments
	included, take $0.10$\,s for $g_4$ and $0.20$\,s for $g_6$. \textsc{Gevp} pays for its
	geometric rate in precision: the Hankel pencils are solved in multiprecision with
	$\lceil2r\log_{10}\bg\rceil+40$ decimal digits, a working precision growing linearly in
	$r$, so the cost per order rises from a few hundredths of a second at $r=10$ to
	$0.85$\,s at $r=30$ ($113$ digits, $g_4$) and $6.7$\,s at $r=55$ ($305$ digits, $g_6$);
	the full sweeps take $7.2$\,s and $79$\,s, on top of $0.4$\,s, resp.\ $1.3$\,s, for the
	exact rational moments. The trade-off nevertheless favours \textsc{Gevp} whenever
	multiprecision is acceptable: at $r=10$ --- five hundredths of a second --- its certified
	width ($5.9\cdot10^{-4}$ for $g_4$, $7.5\cdot10^{-1}$ for $g_6$) is already beyond
	anything \textsc{Cheb} certifies at any order of the sweep, and the terminal widths
	$1.3\cdot10^{-12}$ ($r=30$) and $9.5\cdot10^{-6}$ ($r=55$) of Figure~\ref{fig:nc_gevp}
	are reached at under one, resp.\ seven, seconds per order. The large $\bg_6$ thus taxes
	\textsc{Gevp} a second time: not only through the slower rate $\gamma^{-2r}$ of
	Remark~\ref{rem:nc_rates}, but through the precision --- hence time --- needed per order.
	
	\subsection{Large dimensional balls}\label{sec:lpballs}
	
	To probe the dependence on the dimension we take the $\ell^p$ balls
	\begin{equation}\label{eq:lp_balls}
		K=\{x\in\R^n:\ q_p(x)\le1\},
		\qquad
		q_p(x)=\sum_{i=1}^n x_i^p,
		\qquad p\ \text{even},
	\end{equation}
	on $B=[-1,1]^n$, with
	$\vol K=2^n\,\Gamma(1+1/p)^n/\Gamma(1+n/p)$. The framework applies with
	$g=q_p$, degree $m=p$, weights $w=(1,\dots,1)$, $\ws=n$; the ball touches $\partial B$ at
	the $2n$ axis points, so $\nu|_{[0,1)}$ is exactly the Beta reference
	(Proposition~\ref{prop:nu_restriction}) with $\nu_w=\mathrm{Beta}(n/p,1)$,
	$z_j=n/(n+pj)$, and the residual support is $[1,\bg]$. Two structural facts make the
	family a pure dimension probe. First,
	\[
	\bg=\max_B q_p=n\quad\text{for every }p,
	\]
	attained at the corners: the \textsc{Gevp} rate $\gamma=(\sqrt n+1)/(\sqrt n-1)$ of
	Theorem~\ref{thm:rate} and the \textsc{Cheb} constant $\sqrt{\bg-1}/\pi$ depend on $n$
	alone, so $p$ can affect only constants --- the cleanest instance of
	Remark~\ref{rem:nc_rates}. Second, the moments are one-dimensional:
	$y_k=2^{-n}\int_B q_p^k\,dx$ is a $k$-fold binomial convolution of the univariate moments
	$1/(pj+1)$, computed by binary powering of a truncated exponential generating function in
	$O(\log n\,(2r)^2)$ operations --- in exact rational arithmetic for \textsc{Cheb}, whose
	Chebyshev moments follow by an exact change of basis (the monomial coefficients of
	$T_k(2t/n-1)$ grow like $(3+2\sqrt2)^k$, beyond double precision past $k\approx18$), and
	in multiprecision for \textsc{Gevp} (all terms positive, no cancellation; both pipelines
	agree to $30$ digits). No cubature is needed at any $n$: the curse of dimensionality is
	absent from the moment side, and \emph{$n$ enters the two algorithms only through
		$\bg=n$}.
	
	\subsubsection{Matched orders: the rate is set by $n$ alone}\label{sec:lp_grid}
	
	Table~\ref{tab:lp_grid} runs both algorithms on the grid $p\in\{2,4,6\}$,
	$n\in\{2,3,4\}$ at order $r=30$. The \textsc{Gevp} slope is a function of $n$ only: at
	$n=3$ the bracket width \eqref{eq:rate_gap} of Theorem~\ref{thm:rate} decays by the factor $10^{11.2}$, $10^{11.3}$, $10^{11.4}$ per
	ten orders for $p=2,4,6$ respectively, against the predicted
	$\gamma^{20}=(2+\sqrt 3)^{20} = 10^{20\log(2+\sqrt 3)} = 10^{11.4}$; the exponent $p$ affects only the constant (larger $p$ gives a
	flatter reference $\nu_w$ and a smaller constant). The \textsc{Cheb} brackets follow the
	$\sqrt{n-1}/(\pi r)$ law of eq.~\eqref{eq:relu_rate},  Theorem~\ref{thm:cheb_rate},  with the prefactor $2^n(n+p)/p$, hence relative widths between
	$1.2\cdot10^{-1}$ and $1.5\cdot10^{0}$ at $r=30$ --- already marginal at $n=4$ --- while
	its central estimate remains accurate to a few $10^{-6}$ at $r=60$
	(Figure~\ref{fig:lp_volume}, left): as for the non-convex examples, it is the
	certificate, not the estimate, that is loose. Every displayed \textsc{Cheb} bracket costs
	$3$--$5$\,ms; the \textsc{Gevp} cells cost $0.55$--$0.59$\,s at the working precision of
	$148$ digits used to resolve the deep gaps.
	
	\begin{table}[ht]
		\centering
		\renewcommand{\arraystretch}{1.2}
		\begin{tabular}{cc c cc cc}
			\toprule
			& & & \multicolumn{2}{c}{\textsc{Cheb}} & \multicolumn{2}{c}{\textsc{Gevp}}\\
			\cmidrule(lr){4-5}\cmidrule(lr){6-7}
			$p$ & $n$ & $\vol K$ & rel.\ width & time (s) & rel.\ width & time (s)\\
			\midrule
			$2$ & $2$ & $3.1416$  & $2.1\cdot10^{-1}$ & $0.003$ & $2.7\cdot10^{-45}$ & $0.59$\\
			$2$ & $3$ & $4.1888$  & $5.8\cdot10^{-1}$ & $0.004$ & $7.9\cdot10^{-33}$ & $0.57$\\
			$2$ & $4$ & $4.9348$  & $1.5\cdot10^{0}$  & $0.004$ & $1.9\cdot10^{-26}$ & $0.58$\\
			$4$ & $2$ & $3.7081$  & $1.4\cdot10^{-1}$ & $0.003$ & $1.3\cdot10^{-46}$ & $0.55$\\
			$4$ & $3$ & $6.4820$  & $2.6\cdot10^{-1}$ & $0.005$ & $1.5\cdot10^{-34}$ & $0.57$\\
			$4$ & $4$ & $10.800$  & $4.6\cdot10^{-1}$ & $0.004$ & $1.7\cdot10^{-28}$ & $0.59$\\
			$6$ & $2$ & $3.8552$  & $1.2\cdot10^{-1}$ & $0.003$ & $3.2\cdot10^{-47}$ & $0.56$\\
			$6$ & $3$ & $7.2077$  & $2.0\cdot10^{-1}$ & $0.005$ & $2.6\cdot10^{-35}$ & $0.58$\\
			$6$ & $4$ & $13.129$  & $3.2\cdot10^{-1}$ & $0.005$ & $2.1\cdot10^{-29}$ & $0.58$\\
			\bottomrule
		\end{tabular}
		\caption{$\ell^p$ balls of dimension $n$ at order $r=30$: relative certified widths and
			standalone times per order. The \textsc{Gevp} width depends
			on $n$ essentially alone, the \textsc{Cheb} width carries
			the scale factor $2^n(n+p)/(p\,\vol K)$. \textsc{Cheb} times cover the exact
			moments, the change of basis and the double-precision sum; \textsc{Gevp} times
			the multiprecision moments, pencils and eigensolves ($148$ digits).}
		\label{tab:lp_grid}
	\end{table}
	
	\begin{remark}[Absolute versus relative certificates]\label{rem:lp_scale}
		The \textsc{Cheb} certificate $|\vol K-V_r|\le\tfrac{2^n(\ws+m)}{m}\varepsilon_r$ of Theorem~\ref{thm:cheb_conv} and Proposition~\ref{prop:cheb_bracket}
		pays the uniform error $\|f-f_r\|_\infty$ against the \emph{unit mass} of $\nu$,
		while the signal $\int(1-t)_+\,d\nu=\tfrac{m}{2^n(\ws+m)}\vol K$ is exponentially
		small in $n$: the relative width is
		$\sim2^n(n+p)\sqrt{n-1}/(p\pi r\,\vol K)$, with the rate $1/r$ unimprovable at the
		kink (Remark~\ref{rem:bernstein}). The defect is invariant under renormalising $\mu$
		--- the product of prefactor and total mass is --- and expresses the
		measure-blindness of a sup-norm certificate: it cannot use the fact that the
		integrand is supported on $[0,1]$, where $\nu$ carries only the mass
		$2^{-n}\vol K$, while the bulk of $\nu$ sits near $t=n/3$. \textsc{Gevp} repairs
		this structurally: by~\eqref{eq:vr_pm}, $v_r^+/2^n$ minimises
		$\int q^2\,d\nu/\int q^2\,d\nu_w$ over $\deg q\le r$, i.e.\ it optimises the test
		polynomial \emph{against $\nu$ itself}; the scale mismatch survives only in the
		constant, which the geometric rate converts into an additive
		$\log C/(2\log\gamma)$ extra orders --- the transients of
		Figure~\ref{fig:lp_volume} (right) and the column $r$ of Table~\ref{tab:lp_budget}.
	\end{remark}
	
	\subsubsection{Scaling in the dimension under a one-minute budget}\label{sec:lp_budget}
	
	For $p=2$ we let $n$ grow along the ladder of Table~\ref{tab:lp_budget}, running
	\textsc{Gevp} over the orders $r=1,\dots,4$ and then $r=5,10,15,\dots$, each order
	computed independently at the working precision
	\[
	\mathrm{dps}(r)
	=\bigl\lceil\max\bigl(2r\log_{10}n,\ 1.55\,r\bigr)\bigr\rceil
	+\bigl\lceil-\log_{10}\alpha\bigr\rceil+50,
	\qquad
	\alpha=2^{-n}\vol K,
	\]
	where $\alpha$ is the scale of the generalized eigenvalue itself
	($\alpha\sim10^{-15}$ at $n=32$, $10^{-39}$ at $n=64$) and the term $1.55\,r$ accounts
	for the conditioning of the reference Hankel moment matrix $M_r(\nu_w)$, dominant for
	small $n$; each $n$ is stopped at relative width $10^{-6}$ or at a $60$\,s wall budget.
	Every computed bracket contained the exact volume. The tolerance is reached
	through $n=32$: the volume $4.30\cdot10^{-6}$ of the Euclidean ball in $\R^{32}$ is
	certified to relative width $6.1\cdot10^{-7}$ in $33$\,s, through Hankel pencils of size
	at most $71$. Beyond, the budget binds: $n=40$ still certifies $7.7\cdot10^{-5}$,
	informative certificates persist through $n=48$ (relative width $8\cdot10^{-2}$), and
	collapse at $n=64$. The cost per order is dominated by the eigensolves and grows with
	the precision ladder --- $8.5$\,s at $(r,\mathrm{dps})=(70,276)$ for $n=32$, $12.4$\,s at
	$(75,360)$ for $n=64$ --- so the frontier is set jointly by the rate
	$\gamma(n)^{-2r}$, the constant, and the precision price of each order. The rate fan of
	Figure~\ref{fig:lp_volume} (right), extended beyond the budget until every curve
	certifies $10^{-9}$, in fact exhibits the bound of Theorem~\ref{thm:rate} \emph{with its
		prefactor}: under the full law $r^{\ws/m}\gamma^{-2r}$ the decay factor per order is
	$\gamma^{2}(1+1/r)^{-n/p}$, i.e.\ the local slope is $-2\log_{10}\gamma+(n/p)/(r\ln10)$,
	in quantitative agreement with the data --- at $n=32$ the observed per-order factor near
	$r=70$ is $1.65$, against $1.61$ predicted and the asymptotic $\gamma^{2}=2.04$. The
	dashed references of the figure carry this full law, anchored at the final order and
	drawn over its decaying branch $r\ge n/(4\ln\gamma)$. The pure rate $\gamma^{-2r}$ is
	only the $r\to\infty$ limit of the slope, approached like $1/r$: matching it to within
	$10\%$ requires $r\gtrsim1.25\,n^{3/2}$ for $p=2$, i.e.\ $r\approx225$ at $n=32$, beyond
	any budget --- for large $n$ the prefactor $r^{\ws/m}$ is not a technicality but the
	visible regime. What the references do not capture is the still faster initial collapse
	from the starting scale, of the order of $\vol B/\vol K=2^n/\vol K$: this is the
	constant of Theorem~\ref{thm:rate} being paid.
	
	\begin{table}[ht]
		\centering
		\renewcommand{\arraystretch}{1.2}
		\begin{tabular}{r c rrrr}
			\toprule
			$n$ & $\vol K$ & $r$ & dps & rel.\ width & time (s)\\
			\midrule
			$2$  & $3.14$               & $5$  & $59$  & $1.1\cdot10^{-7}$ & $0.04$\\
			$4$  & $4.93$               & $10$ & $67$  & $3.6\cdot10^{-8}$ & $0.10$\\
			$8$  & $4.06$               & $20$ & $89$  & $4.2\cdot10^{-9}$ & $0.41$\\
			$16$ & $2.35\cdot10^{-1}$   & $35$ & $141$ & $6.5\cdot10^{-8}$ & $2.3$\\
			$24$ & $1.93\cdot10^{-3}$   & $50$ & $199$ & $7.1\cdot10^{-7}$ & $8.2$\\
			$32$ & $4.30\cdot10^{-6}$   & $70$ & $276$ & $6.1\cdot10^{-7}$ & $32.6$\\
			$40$ & $3.60\cdot10^{-9}$   & $80$ & $328$ & $7.7\cdot10^{-5}$ & $57.6$\\
			$48$ & $1.38\cdot10^{-12}$  & $80$ & $346$ & $8.1\cdot10^{-2}$ & $59.7$\\
			$64$ & $3.08\cdot10^{-20}$  & $75$ & $360$ & $2.3\cdot10^{4}$  & $50.8$\\
			\bottomrule
		\end{tabular}
		\caption{\textsc{Gevp} for the Euclidean balls under a $60$\,s budget per
			dimension: last order reached, its working precision, the certified relative
			width and the total time. The target $10^{-6}$ is met through $n=32$;
			informative certificates (relative width $<1$) persist through $n=48$.}
		\label{tab:lp_budget}
	\end{table}
	
	\begin{figure}[ht]
		\centering
		\includegraphics[width=0.49\linewidth]{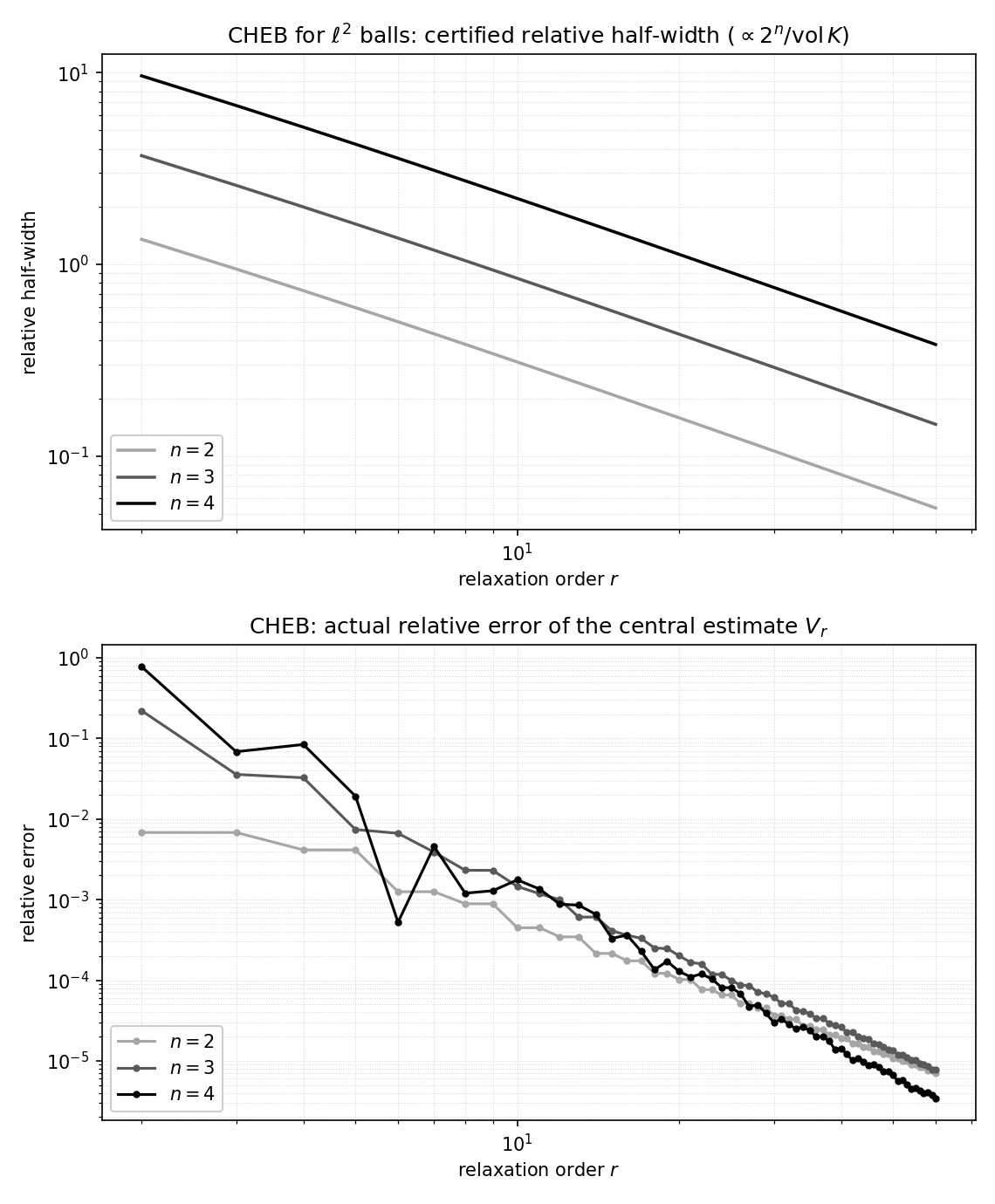}\hfill
		\includegraphics[width=0.49\linewidth]{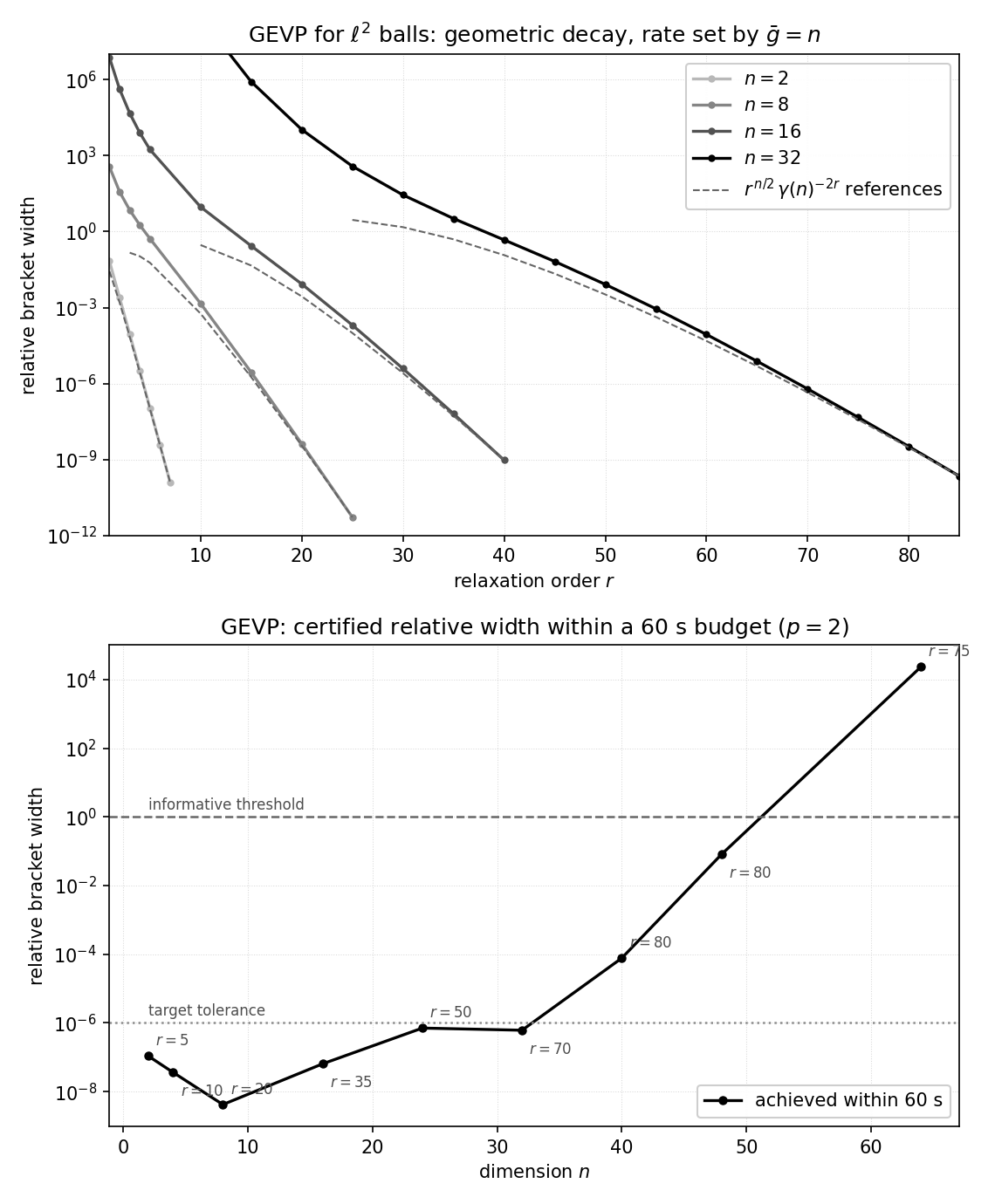}
		\caption{$\ell^2$ balls. \emph{Left} (\textsc{Cheb}): certified relative half-width
			for $n=2,3,4$ (top, the $O(1/r)$ fan with offsets $2^n(n+p)\sqrt{n-1}/(p\pi\,
			\vol K)$) and actual relative error of the central estimate $V_r$ (bottom).
			\emph{Right} (\textsc{Gevp}): relative bracket width against the order for
			$n\in\{2,8,16,32\}$, extended beyond the budget until every curve certifies
			$10^{-9}$, with the $r^{n/2}\gamma(n)^{-2r}$ references --- the full law of
			Theorem~\ref{thm:rate} --- anchored at the final order and drawn over their
			decaying branch (top), and the certified relative width achieved within the
			$60$\,s budget against the dimension, with the informativeness (width $1$) and
			target ($10^{-6}$) thresholds (bottom).}
		\label{fig:lp_volume}
	\end{figure}
	
	\section{Conclusion and outlook}\label{sec:conclusion}
	
	We have shown that a deterministic algorithm for computing the volume of a sublevel set of a nonnegative quasi-homogeneous polynomial need not involve any semidefinite programming:
	the problem can be solved, to arbitrary accuracy, by a purely one-dimensional
	procedure. The construction pushes the Lebesgue measure of the ambient box
	forward through the polynomial, producing a univariate measure on the line
	whose cumulative distribution, evaluated at the prescribed threshold, is
	exactly the volume in question. Because the polynomial is quasi-homogeneous, a
	weighted Stokes identity fixes every moment of this pushforward against an
	explicit Beta reference measure, leaving a single undetermined scalar -- the
	volume itself. The original multivariate volume problem is thereby reduced to a
	truncated one-dimensional moment problem whose data are available in closed
	form.
	
	On this reduced problem we proposed two complementary algorithms, each
	returning a pair of rigorous lower and upper bounds that bracket the volume.
	\textsc{Cheb} assembles the bounds from closed-form Chebyshev coefficients and
	converges at an algebraic rate in the truncation order. \textsc{Gevp} extracts
	them from a single generalized eigenvalue problem and converges geometrically;
	its only overhead is a working precision growing linearly with the order, which
	arbitrary-precision arithmetic absorbs at negligible cost.
	
	Two extensions seem to us especially promising. The first replaces the single
	constraint by several quasi-homogeneous inequalities $g_j \leq 1$, each with its
	own weight vector. This is within reach because the two ingredients that drive
	\textsc{Cheb} and \textsc{Gevp} -- closed-form moments delivered by the weighted
	Stokes identity, and a reference measure pinning those moments down up to the
	volume -- both survive: pushing the Lebesgue measure forward through the map
	$(g_1,\ldots,g_p)$ yields a joint measure whose mixed moments remain
	Stokes-accessible, since each $g_j$ contributes its own scaling, while every
	one-dimensional marginal still carries a Beta-type reference. The volume then
	becomes a truncated moment problem on the unit box that can be approached
	through these joint univariate marginals. Only the bookkeeping changes --
	several weight vectors and a multivariate rather than scalar moment problem --
	not the underlying mechanism.
	
	The second extension abandons quasi-homogeneity entirely. The univariate
	pushforward itself needs no structural assumption on $g$: one can always form
	the image measure $\nu$ and read the volume off its distribution function, and
	the Stokes-based acceleration of moment methods for volume computation continues
	to apply, exactly as developed in \cite{stokesgibbs}. What is lost is the favourable convergence
	rate, and the reason is instructive. By the coarea formula the density of $\nu$
	at a value $t$ is the integral of the reciprocal of the norm of the gradient of $g$ over the level set
	$\{g = t\}$; it therefore blows up wherever that level set meets a critical
	point of $g$, that is, at the \emph{critical values} of $g$, producing
	singularities of $\nu$ that are typically algebraic or logarithmic. In the
	quasi-homogeneous case the only critical value is the origin, sitting harmlessly
	at the endpoint of the interval, so the density reduces to a single power
	function -- a monomial in $t$ -- that is analytic on the interior; this
	analyticity is precisely what makes the Chebyshev and eigenvalue approximations
	converge so fast. A general $g$, by contrast, has critical values scattered
	through the interior of $[0, \bg]$, and approximating a density with such
	interior singularities is intrinsically slower than approximating an analytic
	one.  The bounds produced by both algorithms remain valid, and the degradation admits a
	quantitative form. The proofs of
	Propositions~\ref{prop:rate_upper}--\ref{prop:rate_lower} use only that $\hat\nu$ is
	supported in $[1,\bg]$ and that $\nu_w$ charges a neighbourhood of a point at
	positive distance from that support. Suppose $g$ has no critical value in
	$(1,c)$ for some $c \le \bg$. Then $\nu$ is analytic on $(1,c)$, the pair of
	intervals $[0,1]$ and $[1,c]$ retains a strictly positive conformal separation, and
	the same Bernstein-Walsh argument applied to $[1,c]$ in place of $[1,\bg]$ yields the
	geometric rate $\gamma_c^{-2r}$ with
	$\gamma_c = (\sqrt c + 1)/(\sqrt c - 1)$, at the price of a constant depending on
	$\nu\bigl((c,\bg]\bigr)$. The rate of the general problem should therefore be
	governed by the first critical value of $g$ above the threshold --- the same
	organizing quantity as in the Picard-Fuchs approach of \cite{slm2019} --- rather
	than by the absence of quasi-homogeneity as such. Making this precise, and
	estimating that critical value effectively, seems to us the natural continuation.

	\section*{Acknowledgments}
	
	The authors are grateful to Matteo Tacchi-B\'enard for useful exchanges.
	The authors acknowledge the use of AI (ChatGPT 5.6 and Claude Fable 5) for assistance with brainstorming ideas, mathematical development, coding and drafting the manuscript. The final content, analysis and conclusions remain the sole responsibility of the authors.

\end{document}